\documentclass[10pt]{amsart}
\usepackage[letterpaper, left=3cm, right=3cm, top = 2.8cm, bottom = 2.8cm ]{geometry}

\usepackage{amscd,verbatim}
\usepackage{amssymb}
\usepackage[all]{xy}
\usepackage{color}
\definecolor{winered}{rgb}{0.8,0,0}
\usepackage[pdfauthor={Federico Binda and Shuji Saito}, %
				pdftitle={Relative Cycles with Moduli and Regulator Maps}%
			dvips,colorlinks=true]{hyperref}
\hypersetup{
	bookmarksnumbered=true,
	linkcolor=winered,
	citecolor=blue,
	pagecolor=black, 
	urlcolor=black,	
}

\numberwithin{equation}{section}

\usepackage{color}

\theoremstyle{plain}
\newtheorem{theo}{Theorem}[section]
\newtheorem{lem}[theo]{Lemma}

\newtheorem{prop}[theo]{Proposition}

\newtheorem{coro}[theo]{Corollary}

\theoremstyle{definition}
\newtheorem{defi}[theo]{Definition}

\theoremstyle{remark}
\newtheorem{rem}[theo]{Remark}

\newtheorem{exam}{Example}
\newtheorem{claim}[theo]{Claim}

\newcommand{\Z}{\mathbb{Z}}
\renewcommand{\P}{\mathbb{P}}

\newcommand{\CH}{\mathrm{CH}}
\newcommand{\trdeg}{\mathrm{trdeg}}

\newcommand{\fil}{\mathrm{fil}}
\newcommand{\Ker}{\mathrm{Ker}}
\newcommand{\degn}{\mathrm{degn}}
\newcommand{\Tot}{\mathrm{Tot}}

\newcommand{\bA}{\mathbb{A}}

\newcommand{\cM}{\mathcal{M}}
\newcommand{\bH}{\mathbb{H}}

\newcounter{spec}
{\end{list}}%

\def\gr{{\mathrm{gr}}}     
\def\eff{{\mathrm{eff}}}     
\def\codim{{\mathrm{codim}}}% codimension
\def\CH{{\mathrm{CH}}}      % Chow group
\def\Coker{{\mathrm{Coker}}}% cokernel
\def\Ker{{\mathrm{Ker}}}    % kernel
\def\div{{\mathrm{div}}}    % divisible elements, divisor
\def\dlog{d{\mathrm{log}}}  % log. differential form
\def\et{{\text{\'{e}t}}}    % etale
\def\an{{\text{an}}}    % etale
\def\zar{{\text{zar}}}    % etale

\def\Ext{{\mathrm{Ext}}}    % extension group
\def\H{{\mathrm{H}}}        % (co)homology group
\def\Image{{\mathrm{Image}}}   % image
\def\ker{{\mathrm{Ker}}}    % kernel
\def\mod{\;{\mathrm{mod}}\;}    % modulo
\def\Pic{{\mathrm{Pic}}}    % Picard group
\def\Proj{{\mathrm{Proj}}}  % Proj. functor
\def\Spec{{\mathrm{Spec}}}

\def\ch{{\mathrm{ch}}}    % symbol
\def\trdeg{{\mathrm{trdeg}}}    % symbol
\def\Frac{{\mathrm{Frac}}}    % symbol
\def\Cone{{\mathrm{Cone}}}    % symbol
\def\Res{{\mathrm{Res}}}    % symbol
\def\Tr{{\mathrm{Tr}}}    % symbol
\def\cF{{\mathcal F}}

\def\cI{{\mathcal I}}
\def\cL{{\mathcal L}}

\def\cO{{\mathcal O}}

\def\cC{{\mathcal C}}
\def\cZ{{\mathcal Z}}
\def\cB{{\mathcal B}}
\def\cD{{\mathcal D}}

\def\cS{{\mathcal S}}

\def\cK{{\mathcal K}}

\def\Lam{\Lambda}
\def\lam{\lambda}

\def\ep{\epsilon}
\def\fm{{\frak{m}}}          % Milnor
\def\ol#1{\overline{#1}}
\def\ul#1{\underline{#1}}

\def\bN{{\mathbb N}}
\def\bC{{\mathbb C}}
\def\bR{{\mathbb R}}
\def\bZ{{\mathbb Z}}

\def\bG{{\mathbb G}}

\def\bP{{\mathbb P}}
\def\bF{{\mathbb F}}
\def\bH{{\mathbb H}}
\def\bA{{\mathbb A}}

\def\pz{\bZ/p\bZ}

\def\indlim#1{ \underset{#1}{\underset{\longrightarrow}{\mathrm{lim}}}\; }

\def\rmapo#1{\overset{#1}{\longrightarrow}}

\def\isom{\overset{\cong}{\longrightarrow}}

\def\qaq{\quad\text{and}\quad}
\def\qfor{\quad\text{for }}
\def\qwith{\quad\text{with }}

\def\Ub{\ol U}

\def\Xb{\ol X}
\def\Yb{\ol Y}
\def\Cb{\ol C}

\def\yb{\ol y}
\def\pib{\ol \pi}

\def\ZU#1#2#3{Z^{#3}_{#1,#2}(U)}
\def\BU#1#2#3{B^{#3}_{#1,#2}(U)}

\def\ol#1{\overline{#1}}

\begin{document}

\def\Ycd#1{Y^{(#1)}}
\def\Yx{Y_x}
\def\Yxo{Y_x^\circ}
\def\Yy{Y_y}
\def\Yyo{Y_y^\circ}
\def\Dx{D_x}
\def\Dxo{D_x^\circ}
\def\Dy{D_y}

\def\ZrYD{\bZ(r)_{Y|D_Y}}
\def\ZrYDx{\bZ(r)_{\Yx|\Dx}}
\def\ZrYDy{\bZ(r)_{\Yy|\Dy}}

\def\cKrYD{\cK^M_{r,Y|D_Y}}
\def\cKrYDx{\cK^M_{r,\Yx|\Dx}}
\def\cKrYDy{\cK^M_{r,\Yy|\Dy}}

\def\EXD#1#2{E_{\Xb|D}^{#1,#2}}
\def\FXD#1#2{F_{\Xb|D}^{#1,#2}}

\def\UXDr{U^r_{\Xb|D}}
\def\JXDr{J^r_{\Xb|D}}
\def\JXDredr{J^r_{\Xb|\Dred}}
\def\UXDd{U^d_{\Xb|D}}
\def\JXDd{J^d_{\Xb|D}}
\def\JXDredd{J^d_{\Xb|\Dred}}

\def\gammaXD#1#2{\gamma_{\Xb|D}^{#1,#2}}
\def\HXDZr#1{\underline{H}^{#1}(\Xb|D,\bZ(r))}
\def\VXD#1#2{V_{\Xb|D}^{#1,#2}}
\def\HXD#1{H^{#1}_{\Xb|D}}

\def\HMXDr#1{\H^{#1}_{\cM}(\Xb|D,\bZ(r))}
\def\HDeXDr#1{\H_{\cD}^{#1}(\Xb|D,\bZ(r))}

\def\clOmega#1#2{cl_{\Omega}^{#1}(#2)}
\def\clDR#1#2{cl_{DR}^{#1}(#2)}

\def\WW#1#2{\Omega^{#1}_{#2}}
\def\WWk#1#2{\Omega^{#1}_{#2/k}}

\def\Omegacl#1#2{\Omega^{#1}_{#2,cl}}

\def\Zr{\bZ(r)}
\def\ZrX{\bZ(r)_X}
\def\ZdX{\bZ(d)_X}

\def\Xan{X_{\an}}
\def\Xzar{X_{\zar}}
\def\Yan{Y_{\an}}
\def\Gan{G_{\an}}

\def\Xban{\Xb_{\an}}
\def\Xbzar{\Xb_{\zar}}

\def\Dred{D_{red}}
\def\Zred{Z_{red}}

\def\filD{\fil_D}

\def\PA#1{P_{#1}(A|I)}
\def\PAt#1{\tilde{P}_{#1}(A|I)}

\def\PAgr#1{P_{#1}(A|I)^{\gr}}
\def\PAtgr#1{\tilde{P}_{#1}(A|I)^{\gr}}

\def\NPAgr#1{NP_{#1}(A|I)^{\gr}}
\def\PAgrdeg#1{P_{#1}(A|I)^{\gr}_{deg}}

\def\lamub{\underline{\lambda}}
\def\tub{\underline{t}}

\def\fm{\frak{m}}
\def\nuR#1{\nu_R^{#1}}
\def\nuRI#1{\nu_{R,I}^{#1}}
\def\nuRd#1{\nu_{R'}^{#1}}
\def\nuRId#1{\nu_{R',IR'}^{#1}}
\def\bF{{\mathbb F}}
\def\cO{{\mathcal O}}

\def\rmapo#1{\overset{#1}{\longrightarrow}}

\def\Pp#1{(\bP^1)^{#1}}
\def\cub{\square}
\def\cubb#1{\cub^{#1}}
\def\cubbb#1#2{\overline{\cub}^{#1}_{#2}}

\def\Fni{F^n_i}
\def\Fnj{F^n_j}

\def\Fna{F^n_1}
\def\Enie{E^n_{i,\epsilon}}

%def\Fn{\Sigma_n}
%\def\Fnp{\Sigma_{n-1}}
\def\Fn{F_n}
\def\Fnp{F_{n-1}}

\def\inie{\iota^n_{i,\epsilon}}
\def\ini#1{\iota^n_{i,#1}}

\def\WX#1{\Omega^{#1}_{\Xb}}
\def\WXD#1{\Omega^{#1}_{\Xb|D}}
\def\WXDred#1{\Omega^{#1}_{\Xb|\Dred}}
\def\FrWXD{\Omega^{\geq r}_{\Xb|D}}

\def\WXDp#1{\Omega^{#1}_{\Xb|D'}}
\def\ZXD#1{Z^{#1}_{\Xb|D}}
\def\BXD#1{B^{#1}_{\Xb|D}}
\def\nuXD#1{\nu^{#1}_{\Xb|D}}

\def\WY#1{\Omega^{#1}_{Y}}
\def\WYD#1{\Omega^{#1}_{Y|D}}
\def\WYDp#1{\Omega^{#1}_{Y|D'}}
\def\ZYD#1{Z^{#1}_{Y|D}}
\def\BYD#1{B^{#1}_{Y|D}}
\def\nuYD#1{\nu^{#1}_{Y|D}}

\def\WXlogD#1{\Omega^{#1}_{\Xb}(\log D)}
\def\WXlogDD#1{\Omega^{#1}_{\Xb}(\log D)(-D)}
\def\WPlogH#1{\Omega_{\bP}^{#1}(\log H)}
\def\WPlogHD#1{\Omega_{\bP}^{#1}(\log H+D)}
\def\WPXlogH#1{\Omega_{\Xb\times\bP}^{#1}(\log H)}
\def\WPXlogHD#1{\Omega_{\Xb\times\bP}^{#1}(\log H+D)}

\def\WXlogDn#1{\Omega^{#1}_{\Xb_n}(\log D_n)}
\def\WXlogFDn#1{\Omega^{#1}_{\Xb_n}(\log \Fn+D_n)}
\def\WXlogFDDn#1{\Omega^{#1}_{\Xb_n}(\log \Fn+D_n)(-D_n)}

\def\WU#1{\Omega^{#1}_{U}}
\def\WBU#1{\Omega^{#1}_{U}/d}
\def\ZU#1{Z^{#1}_{U}}
\def\BU#1{B^{#1}_{U}}
\def\nuU#1{\nu^{#1}_{U}}

\def\WXFD#1{\Omega^{#1}_{\Xb|D}(\log F)}
\def\WXFpD#1{\Omega^{#1}_{\Xb|pD}(\log F)}
\def\WXFDp#1{\Omega^{#1}_{\Xb|D'}(\log F)}

\def\WYFD#1{\Omega^{#1}_{Y|D}(\log F)}
\def\WYFpD#1{\Omega^{#1}_{Y|pD}(\log F)}
\def\WYFDp#1{\Omega^{#1}_{Y|D'}(\log F')}
\def\WYFDpp#1{\Omega^{#1}_{Y|D''}(\log F')}
\def\WYFDred#1{\Omega^{#1}_{Y|\Dred}(\log F)}
\def\WYFDn#1{\Omega^{#1}_{Y_n|D_n}(\log F_n)}
\def\WYFDlogn#1{\Omega^{#1}_{Y_n}(\log F_n +D_n)}
\def\WYFDDlogn#1{\Omega^{#1}_{Y_n}(\log F_n +D_n)(-D_n)}

\def\FrWYFD{\Omega^{\geq r}_{Y|D}(\log F)}
\def\FrWYFDn{\Omega^{\geq r}_{Y_n|D_n}(\log F_n)}
\def\FrWYFDstar{\Omega^{\geq r}_{Y_\star|D_\star}(\log F_\star)}

\def\WZFD#1{\Omega^{#1}_{Z|D_Z}(\log F_Z)}

\def\WYFD#1{\Omega^{#1}_{Y|D}(\log F)}
\def\WYFDcl#1{\Omega^{#1}_{Y|D}(\log F)_{cl}}

\def\ZXFD#1{Z^{#1}_{\Xb|D}(\log F)}
\def\ZXFpD#1{Z^{#1}_{\Xb|pD}(\log F)}

\def\BXFD#1{B^{#1}_{\Xb|D}(\log F)}
\def\BXFpD#1{B^{#1}_{\Xb|pD}(\log F)}

\def\WBXFD#1{\Omega^{#1}_{\Xb|D}(\log F)/d}
\def\WBXFDp#1{\Omega^{#1}_{\Xb|D'}(\log F)/d}

\def\ZAFI#1{Z^{#1}_{A|I}(\log F)}
\def\WAFI#1{\Omega^{#1}_{A|I}(\log F)}
\def\WBAFI#1{\Omega^{#1}_{A|I}(\log F)/d}

\def\WAFIp#1{\Omega^{#1}_{A|I'}(\log F)}

\def\ZAs#1{Z^{#1}_{A[s^{-1}]}}
\def\WAs#1{\Omega^{#1}_{A[s^{-1}]}}
\def\WBAs#1{\Omega^{#1}_{A[s^{-1}]}/d}

\def\nuXFD#1{\nu^{#1}_{\Xb|D}(\log F)}
\def\nuU#1{\nu^{#1}_{U}}
\def\nuXFDp#1{\nu^{#1}_{\Xb'|D'}(\log F')}
\def\nuXFDpp#1{\nu^{#1}_{\Xb'|D''}(\log F')}
\def\nuYFD#1{\nu^{#1}_{Y|D_Y}(\log F_Y)}
\def\nuYU#1{\nu^{#1}_{Y\cap U}}

\def\nuXFDn#1{\nu^{#1}_{\Xb_n|D_n}(\log \Fn)}
\def\nuUFDn#1{\nu^{#1}_{U_n|D_n}(\log \Fn)}

\def\cub{\square}

\def\Delb#1{\overline{\Delta^{#1}}}
\def\Del#1{\Delta^{#1}}

\def\zXD#1#2{z^{#1}(\Xb|D,#2)}
\def\zXY#1#2{z^{#1}(\Xb|Y,#2)}

\def\zX#1#2{z^{#1}(\Xb,#2)}
\def\zUD#1#2{z^{#1}(U|D,#2)}
\def\zUY#1#2{z^{#1}(U|Y,#2)}

\def\zXDeff#1#2{z^{#1}(\Xb|D,#2)_{\eff}}

\def\CXD#1#2{C^{#1}(\Xb|D,#2)}
\def\CXY#1#2{C^{#1}(\Xb|Y,#2)}
\def\CXDt#1#2{C^{#1}(\widetilde{\Xb}|\widetilde{D},#2)}

\def\CUD#1#2{C^{#1}(U|D_n,#2)}
\def\CX#1{C^{#1}(\Xb)}

\def\zzXD#1#2{\underline{z}^{#1}(\Xb|D,#2)}
\def\zzUD#1#2{\underline{z}^{#1}(U|D,#2)}
\def\zzXY#1#2{\underline{z}^{#1}(\Xb|Y,#2)}
\def\zzUY#1#2{\underline{z}^{#1}(U|Y,#2)}

\def\zzUDpz#1#2{\underline{z}^{#1}(U|D,#2;\pz)}
\def\zzXDnd#1#2{\underline{z}^{#1}(\Xb|D,#2)_0}
\def\zzXDeff#1#2{\underline{z}^{#1}(\Xb|D,#2)_{\eff}}

\def\zUD#1#2{z^{#1}(U|D,#2)}
\def\zUDpz#1#2{z^{#1}(U|D,#2;\pz)}

\def\CHXD#1#2{CH^{#1}(\Xb|D,#2)}
\def\CHX#1#2{CH^{#1}(\Xb,#2)}

\def\ChXD#1{CH^{#1}(\Xb|D)}
\def\CHXDpz#1#2{CH^{#1}(\Xb|D,#2;\pz)}

\def\CHRI#1#2{CH^{#1}(R|I,#2)}
\def\CHRIpz#1#2{CH^{#1}(R|I,#2;\pz)}

\def\zXDd#1#2{z_{#1}(\Xb|D,#2)}

\def\Wb{\overline{W}}
\def\WbN{\overline{W}^N}
\def\WbNo{\overline{W}^{N,o}}

\def\Vb{\overline{V}}

\def\Dinf{D_{\infty}}

\def\rhoXD#1#2{\rho_{\Xb|D}^{#1,#2}}
\def\bZXDr{\bZ(r)_{\Xb|D}}
\def\bZXDret{\bZ(r)_{\Xb|D}^{\et}}
\def\bZXDran{\bZ(r)_{\Xb|D}^{\an}}
\def\bZXDrde{\bZ(r)_{\Xb|D}^{\cD}}
\def\bZXDonede{\bZ(1)_{\Xb|D}^{\cD}}

\def\pzXDr{\pz(r)_{\Xb|D}}
\def\pzXDret{\pz(r)_{\Xb|D}^{\et}}

\def\bZXDd{\bZ(d)_{\Xb|D}}
\def\bZXDdet{\bZ(d)_{\Xb|D}^{\et}}
\def\pzXDd{\pz(d)_{\Xb|D}}
\def\pzXDdet{\pz(d)_{\Xb|D}^{\et}}

\def\bZYFDrde{\bZ(r)_{(Y,F,D)}^{\cD}}
\def\bZYFDnrde{\bZ(r)_{(Y_n,F_n,D_n)}^{\cD}}
\def\bZYFDstarrde{\bZ(r)_{(Y_\star,F_\star,D_\star)}^{\cD}}

\def\CXFDr{C^r(\Xb,F,D)}
\def\CXFDrp{C^r(\Xb',F',D')}
\def\CYFDr{C^r(Y,F,D)}
\def\CYFDrp{C^r(Y',F',D')}
\def\CZFDr{C^r(Z,F_Z,D_Z)}

\def\iWb{i_{\Wb}}

\def\phiW{\phi_{\Wb}}
\def\phiWi{\phi_{\Wb_i}}

\def\phiWp{\phi_{\Wb'}}
\def\phiV{\phi_{\Vb}}
\def\phiVi{\phi_{\Vb_i}}

\def\uhom{{\mathcal{H}om}}

\def\ul#1{\underline{#1}}
\def\clXD#1{cl_{\Xb|D}(#1)}
\def\clYD#1{cl^r_{Y|D}(#1)}
\def\clZD#1{cl^r_{Z|D_Z}(#1)}

\def\clXDn#1{cl_{\Xb_n|D_n}(#1)}
\def\clUD#1{cl_{U|D}(#1)}
\def\clUDn#1{cl_{U_n|D_n}(#1)}

\def\SrnU{\cS^{r,n}_U}
\def\SrstarU{\cS^{r,\star}_U}

\def\pzXDret{\pz(r)^{\et}_{\Xb|D}}

\def\tD{\widetilde{D}}

\def\clOmega#1#2{cl_{\Omega}^{#1}(#2)}
\def\clDR#1#2{cl_{DR}^{#1}(#2)}
\def\clDe#1#2{cl_{\cD}^{#1}(#2)}
\def\clDeb#1#2{\overline{cl}_{\cD}^{#1}(#2)}

\def\clB#1#2{cl_{B}^{#1}(#2)}

\def\regDe#1#2{\phi_{\cD}^{#1,#2}}
\def\regDR#1#2{\phi_{DR}^{#1,#2}}
\def\regB#1#2{\phi_{B}^{#1,#2}}

\def\reggDe{\phi_{\cD}}
\def\reggDR{\phi_{DR}}
\def\reggB{\phi_{B}}

\def\clK#1#2{cl_{\cK}^{#1}(#2)}

\def\WW#1#2{\Omega^{#1}_{#2}}
\def\WWk#1#2{\Omega^{#1}_{#2/k}}

\def\Omegacl#1#2{\Omega^{#1}_{#2,cl}}
\def\cKM#1#2{\cK^M_{#1,#2}}
\def\cKMYFr{\cK^M_{r,Y}(F)}
\def\cKMYDFr{\cK^M_{r,Y|D}(F)}
\def\cKMYF#1{\cK^M_{#1,Y}(F)}
\def\cKMYDF#1{\cK^M_{#1,Y|D}(F)}
\def\cKMYDFrp{\cK^M_{r,Y'|D'}(F')}

\def\cKMXDr{\cK^M_{r,\Xb|D}}
\def\cKMXr{\cK^M_{r,\Xb}}

\def\cKMXDrn{\cK^M_{r,Y_n|D_n}(F_n)}
\def\cKMXDrstar{\cK^M_{r,Y_\star|D_\star}}

\def\Ycd#1{Y^{(#1)}}
\def\Xcd#1{X^{(#1)}}

\def\KM#1#2{K^M_{#1}(#2)}

\def\cOYDF{\cO^\times_{Y|D}(F)}
\def\cOYF{j_*\cO^\times_{X}}

\def\Dm{D_{\fm}}
\def\Dlam{D_{\lam}}
\def\Dnu{D_{\nu}}
\def\Dnum{D_{\nu,\fm}}

\def\Fm{F^{\fm}}
\def\Fmn{F^{\fm+\delta_\nu}}
\def\mlam{m_{\lam}}
\def\mnu{m_{\nu}}

\def\Im{I_{\fm}}
\def\Imp{I_{\fm'}}
\def\Cmn{C_{\fm,\nu}^{-1}}

\def\grmn{gr^{\fm,\nu}}

\def\Wmn#1{\omega_{\fm,\nu}^{#1}}
\def\Zmn#1{\cZ_{\fm,\nu}^{#1}}
\def\Bmn#1{\cB_{\fm,\nu}^{#1}}
\def\ZZmn#1#2{\cZ_{#2,\fm,\nu}^{#1}}
\def\BBmn#1#2{\cB_{#2,\fm,\nu}^{#1}}
\def\dmn#1{d_{\fm,\nu}^{#1}}
\def\WDn#1{\omega_{D_\nu}^{#1}}
\def\Resmn#1{\Res^{#1}_{\fm,\nu}}    % symbol

\def\WmnX#1{\omega_{\fm,\nu,\Xb}^{#1}}
\def\ZmnX#1{\cZ_{\fm,\nu,\Xb}^{#1}}
\def\BmnX#1{\cB_{\fm,\nu,\Xb}^{#1}}
\def\ZZmnX#1#2{\cZ_{#2,\fm,\nu,\Xb}^{#1}}
\def\BBmnX#1#2{\cB_{#2,\fm,\nu,\Xb}^{#1}}

\def\WmnY#1{\omega_{\fm,\nu,Y}^{#1}}
\def\ZmnY#1{\cZ_{\fm,\nu,Y}^{#1}}
\def\BmnY#1{\cB_{\fm,\nu,Y}^{#1}}
\def\ZZmnY#1#2{\cZ_{#2,\fm,\nu,Y}^{#1}}
\def\BBmnY#1#2{\cB_{#2,\fm,\nu,Y}^{#1}}

\def\WmnPY#1{\omega_{\fm,\nu,\bP^N_Y}^{#1}}
\def\ZmnPY#1{\cZ_{\fm,\nu,,\bP^N_Y}^{#1}}
\def\BmnPY#1{\cB_{\fm,\nu,,\bP^N_Y}^{#1}}
\def\ZZmnPY#1#2{\cZ_{#2,\fm,\nu,,\bP^N_Y}^{#1}}
\def\BBmnPY#1#2{\cB_{#2,\fm,\nu,,\bP^N_Y}^{#1}}

\def\PhirXD{\Phi^r(\Xb,D)}

	\renewcommand{\hom}{\operatorname{Hom}}
	\newcommand{\ext}{\operatorname{Ext}}%
	\newcommand{\tensor}{\otimes}
	\newcommand{\cU}{\mathcal{U}}

%%%%%%%%%%%% ADDED Macro

\title{Relative cycles with moduli and regulator maps}
\author{Federico Binda and Shuji Saito}
\address{Federico Binda\\ 
Fakult\"at f\"ur Mathematik,  Universit\"at Duisburg-Essen, Thea-Leymann Strasse 9, 45127 Essen, Germany}
\curraddr{ Fakult\"at f\"ur Mathematik, Universit\"at Regensburg, 
93040 Regensburg, Germany}
\email{federico.binda@ur.de}

\address{Shuji Saito\\ 
Interactive Research Center of Science, 
Graduate School of Science and Engineering,
Tokyo Institute of Technology\\
Ookayama, Meguro\\
Tokyo 152-8551\\
Japan
}
\email{sshuji@msb.biglobe.ne.jp}


\maketitle
\begin{abstract}Let $\Xb$ be a separated scheme of finite type over a field $k$ and $D$ a non-reduced effective Cartier divisor on it. We attach to the pair $(\Xb, D)$ a cycle complex with modulus, those homotopy groups - called higher Chow groups with modulus - generalize additive higher Chow groups of Bloch-Esnault, R\"ulling, Park and Krishna-Levine, and that sheafified on $\Xb_{Zar}$ gives a candidate definition for a relative motivic complex of the pair, that we compute in weight $1$.
	
	When $\Xb$ is smooth over $k$ and $D$ is such that $D_{red}$ is a normal crossing divisor, we construct a fundamental class in the cohomology of relative differentials for a cycle satisfying the modulus condition, refining El-Zein's explicit construction of the fundamental class of a cycle. This is used to define a natural regulator map from the relative motivic complex of $(\Xb, D)$ to the relative de Rham complex. When $\Xb$ is defined over $\bC$, the same method leads to the construction of a regulator map to a relative version of Deligne cohomology, generalizing Bloch's regulator from higher Chow groups.
	
	Finally, when $\Xb$ is moreover connected and proper over $\bC$, we use relative Deligne cohomology to define relative intermediate Jacobians with modulus $J^r_{\Xb|D}$ of the pair $(\Xb, D)$. For $r=\dim \Xb$, we show that $J^{r}_{\Xb|D}$ is the universal regular quotient of the Chow group of $0$-cycles with modulus.
	\end{abstract}
{\hypersetup{ linkcolor=black}
\tableofcontents
}
\section{Introduction}

\subsection{}A quest for a geometrically defined cohomology theory for an algebraic variety, playing in algebraic geometry the role of ordinary cohomology of a topological space, dates back to the work of A.~Grothendieck and early days of algebraic geometry. In \cite{Be1}, A.~Beilinson gave a precise conjectural framework for such hoped-for theory, foreseeing the existence of an Atiyah-Hirzebruch type spectral sequence for any scheme $S$ %(arbitrary singular)
\begin{equation}\label{eq.AH}E^{p,q}_{2}=\H^{p-q}_{\mathcal{M}}(S, \bZ(-q)) \Rightarrow K_{-p-q}(S)\end{equation}
converging to $K_\bullet(S)$, Quillen's algebraic $K$-theory of $S$. Narrowing the context a little, fix a perfect field $k$ and consider the category $\mathbf{Sch}_k$ of separated schemes of finite type over $k$. When $X$ is smooth and quasi-projective, S.~Bloch's apparently na\"ive definition of higher Chow groups, given in terms of algebraic cycles, provides the right answer, as established in \cite{FrSus} and \cite{LCon}. In larger generality, motivic cohomology groups have been defined by V.~Voevodsky \cite{V} and M.~Levine \cite{LMixed} as Zariski hypercohomology of certain complexes of sheaves, and they are known to agree with Bloch's definition in the smooth case \cite{VChow}. So, if $X$ is any scheme of finite type over $k$, we are now able to consider the motivic cohomology groups
\[X\mapsto \H^*_{\mathcal{M}}(X, \bZ(*))=\H^{*,*}_{\mathcal{M}}(X, \bZ)\]
having a number of good properties, including the existence of the spectral sequence \eqref{eq.AH} for smooth $X$.

While the smooth case is thus established, the conjecture in the general form proposed by Beilinson is still  open. As a motivating example consider, for a smooth variety $X$, the $K$-theory of its $m$-th thickening $X_m$, $K_\bullet(X\times_k \Spec{k[t]/t^m})$. These groups behave very differently from the corresponding motivic cohomology groups, since while it is known that
 \[K_\bullet(X\times_k \Spec{k[t]/t^m}) = K_\bullet(X_m) \neq K_\bullet(X),\]
   according to the current definitions one has
\[{\rm H}^*_{\mathcal{M}}(X, \mathbb{Z}(*)) = {\rm H}^*_{\mathcal{M}}(X\times_k \Spec{k[t]/t^m}, \mathbb{Z}(*)),\]
and this quite obviously prevents the existence of the desired spectral sequence. The insensibility of motivic cohomology to nilpotent thickening is manifesting the fact that, in Voevodsky's triangualated category $\mathbf{DM}(k, \bZ)$, one has $M(X) = M(X_m)$. From this point of view, the available definitions are not completely satisfactory, as they fail to encompass this kind of non-homotopy invariant phenomena.

\subsection{}Without an appropriate categorical framework, such as the one provided by $\mathbf{DM}(k,\bZ)$, the quest starts again from algebraic cycles. The first attempt was made by S.~Bloch and H.~Esnault, that in \cite{BE} introduced additive higher Chow groups of $0$-cycles over a field in order to describe the $K$-theory of the ring $k[t]/(t^2)$ and gave the first evidence in this direction, showing that these groups are isomorphic to the absolute differentials $\Omega^n_k$ agreeing  with  Hasselholt-Madsen description of the $K$-groups of a truncated polynomial algebra. Their work was refined in \cite{BE2} and extended by K.R\"ulling to higher modulus in \cite{Rul}, where the additive higher Chow groups of $0$-cycles were actually shown to be isomorphic to the generalized deRham-Witt complex of Hesselholt-Madsen.

The generalization to schemes was firstly given by J.~Park in \cite{Pa}, that defined additive higher Chow groups for any variety $X$. Park's groups were then further studied by A.~Krishna and M.~Levine in \cite{KL}, that proved a number of structural properties for smooth projective varieties, such as a projective bundle formula, a blow-up formula and some basic functorialities.

\subsubsection{}Additive higher Chow groups are a modified version of Bloch's higher Chow groups, defined by imposing some extra condition, commonly called ``Modulus Condition'', on admissible cycles (i.e.~cycles in good position with respect to certain faces) and are conjectured to describe the relative $K$-groups $K^{nil}_{\bullet}(X, m)$, where $K^{nil}(X,m)$ denotes the homotopy fiber
	\[K(X\times_k \mathbb{A}^1_{k})\to K(X\times_k k[t]/{t^m}).\]
From this point of view, additive higher Chow groups are a candidate definition for the relative motivic cohomology of the pair
	\[(X\times_k \mathbb{A}^1_k, X\times_k k[t]/{t^m} = X_m).\]	
One of the goals of this paper is to generalize this construction, defining for every pair $(\Xb, D)$ consisting of a scheme $\Xb$ (separated and of finite type over $k$) together with a (non reduced) effective Cartier divisor $D\hookrightarrow \Xb$, cubical abelian groups
\[z^{r}(\Xb|D, \bullet) \subset z^{r}(\Xb, \bullet), \quad (\text{Bloch's cubical cycle complex})\]
those $n$-th homotopy groups will be called \textit{higher Chow groups of $\Xb$ with modulus $D$}
\begin{equation}\label{eq.ChowgrMod}\CH^r(\Xb|D, n) = \pi_n(z^{r}(\Xb|D, \bullet)) = \H_n(z^{r}(\Xb|D, *)).\end{equation}
	These groups are controvariantly functorial for flat maps of pairs and covariantly functorial for proper maps of pairs.
Sheafifying this construction on $\Xb_{Zar}$ we obtain, for every $r\geq 0$, complexes of sheaves
\[\bZ_{\Xb|D}(r)\to \bZ_{\Xb}(r)  \]
called \textit{relative motivic complexes}, naturally mapping to $\bZ_{\Xb}(r)$, the complexes of sheaves computing Bloch's higher Chow groups $\CH^r(\Xb; n)$. We call the hypercohomology groups of $\bZ_{\Xb|D}(r)$ the \textit{motivic cohomology groups of the pair $(\Xb,D)$},
\begin{equation}\label{eq.mot-coh-intro}\H^*_{\mathcal{M}}(\Xb|D, \bZ(r)) = \mathbb{H}^{*}(\Xb_{Zar}, \bZ_{\Xb|D}(r)).\end{equation}
There are some new results which provide a ground for the optimistic choice of the words. W.~Kai \cite{KaiMoving} established a moving lemma for cycle complexes with modulus which implies an appropriate contravariant functoriality of the Nisnevich version of \eqref{eq.mot-coh-intro} (see Theorem \ref{thm:KaiMoving} for the precise statement). A work by R.~Iwasa and W.~Kai \cite{IKChern} provides Chern classes from the relative $K$-groups of the pair $(\Xb,D)$ to the Nisnevich motivic cohomology groups $\H^*_{\mathcal{M}, {\rm Nis}}(\Xb|D, \bZ(*))$, while a construction of F.~Binda \cite[Theorem 4.4.10]{BThesis} (see also \cite{Bcyc}) gives cycle classes from the groups of higher $0$-cycles with modulus $\CH^{d+n}(\ol{X}|D, n)$ to the relative $K$-groups $K_n(\ol{X}, D)$. Other positive results are obtained in \cite{KSa}, \cite{RulSaito} and \cite{BKrishna}.

	\subsection{}\label{SecCFT}When $\Xb=C$ is a smooth projective curve over $k$ and $D$ is an effective divisor on it, the Chow group of $0$-cycles with modulus is indeed a classical object. In \cite{Se}, J-P.~Serre introduced and studied the equivalence relation on the set of divisors on $C$ defined by the ``modulus'' $D$ (this explains the choice of the terminology), describing in terms of divisors the relative Picard group $\Pic(C, D)$, that is the group of equivalence classes of pairs $(\mathcal{L}, \sigma)$, where $\mathcal{L}$ is a line bundle on $C$ and $\sigma$ is a fixed trivialization of $\mathcal{L}$ on $D$. When the base field $k$ is finite and $C$ is geometrically connected, the group 
	 \[\varprojlim_{D} \CH_0(C|D)\] is isomorphic to the id\`ele class group of the function field $k(C)$ of $C$.
	
	In \cite{KSa}, M.~Kerz  and S.~Saito introduced Chow groups of $0$-cycles with modulus for varieties over finite fields and used it to prove their main theorem on wildly ramified Class Field Theory. If $X$ is smooth over $k$, take a compactification $X\hookrightarrow \Xb$, with $\Xb$ integral and proper over $k$, and a (possibly non reduced) closed subscheme $D$ supported on $\Xb-X$. Then the group $\CH_0(\Xb|D)$ is defined as the quotient of the group of $0$-cycles $z_0(X)$ modulo rational equivalence with modulus $D$ (see \cite{KSa} and \ref{def.relChowgroup}), and it is used to describe the abelian fundamental group $\pi_1^{ab}(X)$. 
	This work is one of the main sources of motivations for the present paper, and explains our choice of generalizing addivite higher Chow groups to the case of an arbitrary pair. Higher Chow groups with modulus \eqref{eq.ChowgrMod} recover (for $n=0$ and $r=\dim \Xb$) Kerz-Saito definition (see Theorem \ref{Thm.relchow}).

\subsection{}\label{IntroComp}Motivated by \ref{SecCFT}, we can use our relative motivic complexes to give a definition of higher Chow groups with compact support. Let $X$ be a separated scheme of finite type over $k$ and let $\Xb$ be a proper compactification of $X$ such that the complement  of $X$ in $\Xb$ is the support of an effective Cartier divisor $D$. Define for $r, n\geq0$
%\[\H^r_{\mathcal{M}, c}(X, \bZ(n)) = \{ \H^i_{\mathcal{M}}(\Xb| rD, \bZ(n))  \}_r \in pro-\mathcal{A}b\]
\[\CH^r(X,n)_{c} = \{\CH^r(\Xb|mD, n)\}_{m} \in pro-\mathcal{A}b\]
where $pro-\mathcal{A}b$ denotes the category of pro-Abelian groups. This definition does not depend on the choice of the compactification $\Xb$, and it is consistent with the definition of $K$-theory with compact support proposed by M.~Morrow in \cite{Morr}.
\bigskip

We give an overview of the content of the different sections. 

\subsection{}Section \ref{cyclecomplex} contains the  definitions of our objects of interest, namely higher Chow groups with moduli and relative motivic cohomology groups, together with some basic properties. %From that point on we will restrict to the case were the closed subscheme $Y$ is actually given by an effective Cartier divisor $D$ (possibly non reduced) on $\Xb$.
 We define relative Chow groups with modulus, generalizing Kerz-Saito's definition, in Section \ref{relchow}, where they are also shown to be isomorphic to higher Chow groups with modulus for $n=0$. In Section \ref{Cod1} we compute the relative motivic cohomology groups in codimension $1$, showing that
\[\bZ_{\Xb|D}(1)\cong \cO_{\Xb|D}^\times [-1]= \ker(\cO_{\Xb}^\times \to \cO_{D}^\times)[-1]\quad \text{(quasi-isomorphism)}\]
 generalizing Bloch's computation in weight $1$, $\bZ_{\Xb}(1)\cong \cO_{\Xb}^{\times}[-1]$, and proving the first of the expected properties of the relative motivic cohomology groups (see Theorem \ref{theo:weigh-1-mot-coh}).

\subsection{}\label{ab}Suppose that $D$ is an effective Cartier divisor on $\Xb$ such that its reduced part $D_{red}$ is a normal crossing divisor on $\Xb$. Our first main result, presented in Section \ref{clOmega}, is the construction of a fundamental class in the cohomology of relative differentials for a cycle satisfying the modulus condition. More precisely, consider the sheaves
\begin{equation}\label{eq.reldiff.intro}\Omega^{r}_{\Xb|D}=\Omega_{\Xb}^{r}(\log D)\tensor\cO_{\Xb}(-D), \quad r\geq 0\end{equation}
where $\Omega^r_{\Xb}(\log D)$ denotes the sheaf of absolute K\"ahler differential $r$-forms on $\Xb$ with logarithmic poles along $|D_{red}|$. Using El-Zein's explicit construction of the fundamental class of a cycle given in \cite{ElZ}, we can show that if an admissible cycle satisfies the Modulus Condition, then its fundamental class in Hodge cohomology with support appears as restriction of a unique class in the cohomology with support of sheaves constructed out of \eqref{eq.reldiff.intro} (Theorem \ref{thm-purityOmega}). The refined fundamental class is then shown to be compatible with proper push forward (Lemma \ref{pushforDerCat}). Some further technical lemmas are proved in Section 6.

\subsection{}Let $(\Xb, D)$ be as in \ref{ab}. The second main technical result of this paper, presented in Section \ref{clDR}, is the construction, using the fundamental class in relative differentials, of regulator maps from the relative motivic complex $\bZ_{\Xb|D}(r)$ to a the relative de Rham complex of $\Xb$
\[\phi_{dR}\colon \bZ_{\Xb|D}(r) \to \Omega^{\geq r}_{\Xb|D} = \Omega^{\geq r}_{\Xb}(\log D)\tensor \cO_{\Xb}(-D) \quad \text{ in } D^{-}(\Xb_{Zar})\]
where $\Omega^{\geq r}_{\Xb}(\log D)$ denotes the $r$-th truncation of the complex $\Omega_{\Xb}^{\bullet}(\log D)$. The map $\phi_{dR}$ is compatible with flat pullbacks and proper push forwards of pairs. 

When $\Xb$ is a smooth algebraic variety over the field of complex numbers, we can use the same technique to define regulator maps to a relative version of Deligne cohomology (see \eqref{relDelignecomplex}) and to Betti cohomology with compact support
\[\phi_{\mathcal{D}}\colon \epsilon^*\bZ_{\Xb|D}(r) \to \bZ^{\mathcal{D}}_{\Xb|D}(r); \quad \phi_{B}\colon \epsilon^*\bZ_{\Xb|D}(r) \to j_!\bZ(r)_X\quad \text{ in } D^{-}(\Xb_{an}), \]
where $\epsilon$ is the morphism of sites and $j\colon X\to \Xb$ is the open embedding of the complement of $D$ in $\Xb$, generalizing Bloch's regulator from higher Chow groups to Deligne cohomology, constructed in \cite{Bl}. This regulator map is further studied in Section \ref{Sec:addDilog} in the case of additive Chow groups. Evaluated on suitable cycles, our regulator recovers Bloch-Esnault additive dilogarithm (introduced in \cite{BE2}) and gives a refinement of \cite[Proposition 5.1]{BE2} (see \eqref{eq:BE-refined}).

\subsection{}Suppose that $\Xb$ is moreover connected and proper over $\bC$, and  consider the induced maps in cohomology in degree $2r$. We have a commutative diagram (see \ref{RelJac.Dia})
\[
\xymatrix{
&& \HMXDr {2r} \ar[d]^{\phi_{\mathcal{D}}^{2r,r}} \ar[rd]^{\phi_B^{2r,r}}\\
0 \ar[r]& J^r_{\Xb|D}\ar[r] & \H^{2r}_{\mathcal{D}}(\Xb|D, \bZ(r))\ar[r] & \H^{2r}(\Xb_{\an},j_!\ZrX)\\
}\]
and in analogy with the classical situation, we call the kernel $J^{r}_{\Xb|D}$ the $r$-th \textit{relative intermediate Jacobian of the pair $(\Xb, D)$}. We note that they admit a description in terms of extensions groups $\Ext^1$ in the abelian category of enriched Hodge structures defined by S.Bloch and V.Srinivas in \cite{BS}.

One can show that $J^{r}_{\Xb|D}$ fits into an exact sequence 
\[0\to U_{\Xb|D}\to J^{r}_{\Xb|D}\to J^{r}_{\Xb|D_{red}}\to 0,\]
where $U_{\Xb|D}$ is a unipotent group (i.e. a finite product of $\mathbb{G}_a$) 
and $J^{r}_{\Xb|D_{red}}$ (constructed as $J^{r}_{\Xb|D}$ with $D_{red}$ in place of $D$) is an extension of 
a complex torus by a finite product of $\mathbb{G}_m$.
If we compose with the canonical map 
\[\CH^r(\Xb|D) \to \HMXDr {2r}\]
we get an induced map
\begin{equation}\label{eq.rhoXD.intro}
\rho_{\Xb|D}\colon \CH^r(\Xb|D)_{hom} \to J^{r}_{\Xb|D}
\end{equation}
that we may view as the Abel-Jacobi map with $\mathbb{G}_a$-part, 
where $\CH^r(\Xb|D)_{hom}$ is the subgroup of $\CH^r(\Xb|D)$ consisting of the classes of cycles homologically
trivial.

The problem of considering a suitable equivalence relation with modulus for algebraic cycles in order to define a $\mathbb{G}_a$-valued Abel-Jacobi map was already sketched by S.~Bloch in \cite{BlochLetter}, with reference to his joint work with H.~Esnault. In case $r=d:=\dim \Xb$, the Jacobian (or Albanese) $J^{d}_{\Xb|D}$ is actually 
a commutative algebraic group and the map \eqref{eq.rhoXD.intro} becomes
\[
\rho_{\Xb|D}\colon \CH_0(\Xb|D)^0 \to J^{d}_{\Xb|D},
\]
where $\CH_0(\Xb|D)^0$ denotes the degree $0$ part of the Chow gorup $\CH_0(\Xb|D)$ of zero-cycles with modulus.
A different construction of Albanese variety with modulus was given by H.~Russell in \cite{Ru} and (in characteristic zero) by K.~Kato and H.~Russell in \cite{KR} using duality theory for $1$-motives with unipotent part.

In Section 10 we prove, using transcendental arguments, that   $J^{d}_{\Xb|D}$ with $d=\dim \Xb$ is the universal regular quotient of $\CH_0(\Xb|D)^0$, in analogy with the results of H.~Esnault, V.~Srinivas and E.~Viehweg \cite{ESV} and L.~Barbieri-Viale and V.~Srinivas \cite{BarSri} for singular varieties (Theorem \ref{thm.universailty}). 

\subsection*{Acknowledgments}The first author wishes to thank heartily Marc Levine for many friendly conversations and much advice on these topics, for providing an excellent working environment at the University of Duisburg-Essen and for the support via the Alexander von Humboldt foundation. % and the SFB Transregio 45 ``Periods, moduli spaces and arithmetic of algebraic varieties''.
The second author wishes to thank heartily Moritz Kerz for inspiring discussions from which many ideas of this work
arose. He is also very grateful to the department of mathematics of the university of Regensburg
for the financial support via the SFB 1085 ``Higher Invariants'' (Regensburg).
%\newpage

\newcommand{\Set}{\mathbf{Set}}
\newcounter{elno}   
\newenvironment{romanlist}{\begin{list}{\roman{elno})}{\usecounter{elno}}}{\end{list}}
%\usepackage{dsfont}
%\newpage
\section{Cycle complex with modulus}\label{cyclecomplex}

\subsubsection{}\label{StandardCube}We fix a base field $k$. Let $\P^1_k =\Proj{k[Y_0, Y_1]} $ be the projective line over $k$ and denote by $y$ the rational coordinate function $Y_1/Y_0$ on $\P^1_k$. For $n\in \bN\setminus\{0\}, 1\leq i\leq n$, let $p_i^n\colon (\P^1)^n\to (\P^1)^{n-1}$  be the projection onto the $i$-th component. We use on $(\P^1)^n$ the rational coordinate system $(y_1,\ldots, y_n)$, where $y_i = y\circ p_i$. Let \[\square^n = (\P^1_k\setminus\{1\})^n\] and let 
$\iota_{i,\epsilon}^n \colon \square^n\to {\square}^{n+1}$ with 
\[\iota_{i,\epsilon}^n(y_1, \ldots, y_n) = (y_1,\ldots, y_{i-1}, \epsilon, y_i, \ldots, y_n), \text{ for } n\in \bN, 1\leq i\leq n+1, \epsilon \in \{0,\infty\},\]
be the inclusion of the codimension one face given by $y_i = \epsilon, \epsilon \in \{0,\infty\}$. The assignment $n \mapsto \square^n$ defines a cocubical object $\square^{\bullet}$. Note that this is an extended cocubical object in the sense of \cite[1.5]{L}. We conventionally set $\square^{0} =\Spec{k}$.

A face of $\cubb n$ is a closed subscheme $F$ defined by equations of the form
\[
y_{i_1}=\epsilon_1,\dots,y_{i_r}=\epsilon_r\;;\;\; \epsilon_j\in \{0,\infty\}.
\]
For a face $F$, we write $\iota_F: F \hookrightarrow \cubb n$ for the inclusion. Finally, we write $\Fni\subset \Pp n$ for the Cartier divisor on $\Pp n$ defined by $\{y_i=1\}$ and put
$\Fn=\underset{1\leq i\leq n}{\sum} \Fni$.
\subsubsection{}Let $Y$ be a scheme of finite type over $k$, equidimensional over $k$, $D$ an effective Cartier divisor and $F$ a simple normal crossing divisor on $Y$. Assume that $D$ and $F$ have no common components. Let $X$ be the open complement $X= Y- (F+D)$. The following Lemma can be proved using the same argument as  \cite[Proposition 2.4]{KP}.
\begin{lem}\label{Containment_YFD}Let $W$ be an integral closed subscheme of $X$ and let $V\subset W$ be an integral closed subscheme of $W$. 
{  Let $\Wb$ (resp. $\Vb$) be the closure of $W$ (resp. of $V$) in $Y$.  Let $\phiW\colon \Wb^N\to Y$ (resp. $\phiV\colon \Vb^N\to Y$) be the normalization morphism. Then the inequality 
$\phiW^*(D)\leq \phiW^*(F)$ as Cartier divisors on $\Wb^N$ implies the inequality
$\phiV^*(D)\leq \phiV^*(F)$ as Cartier divisors on $\Vb^N$.
}
%Then if $W$ satisfies the modulus condition \ref{moduluscond}, so does $V$.
\end{lem}
%	\begin{proof}We use the same argument as \cite[Proposition 2.4]{KP}. Let $Z=\Wb^N\times_{\Wb}\Vb\; {\hookrightarrow}\; \Wb^N$ and let $Z^N$ be its normalization. By the universal property of the normalization, there exists a unique surjective morphism $h$ making the diagram  
%		\[ \xymatrix{ Z^N \ar@/^1pc/[rr]^{f} \ar[r]\ar[d]^{h} & Z\ar[r] \ar[d] & \Wb^N \ar[d] \ar[rd]^{\phiW}	\\
%		\Vb^N \ar@/_1pc/[rrr]_{\phiV} \ar[r] & \Vb \ar@{^{(}->}[r] & \Wb \ar@{^{(}->}[r] & Y
%		}
%		\] 
%		commutative. Note that all the schemes are of finite type over the base field $k$, so the normalization morphisms are finite and hence $h$ is finite too. 
%By assumptions,  we have \begin{equation}\label{eq_Cont-Lemma}\phiW^*(D)\leq \phiW^*(F)	\end{equation as Cartier divisors on $\Wb^N$. Since $\Wb\cap (Y-F)\cap D = \emptyset$, 
%{ Note that $\Vb$ is not contained in $D_{red}$ nor $F$ and hence intersects properly both $D$ and $F$.}
%We can therefore apply \cite[Lemma 2.1]{KP} to get from {  $\phiW^*(D)\leq \phiW^*(F)$} the inequality
%		\[f^*\phiW^*(D)\leq f^*\phiW^*(F) \text{ on $Z^N$.}\]
%	By the commutativity of the above diagram, we get then $h^*(\phiV^*(F)-\phiV^*(D)) \geq 0$. Since $h$ is a finite surjective morphism between normal varieties, we have that the pullback along $h$ of a Cartier divisor $E$ is effective only if $E$ was already effective (see \cite[Lemma 2.2]{KP}). In particular, we have $\phiV^*(F)-\phiV^*(D) \geq 0$, proving the lemma. 
%		\end{proof}
\subsection{Cycle complexes}
\subsubsection{}\label{cycComplD}
Let $\Xb$ be a scheme of finite type over $k$, equidimensional over $k$, and let $D$ be an effective Cartier divisor on $\Xb$. Let $X$ be the open complement of $D$ in $\Xb$. 
We define the {\it cycle complex of $X$ with modulus $D$} as follows.
%Put $\cub=\bP^1_k-\{1\}$ and $\cubb=\bP^1_k$.
\begin{defi}\label{def-CXD}
%Suppose that $D$ is an effective Cartier divisor on $\Xb$. 
Let $\CXD r n$ be the set of all integral closed subschemes $V$ of 
codimension $r$ on $X\times \cubb n$ which satisfy the following conditions:
\begin{itemize}
\item[(1)]
$V$ has proper intersection with $X\times F$ for all faces $F$ of $\cubb n$.
\item[(2)]
For $n=0$, $\CXD r 0$ is the set of all integral closed subschemes $V$ of 
codimension $r$ on $X$ such that the closure of $V$ in $\Xb$ does not meet $D$.
\item[(3)]
For $n>0$, let $\Vb$ be the closure of $V$ in $\Xb\times \Pp n$ and $\Vb^N$ be its normalization and 
$\phiV:\Vb^N\to \Xb\times \Pp n$ be the natural map. 
If $(D\times \Pp n) \cap \Vb\not=\emptyset$, then the following inequality as Cartier divisors holds:
\begin{equation}\label{def-CXD.eq}
\phiV^*(D\times \Pp n)\leq \phiV^*(\Xb\times \Fn).
\end{equation}
\end{itemize}
%In case $D$ is not an effective Cartier divisor on $\Xb$, we define $\CXD r n=\CXDt r n$, 
An element of $\CXD r n$ is called a {\it relative cycle of codimension $r$ for $(\Xb,D)$}.
\end{defi}

\begin{rem}\label{def-CXD.rem}
The condition \ref{def-CXD}(3) implies
$\Vb\cap (D\times \Pp n)\subset \Xb\times \Fn$ as closed subsets of $\Xb\times \Pp n$, and hence
$\Vb \cap (D\times \cubb n)=\emptyset$ and $V$ is closed in $\Xb\times \cubb n$.
This implies that $\CXD r n$ is viewed as a subset of the set of all integral closed subschemes $W$ of codimension $r$ on $\Xb\times \cubb n$ which intersects properly with $\Xb\times F$ for all faces $F$ of $\cubb n$.
%For all faces $F\subset \cubb n$ of $\dim(F)>0$ (resp. $\dim(F)=0$), $V$ intersects properly with $D\times F$ (resp. $V\cap (D\times F)=\emptyset$).
\end{rem}

Let $V\subset W$ be integral closed subschemes of $X\times\cubb n$ which are closed in $\Xb\times\cubb n$. Lemma \ref{Containment_YFD} shows that if the inequality \eqref{def-CXD.eq} holds for $\Wb$, then it also holds for $\Vb$. This implies then the following
\begin{lem}\label{cor-containment}
Let $V\in \CXD r n$. For a face $F$ of $\cubb n$ of dimension $m$, the cycle
$(id_X\times \iota_F)^* (V)$ on $X\times F\simeq X\times \cubb m$ is in $\CXD r m$.
\end{lem}

\begin{defi}\label{def-zXD}
Let $\zzXD r n$ be the free abelian group on the set $\CXD r n$.
By Lemma \ref{cor-containment}, the cocubical object of schemes $ n \to \cubb n$ gives rises to
a cubical object of abelian groups:
\[
\ul n \to  \zzXD r n\quad (\ul n =\{0,\infty\}^n,\; n=0,1,2,3...).
\]
The associated non-degenerate complex is called the cycle complex $\zXD r \bullet$ of $X$ with modulus $D$:
\[
\zXD r n = \frac{\zzXD r n}{{\zzXD r n}_{degn}}.
\] 
The boundary map of the complex $\zXD r \bullet$ is given by
\[
\partial=\underset{1\leq i\leq n}{\sum}(-1)^i(\partial_i^\infty - \partial_i^0),
\]
where 
$\partial_i^\epsilon: \zXD r n \to \zXD r {n-1}$ is the pullback along $\inie$, well defined by Lemma \ref{cor-containment}.
The $q$-th homology group of the complex will be denoted by
\[
\CHXD r q =\H_q(\zXD r \bullet).
\]
We call it the {\it higher Chow group of $X$ with modulus $D$}.
%For $q=0$ we simply write $\ChXD r =\CHXD r 0$.
 
\end{defi}
\begin{rem}\label{rem-zXD}
\begin{enumerate} \item By Remark \ref{def-CXD.rem}, $\zXD r n$ can be naturally viewed as a subcomplex of $\zX r n$, the (cubical version) of 
Bloch's cycle complex, so that we have a natural map
 \[
 \CHXD r q \to \CHX r q.
 \]
\item The above definition generalizes the additive higher Chow groups defined by Bloch and Esnault \cite{BE}, Park \cite{Pa}, Krishna and Levine \cite{KL}.
In case $\Xb=Y\times \bA^1_k$ with $Y$ of finite type over $k$ and $D=n\cdot Y\times \{0\}$ for $n\in \bZ_{>0}$,
$\CHXD r q$ coincides with $TCH^r(Y,q+1;m)$.
\end{enumerate}
\end{rem}
%
%\bigskip
%
%The following functorialities can be shown.
%
\begin{lem}\label{lem-zXD}Let $\Xb$ and $D$ be as above. Let $r\in \bN$.
\begin{enumerate}
\item 
Let $f:\Yb\to \Xb$ be a proper morphism of schemes of finite type over $k$, equidimensional over $k$. Assume that $f^*D$ is defined as effective Cartier divisor on $\Yb$. Then the push-forward of cycles induces a map of complexes:
\[
f_*: z^{r+\dim(\Yb)-\dim(\Xb)}(\Yb|f^*D,\bullet) \to z^r(\Xb|D,\bullet).
\]
\item 
Let $f:\Yb\to \Xb$ be a flat morphism of schemes of finite type over $k$, equidimensional over $k$. Then the pull-back of cycles induces a map of complexes:
\[
f^*: \zXD r \bullet \to z^r(\Yb|f^*D,\bullet).
\]
\end{enumerate}
\end{lem}
\begin{proof}The proof of the Lemma uses the same argument of \cite{KP}, Theorem 3.1 (1) and (2). 
\end{proof}
%In this section we assume $\Xb$ is smooth over $k$ and $D_{red}$ is a simple normal crossingn divisor on $\Xb$.
\subsubsection{}In \ref{IntroComp}, we introduced the notion of higher Chow group with compact support for a scheme of finite type over $k$ as the cohomology of the pro-complex $\{\zXD r \bullet \}_{D\subset\Xb}$ for a chosen compactification $\Xb$ of $X$ with complement an effective Cartier divisor. The following Lemma shows that this object is well defined and does not depend on the choice of $\Xb$.
\begin{lem}\label{lem-indepcompactn}
Let $X$ be an integral separated scheme of finite type over $k$ and choose a compactification $\tau: X\hookrightarrow \Xb$,
where $\Xb$ is a proper integral scheme over $k$, $\tau$ is an open immersion such that 
$\Xb-X$ is the support of a Cartier divisor.
The pro-complex 
\[
\{\zXD r \bullet \}_{D\subset\Xb}
\]
where $D$ ranges over all effective Cartier divisors with $|D|=\Xb-X$, does not depend on the compactification
$X\hookrightarrow \Xb$. 
\end{lem}

It is indeed enough to show the following Lemma, that is a direct consequence of the definition of the modulus condition and \cite[Lemma 3.2]{KL}.

\begin{lem}\label{lem2-indepcompactn}
Let $X\hookrightarrow \Xb$ and $X\hookrightarrow \Xb'$ be two compactifications as above.
Let $f:\Xb'\to \Xb$ be a proper surjective morphism which is the identity on $X$.
Let $D\subset\Xb $ be an effective Cartier divisor supported on $\Xb-X$ and put $D'=f^*D$.
Then we have the equality (cf. Definition 1.2)
\[
C^r(\Xb|D,n) =  C^r(\Xb'|D',n) 
\]
as subsets of the set of integral closed subschemes of $X\times\cubb n$.
\end{lem}
%\begin{proof}
%This follows immediately from the definition of the modulus condition and \cite[Lemma 3.2]{KL} \end{proof}
\subsubsection{}Let $\Xb$ and $D$ be as in \ref{cycComplD}. 
For $U$ \'etale over $\Xb$,  we let $D$ denote $D\times_{\Xb} U$ for simplicity.
As for Bloch's cycle complex, the presheaves
\[
z^r(-|D,n):\; U\to z^r(U|D,n)
\]
are sheaves on the \'etale site on $\Xb$ and therefore on the small Nisnevich and Zariski site of $\ol{X}$.
%  We define
% \begin{equation}\label{eq.ZXD}
% \bZXDr \quad \text{(resp. $\bZXDret$)})
% \end{equation}
% as the cohomological complex of sheaves $z^r(-|D,2r-i)$ in degree $i$ on $X_{zar}$ (resp. $X_{\et}$).
% \begin{defi}
% %We fix an integer $r>0$.
% We introduce the {\it motivic cohomology of the pair $(\Xb,D)$} as the hypercohomology of the complex of sheaves $\bZXDr$.
%
% \HMXDr q = \bH^q(\Xb_{\zar},\bZXDr ).
% \end{equation*}
% \end{defi}
For $\tau$ any of these topologies and $A$ an abelian group, we define 
\begin{equation}\label{eq.ZXD}
 A_{\ol{X}|D}(r)_\tau = (z^r(-|D_{(-)}, *)_\tau\tensor A)[-2r] \end{equation}
and call it the \textit{relative motivic complex} of the pair $(\ol{X}, D)$. The complex $A_{\ol{X}|D}(r)_\tau$ is unbounded below.
\begin{defi}The \textit{motivic cohomology of the pair $(\ol{X}, D)$} or the \textit{motivic cohomology of $\ol{X}$ with modulus $D$} (with coefficients in $A$) is defined as the hypercohomology of the complex of sheaves $A_{\ol{X}|D}(r)_\tau$,
 \begin{equation*}\label{relmotiviccoh}\H^n_{\cM, \tau}(\ol{X}|D, A(r)) = \mathbb{H}^{n}_\tau(\ol{X}, A_{\ol{X}|D}(r)_\tau).\end{equation*}
    \end{defi}
For $\tau = {\rm Zar}$, we omit it from the notation.     
\subsubsection{}\label{subsec:remark-descent}When $D = \emptyset$, the complex of presheaves $U\mapsto z^r(U|\emptyset, *) = z^r(U, *)$ on $\ol{X}_{\rm Zar}$ satisfies the Mayer-Vietoris property (see \cite[Section 3]{Bloch86} for the statement and \cite{LevineLoc} for the proofs) and therefore has Zariski descent, in the sense that the natural maps
\[\CH^r(\ol{X}, 2r-n) = \H^n(z^r(\ol{X}, *)[-2r]) \xrightarrow{\simeq} \mathbb{H}^{n}_{\rm Zar}(\ol{X}, \bZ_{\ol{X}}(r)_{\rm Zar})\]
are isomorphisms. When $D\neq \emptyset$, the situation is considerably more intricate. The natural map
\[\CH^r(\ol{X}|D, 2r-n) = \H^n(z^r(\ol{X}|D, *)[-2r]) \to \mathbb{H}^{n}_{\rm Zar}(\ol{X}, \bZ_{\ol{X}|D}(r)_{\rm Zar}) = \H^n_{\cM, {\rm Zar}}(\ol{X}|D, \bZ(r))\]
has been object of several speculations and, in general, is not expected to be an isomorphism. An evident example is the case where $\ol{X}$ is a smooth projective variety and $D$ is a very ample divisor on it. For $n=2r$, there are simply no cycles missing $D$, so that the group $\CH^r(\ol{X}|D, 0) = 0$ for $r<\dim \ol{X}$, while, in general, the groups $\H^{2r}_{\cM, {\rm Zar}}(\ol{X}|D, \bZ(r))$ have no reason to be zero (we discuss the case $r=1$ in Section 4).

To get a more serious example, which illustrates the rather pathological nature of the ``naive'' cycle groups with modulus (i.e., the actual homology groups of $z^r(\ol{X}|D, *)$ and not the relative motivic cohomology groups defined above), we mention the following result (see \cite{KrishnaOnCycles}).
\begin{prop}[A.~Krishna] Let $k$ be an algebraically closed field of characteristic zero with infinite transcendence degree over the field of rational numbers. Let $Y$ be a connected projective curve over $k$ of genus $g\geq 1$.  For $m\geq 2$, let $D_m = \Spec{k[t]/(t^m)} \hookrightarrow \mathbb{A}^1_k$. Then for any inclusion $i\colon \{P\}\hookrightarrow Y$ of a closed point, the localization sequence 
    \[\CH_0(\{P\}\times \mathbb{A}^1 |\{P\} \times D_m ) \xrightarrow{i_*} \CH_0(Y\times \mathbb{A}^1_k| Y\times D_m) \xrightarrow{j^*} \CH_0(Y\setminus \{P\} \times \mathbb{A}^1_k| Y\setminus \{P\} \times D_m )\to 0\]
    fails to be exact at $\CH_0(Y\times \mathbb{A}^1_k| Y\times D_m)$.
    \end{prop}
For different counter-examples to the descent property, see \cite{RulSaito}, before Theorem 3.
    
On the bright side, we mention the following important result recently obtained by W.~Kai (see \cite[Theorem 1.4]{KaiMoving}). If $(\ol{X}, D)$ and $(\ol{Y}, E)$ are two pairs of schemes of finite type over $k$ equipped with effective Cartier divisors, we say that a morphism $f\colon \Xb \to \ol{Y}$ is an {\it admissible morphism of pairs} if $f^*E$ is defined as Cartier divisor on $\Xb$ and $f$ restricts to a morphism $D\to E$.
\begin{theo}[W.~Kai]\label{thm:KaiMoving} Let $(\ol{X}, D)$ and $(\ol{Y}, E)$  be pairs of equidimensional schemes of finite type over $k$ and effective divisors on them. Assume that $\ol{Y}\setminus E$ is smooth. Let $f\colon (\ol{X},D)\to (\ol{Y}, E)$ be an admissible morphism of pairs. Then there are natural maps of abelian groups
    \[f^*\colon  \H^n_{\cM, {\rm Nis}}(\ol{Y}|E, \bZ(r))\to \H^n_{\cM, {\rm Nis}}(\ol{X}|D, \bZ(r)).
    \]
    This makes Nisnevich motivic cohomology with modulus contravariantly functorial for any map of smooth schemes with effective Cartier divisors.
    \end{theo}
    
\begin{rem}It is an interesting question whether the groups $\H^{2r}_{\cM}(\ol{X}|D, \Z(r))$ coincides with some modified Chow group with modulus defined using relative $K$-sheaves or relative Milnor $K$-sheaves. There are some results (see \cite{KSa}, \cite[Theorem 1.9]{BKrishna} and \cite[Theorem 3]{RulSaito}) related to this question, indicating that the descent property may hold for higher Chow groups with modulus at least in the case of $0$-cycles.
    \end{rem}
%In the following Sections we will prove some properties of these groups and relate them to various cohomology theories, namely  a  relative version of de Rham cohomology and Deligne cohomology, constructing the corresponding Regulator maps.

%\newpage

\section{Relative Chow groups with modulus}\label{relchow}
 
%\subsection{Relative Chow group with modulus}\label{relChow}
Let $\Xb$ and $D$ be again a pair consisting of an integral scheme $\Xb$ of finite type over $k$ and an effective Cartier divisor $D$ on it. Fix an integer $r\geq 1$.
In this Section we give a description of the groups $\CHXD r 0$ in terms of the relative Chow groups $\ChXD r$ defined below. 

\begin{defi}\label{def.relChowgroup}Let $Z^r(X)$ (resp. $Z^r(\Xb)_D$) be the free abelian group on the set $C^r(X)$ (resp. $C^r(\Xb)_D$) 
of integral closed subschemes $W\subset X$ of codimension $r$ 
(resp. integral closed subschemes $W\subset \Xb$ of codimension $r$ such that $ W\cap D=\emptyset$). 
For an integral scheme $Z$ and an effective Cartier divisor $E$ on $Z$, we set
\begin{equation}\label{eq.DefG}
G(Z,E)= \indlim U \Gamma(U,\Ker(\cO_U^\times \to \cO_E^\times)),
\end{equation}
where $U$ ranges over all open subscheme of $Z$ containing $|E|$.
We then put
\begin{equation}\label{eq.PhirXD}
\PhirXD =\underset{W\in C^{r-1}(X)}{\bigoplus}\; G(\Wb^N,\gamma_W^*D),
\end{equation}
where $\Wb^N$ denotes the normalization of the closure $\Wb$ of $W$ in $\Xb$ and $\gamma_W^*D$ is the pullback of 
the Cartier divisor $D$ via the natural map $\gamma_W:\Wb^N \to \Xb$. We set 
\begin{equation}\label{eq.relchow}
\ChXD r=\Coker\big(\PhirXD \rmapo{\delta} Z^r(\Xb)_D\big),
\end{equation}
where $\delta$ is induced by the composite of the divisor map on $\Wb^N$ and 
the pushforward map of cycles via  $\gamma_W$ for $W\in C^{r-1}(X)$. The groups $\ChXD r$ are called {\it Chow groups of $\Xb$ with modulus $D$}.
\end{defi}
\begin{rem}The notations of Definition \ref{def.relChowgroup} should be compared with the one in \cite[1.1 and 2.9]{KSY}. Note that in the definition of a modulus pair $(\Xb, Y)$ in {\it loc.~cit.}, $X = \Xb - |Y|$ is required to be quasi-affine over $k$. In this paper we don't need this condition. 
	\end{rem}
The main result of this section is the following
\begin{theo}\label{Thm.relchow}
There is a natural isomorphism
\[ \CH^r(\Xb|D,0) \isom \CH^r(\Xb|D)\]
\end{theo}
\begin{rem} As noticed in \ref{subsec:remark-descent}, when $\Xb$ is projective and $D$ is an ample divisor, there are no positive dimensional closed subschemes of $\Xb$ missing $D$ and the groups $\CH^r(\Xb|D)$ are therefore trivial except for the case $d=\dim \Xb$. This is one of the reason for introducing the relative motivic cohomology groups as hypercohomology of the relative cycle complex rather then as the naive higher Chow groups.
     
  An interesting problem suggested by Bloch \cite{BlochLetter} is to extend our definition of cycles with modulus $\CH^r(\Xb|D, n)$ replacing the divisor $D$ with a closed subscheme $T\subset \Xb$. If $\dim T <r$, there are many cycles of dimension $r$ not meeting $T$, so that one gets non-trivial generators of the corresponding relative Chow group. A possible way to extend our definition is to set
  \[ \CH^r(\Xb|T, n) = \CH^r(\tilde{\Xb}|\tilde{T}, n)\]
  where $\pi\colon \tilde{\Xb}\to \Xb$ is the blow-up of $\Xb$ with center $T$ and $\tilde{T}$ is the divisor $\pi^{-1}(T)$. The analogue of Theorem \ref{Thm.relchow} in this generality implies that $\CH^r(\Xb|T,0)$ coincides with the group considered by Bloch in  {\it loc.~cit.}.   
\end{rem}
    
\subsection{A description of relative cycles}\label{description}%
%For the proof of Theorem \ref{Thm.relchow} we need to give an alternative and more concrete presentation of the modulus condition of Definition \ref{def-CXD}. This is the content of Lemma \ref{lem-cyclecondition0} and \ref{lem.divisor-functionMC} below.
%(see Lemma \ref{lem-cyclecondition}).
\subsubsection{}
 Let $n\in \bN$. 
%This will be used also in the computation of relative higher Chow groups of codimension $1$ 
%in the next section (see Theorem \ref{thm.codim1}). 
For $1\leq i\leq n$, we denote by $ \cubbb n i$ the closure of the $n$-th dimensional box $\cub^n$  in the $i$-th direction, i.e.
\[
\Pp n \supset \cubbb n i =
\cub \times \cdots \times \overset{\overset{i}{\vee}}{\bP^1} \times \cdots \times \cub
= \Pp n - \underset{j\not=i}{\sum} \Fnj \;\simeq \;\cubb {n-1}\times \bP^1.
\]
We let $\Fni$ denote $\Fni\cap \cubbb n i$ for simplicity and write $pr_i:\Xb\times \cubbb n i \to \Xb\times \cubb {n-1}$ for the projections removing the $i$-th factor $\bP^1$. 
%Recall that we have chosen a standard coordinate $y$ on $\bP^1-\{\infty\}$.
%\medbreak

%The following lemma is an easy consequencce of \ref{def-CXD}[3].

\begin{lem}\label{lem-cyclecondition0}
%Let $n\geq 1$ be an integer.
Let $V\in C^{r}(X\times \cub^n)$ be an integral cycle,  
%Let $V\subset X\times \cubb n$ be an integral closed subscheme of codimension $r$ 
and $\Vb$  the closure of $V$ in $\Xb\times \Pp n$.
For $1\leq i\leq n$, let $\Vb_i$ be the closure of $V$ in $\Xb\times \cubbb n i$, $\Vb_i^N$ 
be its normalization. Let  $\phiVi:\Vb_i^N\to \Xb\times \cubbb n i$ be the natural map. Then 
the condition (3) of Definition \ref{def-CXD} implies the following condition:
\begin{itemize}
\item[(3)']
The following inequality as Cartier divisors holds for all $1\leq i\leq n$:
\begin{equation}\label{def-CXD.eq'}
\phiVi^*(D\times \cubbb n i)\leq \phiVi^*(\Xb\times \Fni).
\end{equation}
\end{itemize}
The converse implication holds if either $n=1$ or none of the components of $\Vb\cap (D\times\Pp n)$ is contained in 
$\underset{1\leq i\leq n}{\cap}\; \Xb\times \Fni$. 
\end{lem}
\begin{proof}
The condition (3)' follows from  Definition \ref{def-CXD}(3) by  base change via
the open immersion $\Pp n - \underset{j\not=i}{\sum} \Fnj \hookrightarrow \Pp n$.
The converse implication holds if the generic points of $\Vb\cap (D\times\Pp n)$ are all in
\[
\underset{1\leq i\leq n}{\cup}\; (\Pp n - \underset{j\not=i}{\sum} \Fnj)\;=
\left.\left\{\begin{gathered}
 \Pp n \;\quad \text{if $n=1$,}, \\ 
 \Pp n - \underset{1\leq i\leq n}{\cap}\; \Fnj \;\quad \text{if $n>1$,}\\
\end{gathered}\right.\quad
\right.
\]
proving the last assertion.
\end{proof}

\begin{rem}\label{lem-cyclecondition0.rem}
The condition (3)' of Lemma \ref{lem-cyclecondition0}  implies
$\Vb_i\cap (D\times \cubbb n i)\subset \Xb\times \Fni$ as closed subsets of $\Xb\times \cubbb ni$, and this in turn implies that $V$ is closed in $\Xb\times \cubb n$, namely $V\in C^r(\Xb\times \cub^n)_{D\times \cub^n}$
(cf. Definition \ref{def.relChowgroup}).
\end{rem}

%Put $\cub=\bP^1_k-\{1\}$ and $\cubb=\bP^1_k$.
\begin{lem}\label{lem-cyclecondition}
Let $n\geq 1$ and let $V\in C^r(\Xb\times \cub^n)_{D\times \cub^n}$. Suppose that $V$ intersects properly $\ini \infty(X \times \cubb {n-1})$.
%Let $V\subset \Xb\times \cubb n$ be an integral closed subscheme of codimension $r$ such that
%$V\cap (D\times \cubb n)=\emptyset$ and that it intersects properly with $\ini \infty(X \times \cubb {n-1})$.
 Write
\[
\partial_i^\infty V=(\ini \infty)^{-1}(V)\subset \Xb\times \cubb {n-1}.
\]
%Note that $\partial_i^\infty V$ is closed in $\Xb\times \cubb {n-1}$ by the above assumption.
For $1\leq i\leq n$, let $\Vb_i$ be the closure of $V$ in $\Xb\times \cubbb n i$ and put 
\[
\Wb_i=pr_i(\Vb_i) \subset \Xb\times \cubb {n-1}, \quad
\Wb_i^o=\Wb_i \backslash \partial_i^\infty V,\quad \Vb_i^o=\Vb_i\times_{\Wb_i} \Wb_i^o.
\]
Then $\Vb_i^o$ is finite over $\Wb_i^o$ and 
\[
\Vb_i\cap (D\times \cubbb n i)\subset \Vb_i^o\subset \Wb_i^o[y],
\]
where $\Wb_i^o[y]$ denotes $\Wb_i^o\times (\bP^1-\{\infty\})$.
\end{lem}
\begin{proof}
	By definition, $\Vb_i^o$ is proper over $\Wb_i^o$ and closed in $\Wb_i^o[y]=\Wb_i^o\times(\bP^1-\{\infty\})$. 
	Since $\Wb_i^o[y]$ is affine over $\Wb_i^o$, we have immediately that
	$\Vb_i^o$ is finite over $\Wb_i^o$. 
	By  assumption, $V\cap (D\times \cubb n)=\emptyset$, so that we have $\partial_i^\infty V\cap (D\times \cubb {n-1})=\emptyset$. Hence $D\times \cubbb ni=pr_i^{-1}(D\times \cubb {n-1})$ does not intersect $pr_i^{-1}(\partial_i^\infty V)$, and therefore $\Vb_i\cap (D\times \cubbb ni)\subset \Vb_i^o=\Vb\backslash pr_i^{-1}(\partial_i^\infty V)$. 
\end{proof}
\begin{lem}\label{lem.divisor-functionMC} Let $V$ be as in   Lemma \ref{lem-cyclecondition}. Then the condition (3)' of Lemma \ref{lem-cyclecondition0} for $V$ is equivalent to the following condition:
\begin{itemize}
\item[(3)'']
Let $\Wb_i^N$ be the normalization of $\Wb_i$ and $\WbNo_i=\WbN_i\times_{\Wb_i}\Wb^o_i$.
Then there exists an integer $\nu \geq 1$ such that 
\[
\Vb_i^o \times_{\Wb_i} \Wb_i^N \subset \WbNo_i[y]:=\WbNo_i\times(\bP^1-\{\infty\})
\] 
is the divisor of a function of the form
\medbreak\noindent
$f= (1-y)^m+ \underset{1\leq \nu \leq m}{\sum}\; a_\nu (1-y)^{m-\nu} 
\qwith a_\nu\in \Gamma(\WbNo_i,I^\nu_{\Wb_i^N}),$
\medbreak\noindent
where $I_{\Wb_i^N}\subset \cO_{\Wb_i^N}$ is the  ideal sheaf for the divisor
$D\times_{\Xb} \Wb_i^N$ in $\Wb_i^N$ and $I^\nu_{\Wb_i^N}$ denotes its $\nu$-th power.
\end{itemize}
%Conversely, if $V$ satisfies the above conditions and $\partial_i^\infty V\cap (D\times\cubb {n-1})=\emptyset$,
%then $V$ satisfies the condition \ref{lem-cyclecondition0}(3)'.
\end{lem}
\begin{proof}% so that $\Vb_i\cap (D\times \cubbb ni)=\Vb_i^o\cap (D\times \cubbb ni)$.
%We show the last assertion of Lemma \ref{lem-cyclecondition}.
Let $\yb\in \Gamma(\Vb_i^o,\cO)$ be the image of $y$. By Lemma \ref{lem-cyclecondition}, $\Vb_i^o$ is finite over $\Wb^o_i$ and 
the minimal polynomial of $\yb$ over the function field $K$ of $\Wb_i$ can be written as:
\[
f(T)= (1-T)^m+ \underset{1\leq \nu \leq m}{\sum}\; a_\nu (1-T)^{m-\nu} \;\in 
\Gamma(\WbNo_i,\cO)[T].
\]
We claim that $\Vb_i^o\times_{\Wb_i} \Wb_i^N\subset \WbNo_i[y]$ coincides with the divisor of the function
\[
h=(1-y)^m+ \underset{1\leq \nu\leq m}{\sum}\; a_\nu (1-y)^{m-\nu}\in  \Gamma(\WbNo_i[y],\cO).
\]
%\begin{claim}\label{lem-cyclecondition-claim2}
%\end{claim}
Indeed, since $\div(h) \subseteq \Vb_i^o\times_{\Wb_i} \Wb_i^N$, it is enough to show that it is irreducible. But this is clear as $\div(h)$ is finite over $\WbNo_i$ and its generic fiber over $\WbNo_i$ is irreducible. The last assertion of Lemma \ref{lem.divisor-functionMC} follows then from the following.
\begin{claim}\label{lem-cyclecondition-claim1}
The condition \ref{lem-cyclecondition0}(3)' holds if and only if
$a_\nu\in \Gamma(\WbNo_i,I^\nu_{\Wb_i^N})$ for all $\nu \geq 1$.
\end{claim}
\def\Ulam{U_\lambda}
\def\pilam{\pi_\lambda}
\def\piblam{\pib_\lambda}

%It suffices to show $a_\nu\in \Gamma(\WbNo_i\times_{\Xb} \Ulam,I^\nu_{\Wb_i^N})$ for any open covering $\Xb=\underset{\lambda}{\cup} \Ulam$ and any $\lambda$.
The question is local on $\Xb$ and we may assume that the ideal sheaf $I_D\subset \cO_{\Xb}$ is generated by a regular function $\pi \in \Gamma(\Xb,\cO)$. 
%Let $\piblam\in \Gamma(\Vb_i_i\times_X \Ulam,\cO)$ be the images of $\pilam$. 
Note that, by Lemma \ref{lem-cyclecondition}, we have
\[\Vb_i \cap (D\times \cubbb  n i) \subset  \Vb_i^{o},\]
so that  we can actually remove $\partial_{i}^{\infty} V$ to check the modulus condition. If we still denote by $\pi$ the image of $\pi$ in $\Gamma(\Vb_i^o, \cO)$, we see then that the condition \ref{lem-cyclecondition0}(3)' is equivalent to require that 
\begin{equation}\label{lem-cyclecondition-eq1}
\theta:=\displaystyle{\frac{1-\yb}{\pi}}\in \Gamma(\Vb_i^N\times_{\Wb_i} \Wb_i^o,\cO),
\end{equation} 
for every $i=1, \ldots, n$, where $\Vb_i^N$ is the normalization of $\Vb_i$. Since $\pi$ does not vanish identically on $\Wb_i$, we have $\pi\in K$ and thus the minimal polynomial of $\theta$ over $K$ is 
\[
g(T) =T^m+ \underset{1\leq \nu \leq m}{\sum}\; \frac{a_\nu}{\pi^\nu} \;T^{m-\nu}.
\]
Since $\Vb_i^o$ is finite over $\Wb_i^o$, \eqref{lem-cyclecondition-eq1} is equivalent to the condition that 
$\theta$ is finite over $\Gamma(\WbNo_i,\cO)$, which is equivalent to  
\[
\frac{a_\nu}{\pi^\nu}\in\Gamma(\WbNo_i,\cO)\quad\text{ for all }\nu,
\]
completing the proof of Claim \ref{lem-cyclecondition-claim1}.
\end{proof}

\subsection{Proof of theorem \ref{Thm.relchow}}
By definition, the groups $\CHXD r 0$ and $\ChXD r$ have the same set of generators $\zzXD r 0=Z^r(\Xb)_D$ and to prove the theorem it suffices to construct a surjective homomorphism 
$\phi\colon \zzXD r 1\to \PhirXD$ which fits into a commutative diagram
\begin{equation}\label{eq0.proof-thmrelchow}
\begin{CD}
 \zzXD r 1 @>{\partial}>> \zzXD r 0 \\
 @VV{\phi}V @| \\
 \PhirXD @>{\delta}>> Z^r(\Xb)_D \\ 
\end{CD}
\end{equation}
Let $V\in \CXD r 1$ be an integral cycle of codimension $r$, $V\subset X\times \cub^1$ satisfying the modulus condition of Definition \ref{def-CXD}.
By Remark \ref{def-CXD.rem}, we note that $V$ is actually closed in $\Xb\times \cub^1$.
Let $\Vb$ be the closure of $V$ in $\Xb\times \bP^1$, $\Wb\subset \Xb$  its image along the projection $\Xb\times \bP^1\to\Xb$ and $\Wb^N$
the normalization of $\Wb$. We write  $\gamma_W$  for the natural map $\Wb^N \to \Xb$. Let $\partial^\infty V$ denote $\iota_\infty^{-1}(V)$ (resp. $\partial^0 V$ denote $\iota_0^{-1}(V)$), where $\iota_\infty$ and $\iota_0$ are the two closed immersions $\Xb \to \Xb\times\bP^1$ induced by $\infty\in \bP^1$ and $0 \in\bP^1$ respectively. The restriction to $V$ of the projection $X\times \bP^1\to\bP^1$  induces a rational function $g_V\in k(\Vb)^\times$, and by \cite[Prop.1.4 and \S 1.6]{Ful} we have
\begin{equation}\label{eq1.proof-thmrelchow}
\partial V =  \partial^\infty V- \partial^0 V= {\gamma_W}_* (\div_{\Wb^N}(N_{k(\Vb)/k(\Wb)} g_V)),
\end{equation}
where $N_{k(\Vb)/k(\Wb)}\colon k(\Vb)^\times \to k(\Wb)^\times$ denotes the norm map induced by $\Vb\to \Wb$,
which is generically finite by Lemma \ref{lem-cyclecondition}. We claim that 
\begin{equation}\label{eq.claim.proof} N_{k(\Vb)/k(\Wb)} g_V \in G(\Wb^N,\gamma_W^*D)
\end{equation}
Indeed, let $\Wb^o=\Wb\backslash \partial^\infty V$ and $\WbNo=\Wb^N\times_{\Wb} \Wb^o$. By Lemma \ref{lem-cyclecondition} we have that
% let $\partial^\infty V$ denote $\iota_\infty^{-1}(V)$ with $\iota_\infty : \Xb \to \Xb\times\bP^1$ induce by $\infty\in \bP^1$.
%By Lemma \ref{lem-cyclecondition}, letting $\Wb^o=\Wb\backslash \partial^\infty V$ and $\WbNo=\Wb^N\times_{\Wb} \Wb^o$,
\[
\Vb \times_{\Wb} \WbNo \subset \WbNo[y]=\WbNo\times(\bP^1-\{\infty\})
\] 
is the divisor of a function of the form
\[f(y)= (1-y)^m+ \underset{1\leq \nu \leq m}{\sum}\; a_\nu (1-y)^{m-\nu} \]
with $a_\nu\in \Gamma(\WbNo,I^\nu_{\Wb^N}(D))$ and where
  $I_{\Wb^N}(D)\subset \cO_{\Wb^N}$ denotes the ideal sheaf of the pullback $\gamma_W^*D$ of $D$ to $\Wb^N$.
Since $\partial^\infty V\cap D=\emptyset$, \eqref{eq.claim.proof} follows from the equality: 
\[
N_{k(\Vb)/k(\Wb)} g_V = f(0) = 1 + \underset{1\leq \nu \leq m}{\sum}\; a_\nu .
\]
We define now $\phi\colon \zzXD r 1\to \PhirXD$ by the assignment 
\[
\phi(V)= N_{k(\Vb)/k(\Wb)} g_V \in G(\Wb^N,\gamma_W^*D)\subset \PhirXD\quad(V \in \CXD r 1).
\]
The commutativity of \eqref{eq0.proof-thmrelchow} then follows from \eqref{eq1.proof-thmrelchow}.

To complete the proof of Theorem \ref{Thm.relchow}, it remains to show the surjectivity of $\phi$.
Take $W\in C^{r-1}(X)$ and $g\in G(\Wb^N,\gamma_W^*D)$. % (cf. \eqref{eq.PhirXD}).
%Let $\Sigma\subset \Wb^N$ be the union of the prime divisors $T$ on $\Wb^N$ such that $v_T(g)<0$, where $v_T$ is the valuation associated to $T$. 
Let $\Sigma\subset \Wb^N$ be the closure of the union of points $x\in \Wb^N$ of codimension one such that $v_x(g)<0$,
where $v_T$ is the valuation associated to $x$. 
Since $\Wb^N$ is normal, we have $g\in \Gamma(\Wb^N-\Sigma,\cO)$
and the assumption $g\in G(\Wb^N,\gamma_W^*D)$ implies
\begin{equation}%\label{eq3.proof-thmrelchow}
g-1 \in \Gamma(\Wb^N-\Sigma,I_{\Wb^N}(D)).
\end{equation}
Now we identify $g$  with a morphism $\psi_g: \Wb^N-\Sigma \to \bP^1-\{\infty\}$.
Let $\Gamma\subset \Wb^N \times \bP^1$ be the closure of the graph of $\psi_g$, 
$\Vb\subset \Wb\times \bP^1$  its image along   $\Wb^N\times \bP^1\to\Wb\times \bP^1$ and $V=\Vb\cap (\Wb\times \cub^1)\subset \Xb\times \cub^1$. 
It suffices  to show that the cycle $V$ defined in this way belongs to $\zzXD r 1$, i.e.~that it satisfies the modulus condition, and that $\phi(V) = g$. Note that once the first assertion is proven, the second follows from the very construction of $V$. 

We have the following diagram of schemes
\begin{equation}\label{eq1.52.proof-thmrelchow}
\xymatrix{
\Wb \ar[r]^{\hskip -20pt\iota_\infty} & \Wb\times\bP^1 & \ar[l] \Vb \\
\Wb^N \ar[r]^{\hskip -20pt\iota_\infty}\ar[u]^{\pi_{\Wb}} \ar[rd]_{id_{\Wb}} 
& \Wb^N\times\bP^1\ar[u]\ar[d]^{pr} & \ar[l]\ar[u]_{\pi_{\Vb}}\ar[ld]^{pr_\Gamma} \Gamma \\
&\Wb^N\\
}
\end{equation}
where the horizontal arrows denoted  $\iota_\infty$ are induced by the inclusion $\infty\in \P^1$ and where the squares are cartesians. %Indeed this is obvious for the left one, and for the right one we notice that the natural map $\Gamma\to \Vb\times_{\Wb} \Wb^N$ is an isomorphism, since it is both birational and closed (both are closed subschemes of $\Wb^N\times \bP^1$). 
Moreover, we note that
\begin{equation}\label{eq3.proof-thmrelchow}
\Sigma \subset pr_{\Gamma}\big((\Wb^N\times {\infty})\cap \Gamma\big).
\end{equation}
Indeed, let $\eta$ be a generic point $\eta \in \Sigma$. Then there exists a unique $\xi \in \Gamma$, of codimension $1$, such that $\eta = pr_{\Gamma}(\xi)$ and we have $v_\xi(g) = v_\eta(g)<0$ (note that $g\in k(\Gamma)=k(\Wb^N)$, as $\Gamma$ is birational to $\Wb^N$). Such point $\xi$ is actually in $(\Wb^N\times {\infty})\cap \Gamma$:  the projection $\Wb^N\times\bP^1\to \bP^1$ induces a well-defined morphism
\begin{equation*}
\Gamma\backslash (\Wb^N\times {\infty}) \to \bP^1-\{\infty\}
\end{equation*}
that corresponds to $g$, so that the point $\xi$ where $g$ is not regular is forced to belong to $(\Wb^N\times {\infty})\cap \Gamma$.
From the diagram \eqref{eq1.52.proof-thmrelchow}, one sees that \eqref{eq3.proof-thmrelchow} is equivalent to 
\begin{equation*}%\label{eq2.proof-thmrelchow}
\Sigma \subset \pi_{\Wb}^{-1}(\iota_\infty^{-1}(\Vb))=\iota_\infty^{-1}(\Gamma),
\end{equation*}
so that
$\WbNo:=\Wb^N\times_{\Wb}(\Wb\backslash \iota_\infty^{-1}(\Vb)) \subset\Wb^N-\Sigma$, and hence
\[
\Vb\times_{\Wb}\WbNo =\Gamma\times_{\Wb^N}\WbNo \subset \WbNo\times(\bP^1-\{\infty\})
\]
is given by $\Gamma\cap \big((\Wb^N-\Sigma)\times (\bP^1-\{\infty\})\big)$. This is, by definition, the graph of $\psi_g$ and hence it is the divisor of $y-g$ where $y$ is the standard coordinate of $\bP^1$. By the equivalent condition given by Lemma \ref{lem.divisor-functionMC}, this proves that $V$ satisfies the modulus condition of Definition \ref{def-CXD} (and in particular it is closed in $X\times \cub^1$). This completes the proof of Theorem \ref{Thm.relchow}.

%For a generic point $x\in \Sigma$, there is a unique point $y\in \Gamma$ such that $x=pr_{\Gamma}(y)$ ($y$ is always
%of codimension one in $\Gamma$) and we have $v_y(g)=v_x(g)<0$ for $g\in k(\Gamma)=k(\Wb^N)$.

%\newpage

\section{Relative cycles of codimension $1$}\label{Cod1}

Let $\Xb$ and $D$ be again as in Section \ref{cyclecomplex}, with $D$ an effective Cartier divisor on $\Xb$. 
We assume $D\not=\emptyset$.
In this Section we investigate the relative motivic cohomology groups ${\H^{n}_{\cM}(\Xb|D,\bZ(1))}$ in weight $1$. 
\begin{theo}\label{theo:weigh-1-mot-coh}Let $\ol{X}$ be a regular $k$-scheme. There is a quasi-isomorphism of complexes of Zariski sheaves on $\ol{X}$
    \[\mathbb{Z}_{\Xb|D}(1)\xrightarrow{\cong} \cO_{\ol{X}|D}^{\times}[-1],\]
    where $\cO_{\ol{X}|D}^{\times}$ denotes the kernel of the natural surjection $\cO_{\Xb}^\times\to \cO_D^\times$.
    \end{theo}
    \begin{coro}Let $\Xb$ be a regular $k$-scheme and $D$ an effective Cartier divisor on it. Then we have:
\[
\H_{\mathcal{M}}^{p}(\Xb|D, \bZ(1))=
\left.\left\{\begin{gathered}
 0 \\ 
 \Pic(\Xb,D) \\
 \Gamma(\Xb,  \cO_{\ol{X}|D}^{\times}) \;\\
\end{gathered}\right.\quad
\begin{aligned}
&\text{if $p> 2$,}\\
&\text{if $p=2$,}\\
&\text{if $p=1$.}
\end{aligned}\right.
\]
        \end{coro}
\subsubsection{}\label{Not-RelCod1-1}Assume in what follows that $\Xb=\Spec(A)$ is the spectrum of a regular local ring $A$. We write $I_D$ or $I$ for the invertible ideal of $D\subset \Xb$. Let $A[t_1,\dots,t_n]$ be the polynomial ring in the variables $t_i$ on $A$ and write $f\in A[t_1,\dots,t_n]$ as
\[
f=\underset{\lamub \in \Lam}{\sum}\; a_{\lamub} (1-\tub)^{\lamub} \quad (a_{\lamub}\in A),
\]
for the multi-index 
\[
\Lambda=\{\lamub=(\lam_1,\dots,\lam_n)\;|\; \lam_i\in \bZ_{\geq 0}\},\quad
(1-\tub)^{\lamub} =\underset{1\leq i\leq}{\prod}\; (1-t_i)^{\lam_i}.
\]
We say that 
$f$ is \it admissible for $I_D$ \rm if $a_{(0,\cdots,0)}\in A^\times$ and 
\[
a_{\lamub}\;\in I_D^{|\lam|} \qfor \lamub\not=(0,\cdots,0),
\]
where $|\lamub|=\max\{\lam_i|1\leq i\leq n\}$ for $\lamub=(\lam_1,\cdots,\lam_n)$.
We let $\PAt n$ denote the set of $f\in A[t_1,\dots,t_n]$ which are admissible for $I_D$.
It's easy to check that that $\PAt n$ forms a monoid under multiplication.

\subsubsection{}\label{Not-RelCod1-2}Let $y = Y_0/Y_1$ be the  rational coordinate function on $\bP^1$ of \ref{StandardCube}. %Recall that we have chosen a standard coordinate $y$ on $\cub=\bP^1-\{\infty\}$.
We fix the affine coordinate 
$\displaystyle{t= 1-\frac{1}{1-y}}$ on $\cub = \bP^1\setminus \{1\}$, so that $\cub=\Spec(k[t])$. Similarly, we choose a coordinate system $t_1,\ldots, t_n$ on $\cub^n$ so that $X\times \cub^n = \Spec{A[t_1,\ldots, t_n]}$.
\begin{lem}\label{lem-cyclecondition1}
We keep the notations of \ref{Not-RelCod1-1} and \ref{Not-RelCod1-2}. %We assume that $X=\Spec(A)$ is affine and let $I\subset A$ be the ideal for $D\subset X$.
Let $V\subset X\times \cubb n$ be an integral closed subscheme of codimension $1$.
Then $V\in \zzXD 1 n$ if and only if there exists $f\in \PAt n$ such that $V=\div(f)$ on $X\times \cubb n$. 
%the condition \ref{lem-cyclecondition0}(3)' is equivalent to the condition that there is an
%irreducible polynomial $f\in A[t_1,\dots,t_n]$ which is admissible for $I$, 
\end{lem}
\begin{proof}
%Let the notation be as in the proof of \ref{lem-cyclecondition0}.
%Since $X$ is local, the open subset $U\subset X$ containing $D$ is $X$.
Assume $V\in \zzXD 1 n$. Then $V$ satisfies the conditions of Lemma \ref{lem-cyclecondition}, and in particular it has proper intersection with $\ini \infty(X \times \cubb {n-1})$ for every $i=1,\ldots, n$. 
We put $\partial_i^\infty V =(\ini \infty)^{-1}(V)$.
Since $\Vb_i$ is of codimension $1$ in $\Xb\times \cubbb ni$, the restriction to $\Vb_i$ of the projection \[X\times \cubbb ni\to X\times \cubb {n-1} \]   is surjective (see Lemma \ref{lem-cyclecondition}). 
For $n=1$, we have $\CXD 1 0=\emptyset$ since $\Xb$ is local and $D\not=\emptyset$ 
(cf. Definition \ref{def-CXD}(2)). This implies $\partial_1^\infty V=\emptyset$ and the desired assertion follows from Lemma \ref{lem.divisor-functionMC}.
For $n>1$ we proceed by induction on $n$ and assume that there exists $g\in \PAt {n-1}$ such that
\[
\partial_i^\infty V =\div(g(t_1,\dots,t_{i-1},t_{i+1},\dots,t_n)) \subset X\times \cubb {n-1}.
\]
%its image $\Wb_i$ in $X\times\cubb {n-1}$ is precisely 
%$\Xb\times\cubb {n-1}$ . %note that we are using a slightly different notation from  Lemma \ref{lem-cyclecondition}.
Then, by Lemma \ref{lem.divisor-functionMC}, the Modulus condition for $V$ is equivalent to the condition that $V=\div(f)$ with
$f\in A[t_1,\dots,t_n]$ of the form
%such that for all $1\leq i\leq n$, $f$ is written as
\[
f= \big(1+ \underset{1\leq \nu \leq m}{\sum}\; a_\nu (1-t_i)^{\nu}\big)\cdot g^N,
\]
where $N\geq 0$ is some integer and 
\[
a_\nu\in \Gamma(X\times \cubb {n-1}-\partial_i^\infty V,I^\nu\cO)=
I^\nu\cdot A[t_1,\dots,\overset{\vee}{t_i},\dots,t_n][g^{-1}] \quad \text{ for }\nu \geq 1.
\] 
It is easy to see that this implies $f\in \PAt n$.
%the last condition is equivalent to that $f$ is admissible for $\pi\in A$.

Conversely assume $V=\div(f)$ for $f\in \PAt n$. It is easy to see that
$V$ and $\partial_i^\infty V$ satisfy the conditions of Lemma \ref{lem-cyclecondition} and Lemma \ref{lem.divisor-functionMC}.
%By \ref{lem-cyclecondition} and \ref{lem-cyclecondition0.rem},
%$V$ satisfies the condition \ref{lem-cyclecondition0} (3)'.
One can also check that none of the components of $\Vb\cap (D\times\Pp n)$ is contained in 
$\underset{1\leq i\leq n}{\cap}\ X\times \Fni$. 
Hence, by Lemma \ref{lem-cyclecondition0}, $V$ satisfies the condition (3) of \ref{def-CXD}. The good position condition with respect to the faces is clear for every cycle of the form $\div(f)$ and we conclude that $V\in \zzXD 1 n$.
\end{proof}

\begin{coro}\label{cor-cyclecondition1}
We keep the notations of \ref{Not-RelCod1-1} and \ref{Not-RelCod1-2}.
Let $\zzXDeff r n\subset \zzXD rn$ be the submonoid of effective relative cycles and let $\PA n = (\PAt n)/A^\times$. 
For $n\geq 1$ there is an isomorphism of monoids
\[
V: {\PA n }\isom \zzXDeff 1n\;;\; f \to V(f):=\div(f),
\]
and an isomorphism of groups
\[
V: {\PAgr n} \isom \zzXD 1n\;;\; f/g \to V(f)-V(g),
\]
where 
\[
\PAgr n =\{f/g\;|\; f,g\in \PA n\}.%\subset \Frac(A[t_1,\dots,t_n]).
\]
%and $\PA n = (\PAt n)/A^\times$. 
\end{coro}
\subsubsection{}\label{cubObjP}
We follow the notation of \cite[\S1.1]{KL}). The assignment
\[
\ul n \mapsto \PA n\quad (\ul n =\{0,\infty\}^n,\; n=0,1,2,3...)
\]
defines an extended cubical object of monoids (see \cite[1.5]{L}) in the following way. For the inclusions 
$\eta_{n,i,\ep}:\ul {n-1} \to \ul {n}$ ($\ep=0,\infty,\; i=1,...,n$), we define boundary maps 
\begin{equation}
\eta_{n,i,\ep}^*: A[t_1,\dots,t_{n}] \to A[t_1,\dots,t_{n-1}] \text{ for } \ep \in \{0, \infty\}
\end{equation}
by
\[
\begin{aligned}
 \eta_{n,i,0}^*(f(t_1,\dots,t_{n})) &= f(t_1,\dots,t_{i-1},0,t_i,\dots, t_{n-1}),\\
 \eta_{n,i,\infty}^*(f(t_1,\dots,t_{n}))&=f(t_1,\dots,t_{i-1},1,t_i,\dots, t_{n-1}).
\end{aligned}
\]
For projections $pr_i: \ul n \to \ul {n-1}$ ($i=1,...,n$), we define 
\[
pr_{n,i}^*: A[t_1,\dots,t_{n-1}] \to A[t_1,\dots,t_{n}]
\] 
by
\[
pr_{n,i}^*(f(t_1,\dots,t_{n-1}))= f(t_1,\dots,t_{i-1},t_{i+1},\dots, t_{n}).
\]
They all induce corresponding maps on $\PA n$, denoted with the same letters.
Permutation of factors are defined in an obvious way and involutions $\tau_{n,i}^*$ are defined as the maps $\PA n \to \PA n$ induced by $t_i \to 1-t_i$.
For the multiplications
\[
\mu:\{0,\infty\}^2 \to \{0,\infty\}\;;\; 
\mu(\infty,\infty)=\infty; \mu(a,b)=0\text{  for } (a,b)\not= (\infty,\infty),
\] 
we define $\mu^*:\PA 1 \to \PA 2$ as the map induced by $1-t \to (1-t_1)(1-t_2)$.
The isomorphisms in Corollary \ref{cor-cyclecondition1} are compatible with cubical structure.
\begin{theo}\label{thm.codim1}
Under the above assumptions, we have $\CHXD 1 n =0$ for $n\not=1$ and there is a natural isomorphism
\[
\delta: \CHXD 1 1 \isom (1+I)^\times,
\]
where $(1+I)^\times=(1+I)\cap A^\times$ viewed as a subgroup of $A^\times$.
\end{theo}
\begin{proof}
\def\ul#1{\underline{#1}}
The vanishing of $\CHXD 1 n $ for $n=0$ is a direct consequence of the definition, since $X$ is local.
In what follows we show the assertion for $n\geq 1$. 
By \ref{cubObjP} we are reduced to compute the homotopy groups $H_i(\PAgr \bullet)$ of the cubical objects of abelian group
\[
\ul n \to {\PAgr n}.
\]
Recall that these are homology groups of the complex
\begin{align*}
\cdots \rmapo{\partial} \PAgr n/\PAgrdeg n \rmapo{\partial}  \PAgr {n-1}/\PAgrdeg {n-1}
\rmapo{\partial} \\
\cdots \rmapo{\partial} \PAgr {1}/\PAgrdeg {1},
\end{align*}
where
\[
\PAgrdeg n=\sum_{i=1}^{n} pr_{n,i}^*(\PAgr {n-1}),\quad
 \partial = \sum_{i=1}^{n} (-1)^i (\eta_{n,i,0}^*-\eta_{n,i,\infty}^*).
\]
%We note that this is an extended cubical object in the sense of \cite[\S1.5]{L}.
Let
\[
\NPAgr n= \underset{2\leq i\leq n}{\cap}\Ker(\eta_{n,i,0}^*) \cap 
\underset{1\leq i\leq n}{\cap}\Ker(\eta_{n,i,\infty}^*)
\]
and consider the complex
\[
\cdots \rmapo{\eta_{n+1,1,0}^*} \NPAgr n \rmapo{\eta_{n,1,0}^*} \NPAgr {n-1} \rmapo{\eta_{n-1,1,0}^*} 
\cdots\to \NPAgr 1.
\]
By \cite[Lemma 1.6]{L}, we have a natural isomorphism
\[
H_i(\NPAgr \bullet)\isom H_i(\PAgr \bullet / \PAgrdeg \bullet)
\]
and we are reduced to show the existence of isomorphisms
\begin{equation}\label{eq.NPgr}
H_i(\NPAgr \bullet)\simeq 
\left.\left\{\begin{gathered}
 1+I\; \\ 
 0 \\
\end{gathered}\right.\quad
\begin{aligned}
&\text{if $n=1$,}\\
&\text{if $n>1$.}
\end{aligned}\right.
\end{equation}
Consider
\[
H : \Frac(A[t_1,\dots,t_n]) \to \Frac(A[t_1,\dots,t_{n+1}])
\]
defined by
\[
H(f(t_1,\dots,t_{n}))= 1+\big(f(\underline{1})^{-1}f(t_2,\dots,t_{n+1})-1\big)(1-t_{1}),
\]
where $(\underline{1}) = (1, \ldots, 1)$. One easily checks that this induces the maps (of sets)
\[
H :\PA n \to \PA {n+1},\quad H :\PAgr n \to \PAgr {n+1}
\]
and we have, for $\phi\in \PAgr n$
\begin{equation}\label{eq1-1}
\eta_{n+1,i,\ep}^*(H(\phi))=
\left.\left\{\begin{gathered}
 H(\eta_{n,i-1,\ep}^*(\phi))\; \\ 
 1 \\
 \phi \pmod{A^\times} \;\\
\end{gathered}\right.\quad
\begin{aligned}
&\text{if $2\leq i\leq n+1$,}\\
&\text{if $i=1$ and $\epsilon=\infty$,}\\
&\text{if $i=1$ and $\epsilon=0$.}
\end{aligned}\right.
\end{equation}
Hence $H$ induces a map $\NPAgr n \to \NPAgr {n+1}$, and if $n>1$ we have
\[
\eta_{n+1,1,0}^*(H(\phi)) =\phi \qfor \phi\in \Ker(\NPAgr n\rmapo{\eta_{n,1,0}^*} \NPAgr {n-1}).
\]
This proves \eqref{eq.NPgr} for $n>1$. To show \eqref{eq.NPgr} for $n=1$, 
we define a map
\[
\delta: \PAgr 1 \to (1+I)^\times\;;\; f/g \to f(0)g(1)/g(0)f(1)\quad (f,g\in \PA 1).
\]
It is easy to see that this is a well-defined group homomorphism and that 
\begin{equation}\label{NPAGRexact}
\NPAgr 2\rmapo{\eta_{2,1,0}^*} \NPAgr {1} \rmapo{\delta} 1+I
\end{equation}
is a complex (note that $\NPAgr 1 =\PAgr 1$).
To show that it is exact, we compute the boundary for $f\in \PA 1$
\[
\eta_{2,i,\ep}^*(H(f))=
\left.\left\{\begin{gathered}
  1 \\
  1+(\frac{f(0)}{f(1)} - 1)(1-t_1) \;\\
  f \pmod{A^\times} \;\\
\end{gathered}\right.\quad
\begin{aligned}
&\text{if $\epsilon=\infty$,}\\
&\text{if $i=2$ and $\epsilon=0$,}\\
&\text{if $i=1$ and $\epsilon=0$.}
\end{aligned}\right.
\]
Hence, for $f,g\in \PA 1$ with $f(0)/f(1)=g(0)/g(1)$, we have
\[
H(f)/H(g)\in \NPAgr 2\qaq \eta_{2,1,0}^*(H(f)/H(g))= f/g.
\]
This proves the exactness of \eqref{NPAGRexact} and completes the proof of Theorem \ref{thm.codim1}.
\end{proof}

\section{Fundamental class in relative differentials}\label{clOmega}
\def\Yp{Y^{\prime}}
\def\Fp{F^{\prime}}
\def\Dp{D^{\prime}}
\def\YFDp{(\Yp, \Fp, \Dp)}

In \cite{ElZ}, El Zein gave an explicit construction of Grothendieck's ``fundamental class'' \cite{Gr58} of a cycle on a smooth scheme $Y/k$ in Hodge cohomology, defining a morphims from the Chow ring of $Y$ to $\H^*(Y, \Omega^{*}_{Y/k})$. It turns out that this approach can be partially followed and extended to construct a relative version of El Zein's fundamental class.

In this section, we consider an integral scheme $Y$ of pure dimension $n$, smooth and separated over a field $k$. We write  $\Omega^1_{Y/k}$ for the Zariski sheaf of relative K\"ahler differentials on $Y$ and  $\Omega^1_{Y} = \Omega^1_{Y/\Z}$ for the sheaf of absolute differentials. % Note that $\Omega^1_{Y/k}$ is a coherent sheaf on $Y$, while $\Omega^1_Y$ is just quasi-coherent, in general.
For $r\geq0$, we let $C^r(Y)$ be the set of integral closed subschemes of codimension  $r$ on $Y$. %Let $k$ be a field. In this section $Y$ denotes a smooth variety over $k$ and $C^r(Y)$ denotes the set of integral closed subschemes 
%of codimension $r$ on $Y$.

\subsection{El Zein's fundamental class}\label{reviewELZ}

\subsubsection{}\label{reviewElZ}
For $W\in C^r(Y)$ let
%and let $\iota: W\to Y$ be the inclusion. 
%\[\ext^r_{\cO_Y}(\iota_* \cO_W, \Omega^r_{Y/k})\xrightarrow{\sim} \H^r_W(Y, \Omega^r_{Y/k}),\]
 \[{\clOmega r W}_k \in \H^r_W(Y, \Omega^r_{Y/k})\]
%= \indlim n\ext^r_{\cO_Y}(\cO_{W_n}, \Omega^r_{Y/k}) ,\]
be the fundamental class of $W$ constructed in \cite{ElZ}.
Using the technique of \cite[A.6]{KSY}, one can construct a cycle class 
\[\clOmega r W \in \H^r_W(Y,\WW r Y)\]
in the absolute Hodge cohomology group with support starting from ${\clOmega r W}_k$. We quickly recall the argument. Let $k_0\subset k$ be the prime subfield of $k$ and let $\cI$ be the  set of smooth $k_0$-subalgebras of $k$, partially ordered by inclusion: it is a filtered set. We fix a $k_0$-algebra $A$ and a smooth separated $A$-scheme $Y_A$ together with a closed integral subscheme $W_A$, flat over $A$, such that
 \begin{align*} Y_A \tensor_A k = Y \text{ and } W_A\tensor_A k = W.
	 \end{align*}
	 For every $B\in \cI$ containing $A$, we write $Y_B$ (resp. $W_B$) for the base change $Y_A\tensor_A B$ (resp. $W_A\tensor_A B$). A \v{C}ech cohomology computation shows  then that
	 \[\H^r_W(Y, \Omega^r_{Y}) =\H^r_W(Y, \Omega^r_{Y/k_0}) \xrightarrow{\sim} \indlim {B\in\cI} \H^r_{W_B}(Y_B, \Omega^r_{Y_B/k_0}).\]
 Note that, by construction, $\codim_Y W =\codim_{Y_B} W_B$ for every $B\in \cI$ containing $A$.
We then define 
\[ \clOmega r W = \indlim {B\in\cI} {\clOmega r {W_B}}_{k_0}.\] 
 
\begin{rem}\label{rem-CC1}One can show (see again \cite[Thm. 3.1]{ElZ}) that ${\clOmega r W}_k$ lies in the image of $\H^r_W(Y,\Omegacl r {Y/k})$, where
$\Omegacl r {Y/k}\subset \WW r {Y/k}$ is the subsheaf of closed forms.
This implies that $\clOmega r W$ lies in the image of $\H^r_W(Y,\Omegacl r {Y})$
\end{rem}
 
We now give a description of $\clOmega r W$ by an explicit \v{C}ech cocycle.
Let $V\subset Y$ be an affine open neighborhood of the generic point of $W$ such that there exist
a regular sequence $f_1,\dots f_r\in \Gamma(V,\cO)$ which defines $W\cap V$ in $V$. Let $\cU$ be the covering of $V\setminus W$ given by the open subsets $\{U_i = D(f_i)\}_{i=1,\ldots, r}$ and write $\check{C}^{\bullet}(\cU, \Omega^{r}_{V})$ for the \v{C}ech complex of $ \Omega^{r}_{V}$ with respect to the covering $\cU$. Then the cohomology class of the \v{C}ech cocyle
\begin{equation}\label{affineDescrClass}
%\label{clOmega-eq0}
\dlog {f_1}\wedge \cdots \wedge \dlog {f_r} = \frac{{d}f_1}{f_1}\wedge \cdots \wedge  \frac{{d}f_r}{f_r}\in \H^0(U_1\cap \cdots U_r,\WW r {V})
\end{equation}
gives an element in $ \H^{r-1}(V\backslash W,\WW r {V})$  that maps to the class ${\clOmega r W}_{|V}$ in $\H^r_{W\cap V}(V,\WW r {V})$ via the boundary morphism
(This is a consequence of the Trace formula, see \cite[A.2]{Conr}, in particular Lemma A.2.1, and \cite[p.37]{ElZ}).
By the localization exact sequence and \eqref{clOmega-eq1} below the restriction map 
\begin{equation}\label{clOmega-eq3}
\H^r_{W}(Y,\WW r Y)\hookrightarrow \H^r_{W\cap V}(V,\WW r V)
\end{equation}
is injective. Hence the affine description \eqref{affineDescrClass} characterizes $\clOmega r W$ (as well as ${\clOmega r W}_k$).

%We have the fundamental class associated to $W\in C^r(Y)$:
%\begin{equation}\label{clOmega-eq-1}
%\clOmega r W \in H^r_W(Y,\WW r Y),
%\end{equation}
%where $\WW r Y$ is the sheaf of absolute K\"ahler differentials. It satisfies the following conditions: 
%\begin{itemize}
%\item[$(CC1)$]
%$\clOmega r W$ lies in the image of $H^r_W(Y,\Omegacl r Y)$, where
%$\Omegacl r Y\subset \WW r Y$ is the subsheaf of closed forms.
%\item[$(CC2)$]
%Let $V\subset Y$ be an affine open neighborhood of the generic point $\eta$ of $W$ and 
%$f_1,\dots,f_r\in \Gamma(V,\cO)$ be a regular sequence which define $W\cap V$ in $V$.
%Let $U_i\subset V$ be the complement of the support of the divisor of $f_i$.
%Then $\{U_i\}_{1\leq i\leq r}$ is an open covering of $V$ and 
%\[
%{\clOmega r W}_{|V}\in   H^r_{W\cap V}(V,\WW r V)
%\simeq H^{r-1}(V\backslash W,\WW r V)
%\]
%is given by the $\v{C}$ech cocyle
%\begin{equation*}
%\label{clOmega-eq0}
%\dlog {f_1}\wedge \cdots \wedge \dlog {f_r}\in H^0(U_1\cap \cdots U_r,\WW r V).
%\end{equation*}
%\end{itemize}

\begin{lem}\label{clOmega-lem-Vanishing}For a closed subscheme $T\subset Y$ of pure codimension $a$, 
\begin{equation}\label{clOmega-eq1}
\H^q_{T}(Y,\WW j Y) =0\qfor q<a.
\end{equation}\end{lem}
\begin{proof}
Arguing as in \ref{reviewElZ} we may replace $\WW j Y$ by $\WWk j Y$ which is a locally free $\cO_Y$-module. 
By the localization sequence one is reduced to the case where $Y$ is local and $T$ is the closed point.
Then the assertion follows from the fact that a regular local ring is Gorenstein.
\end{proof}

%This follows from \cite[Th.6.3]{Grot} and the fact that there is a filtration $F^a\WW r Y \subset \WW r Y$ such that
%\[
%F^a\WW r Y/F^{a+1}\WW r Y \simeq \WWk {r-a} Y\otimes_k \WW a k.
%\]

\subsection{Relative version of El Zein's fundamental class}\label{relELZ}
%In what follows we give a refinement of the fundamental class in a relative setting with modulus (see Theorem \ref{thm-purityOmega}).

\subsubsection{}\label{Setting_YFD}Let $Y$ be again a smooth variety over $k$. We fix now  a (reduced) simple normal crossing divisor $F$  and an effective Cartier divisor $D$ on $Y$.  In what follows, we will assume that $F$ and $D$ satisfy the following condition:
\begin{enumerate}
\item[$(\bigstar)$]
There is no common component of $D$ and $F$, and 
$D_{red}+F$ is a (reduced) simple normal crossing divisor on $Y$
\end{enumerate}

Write
\[
X=Y-(F+D)\;{\hookrightarrow}\; Y-F \;{\hookrightarrow}\; Y 
\]
for the open complement of $F+D$ in $Y$ and $\iota_X$ for the open immersion $X{\hookrightarrow}\; Y$.
\begin{rem}In section \ref{cyclemapDR}, we will work in a situation where $(X,Y-F,Y)=(X_n,\Xb_n,Y_n)$ with
\[
X_n = X\times \cubb n {\hookrightarrow} \Xb_n=\Xb \times \cubb n 
{\hookrightarrow} Y_n =\Xb \times \Pp n \supset D_n=D \times \Pp n
\] 
where $X\subset \Xb\supset D$ are as in \S\ref{cyclecomplex} and $\Xb$ is smooth over $k$ and 
the reduced part of $D$ is a simple normal crossing divisor.
\end{rem}
\begin{defi}\label{def-CYFD}[see Definition \ref{def-CXD}]
Let $X,Y, F, D$ be as above and let $W$ be an integral closed subscheme of codimension $r$ on $X$. Let $\Wb$ be the closure of $W$ in $Y$, $\Wb^N$ its normalization and $\phiW:\Wb^N\to Y$  the natural map. We say that $W$ {\it satisfies the modulus condition (with respect to the divisor $D$ and the face $F$)} if the following inequality as Cartier divisors on $\Wb^N$ holds
% for some $1\leq i\leq n$:
\begin{equation}\label{moduluscond}
\phiW^*(D)\leq \phiW^*(F).
\end{equation}
We denote by $\CYFDr$  the set of integral closed subschemes $W$ of codimension $r$ on $X$ that do satisfy the modulus condition.
%which satisfies the following \textit{modulus} condition:
%\begin{enumerate}
%\item[$(\sharp)$]
%Let $\Wb$ be the closure of $W$ in $Y$ and $\Wb^N$ be its normalization and $\phiW:\Wb^N\to Y$ be 
%the natural map. The following inequality as Cartier divisors on $\Wb^N$ holds:
% for some $1\leq i\leq n$:
%\begin{equation}\label{moduluscond}
%\phiW^*(D)\leq \phiW^*(F).
%\end{equation}
%\end{enumerate}

Note that the condition implies that $\Wb \cap (Y-F) \cap D=\emptyset$ and that $W$ is closed in $Y-F$. 
%$\Wb \cap (Y-F)\cap D=\emptyset$ and 
\end{defi}
%Let $\WW 1 Y(\log F+D)$ be the sheaf of logarithimic differentials for the standard log structure on $Y$ defined by the divisor $D_{red}+F$ (see \cite[1.5-1.7]{K}): it is the subsheaf of $(\iota_{X})_*\Omega^1_{X}$ of differential forms with logarithmic poles along $D_{red}+F$. We write $\WW r Y(\log F+D)$ for its $r$-th external product and set
\begin{defi}\label{defi-log-diff} Let $\WW 1 Y(\log F+D)$ be the sheaf of absolute differential forms with logarithmic poles along $D_{red}+F$. We write $\WW r Y(\log F+D)$ for its $r$-th external product and set
\[
 \WYFD r =\WW r Y(\log F+D)\otimes_{\cO_Y} \cO_Y(-D). % \WW r Y(\log F+D)(-D)
\]
\end{defi}
By the same argument used to prove Lemma \ref{clOmega-lem-Vanishing}, we have the following
\begin{lem}\label{clOmega-lem-Vanishing-Relative}For a closed subscheme $T\subset Y$, 
	\begin{equation}\label{clOmega-eq4}
	\H^q_{T}(Y,\WYFD r) =0\qfor q<\codim_Y(T).
	\end{equation}
\end{lem}

\begin{rem}Let $W\in C^r(X)$ and $\Wb\subset Y$ be its closure.
The long exact localization sequence, together with \eqref{clOmega-eq4}, implies the injection
\begin{equation}\label{clOmega-eq5}
\H^r_{\Wb}(Y,\WYFD r)\hookrightarrow \H^r_{W}(X,\WW r X).
\end{equation}
 \end{rem}
 We now come to the main result of this section:
\begin{theo}\label{thm-purityOmega}
For $W\in \CYFDr$ there is an element
\[
 \clOmega r W \in \H^r_{\Wb}(Y,\WYFD r)
\]
which maps to the fundamental class $\clOmega r W \in \H^r_{W}(X,\WW r X)$ under the map \eqref{clOmega-eq5}.
%(Note that by the injectivity of \eqref{clOmega-eq5} such a class is uniquely determined).
\end{theo}
The proof is divided into several steps. We start with the following reduction
\begin{claim}\label{thm-purityOmega-claim}
Let $F_1,\dots,F_n$ be the irreducible components of $F$ and let $Z$ be the reduced part of $\Wb\times_Y F$. %Let $Z$ be the reduced part of $\Wb\times_Y F$. 
We may assume the following conditions:
\begin{enumerate}
\item[$(\clubsuit 1)$]
$Z$ is irreducible of pure codimension $r+1$ in $Y$,
% and the reduce part $\Zred$ of $Z$ is regular,
\item[$(\clubsuit 2)$]
$Y=\Spec(A)$ is affine equipped with
$\pi\in A$ and $s_i\in A$ with $1\leq i\leq n$ such that $D=\Spec(A/(\pi))$ and $F_i=\Spec(A/(s_i))$. 
% For any component $G$ of $D+F$, either $Z\subset G$ or $Z\cap G=\emptyset$ (note $Z\subset F_i$ by the assumption).
\end{enumerate}
\end{claim}
\begin{proof}
Lemma \ref{clOmega-lem-Vanishing-Relative} together with the localization sequence implies that we have
\begin{equation}\label{purityOmega-eq6}
 \H^r_{\Wb}(Y,\WYFD r)\xrightarrow{\sim} \H^r_{\Wb-T}(Y-T,{\WYFD r}_{|Y-T})
\end{equation}
for every closed subscheme $T\subset \Wb$ of codimension strictly larger then $r+1$ in $Y$. Therefore we can disregard the irreducible components $Z_i$ of $Z$ of with $\codim_Y(Z_i)>r+1$ and, by shrinking $Y$ around the generic points of $Z$ of codimension $r+1$ in $Y$, we can 
%Hence we may replace $Y$ by an open neighbourhood of the generic points of $Z$ of codimension $r+1$ in $Y$
 assume the conditions of Claim \ref{thm-purityOmega-claim} except for
the irreducibility of $Z$. 

This last reduction can be shown as follows: take a finite open covering $Y=\underset{i\in I}{\cup} U_i$ such that
each $U_i$ contains at most one irreducible component of $Z$. 
Fixing a total order on $I$, let $I^{(a)}$ for $a\in \bZ_{\geq 0}$ be the set of tuples $\alpha=(i_0,\dots,i_a)$ in $I$
with $i_0<\cdots< i_a$. For $(i_0,\dots,i_a)\in I^{(a)}$, put
$U_\alpha=U_{i_0}\cap \cdots \cap U_{i_a}$. We have the Mayer-Vietoris spectral sequence associated to the the covering $\underset{i\in I}{\cup} U_i$
\begin{equation}\label{purityOmega-eq6Sp}
E_1^{a,b} =\underset{\alpha\in I^{(a)}}{\bigoplus}\; \H^b_{\Wb\cap U_\alpha}(U_\alpha,{\WYFD r}_{|U_\alpha}) \; \Rightarrow\;
\H^{a+b}_{\Wb}(Y,\WYFD r).
\end{equation}
Putting $V_i=U_i\cap X$ we have the induced covering $X= \underset{i\in I}{\cup} V_i$ and the analogue of \eqref{purityOmega-eq6Sp}
\begin{equation}\label{purityOmega-eq6Sp2}
E_1^{a,b} =\underset{\alpha\in I^{(a)}}{\bigoplus}\; \H^b_{W\cap V_\alpha}(V_\alpha,{\WW r X}_{|V_\alpha}) \; \Rightarrow\;
\H^{a+b}_{W}(X,\WW r X).
\end{equation}
By \eqref{clOmega-eq4}, \eqref{purityOmega-eq6Sp} gives rise to an exact sequence
\[
0 \to \H^r_{\Wb}(Y,\WYFD r) \to \underset{i\in I}{\bigoplus}  \H^r_{\Wb\cap U_i}(U_i,\WYFD r)
\to \underset{(i,j)\in I^{(1)}}{\bigoplus}  \H^r_{\Wb\cap U_i\cap U_j}(U_i\cap U_j,\WYFD r)
\]
and \eqref{purityOmega-eq6Sp2} gives the similar exact sequence for $\H^r_{W}(X,\WW r X)$. We can therefore replace $Y$ by $U_i$ and assume that $Z$ is irreducible.
\end{proof}

\subsection{The case $\Wb$ is a normal variety}\label{NormalCase} We first prove the following
\begin{lem}\label{purityOmega-lem1}
Let $I_{\Wb}\subset A$ be the ideal of definition for $\Wb\subset Y=\Spec(A)$. 
There exist $f\in I_{\Wb}$ and $a\in A$ such that $f=s_i+\pi a$ for some $1\leq i\leq n$.
\end{lem}
\begin{proof}
Indeed, the modulus condition \eqref{moduluscond} implies that $\Wb\times_X D\subset \Wb\times_X F \subset F$.
Since $\Wb\times_X F$ is assumed to be irreducible, $\Wb\times_X D\subset F_i$ for some $1\leq i\leq n$ 
so that $s_i\in I_{\Wb}+(\pi)$. % This proves the lemma.
\end{proof}
\medbreak

Since $Z$ is contained in $F_i$ by the proof of Lemma \ref{purityOmega-lem1},
$f=s_i+\pi a$ is a regular parameter of the local rings of $Y$ at points of $Z$.
By the normality of $\Wb$, $\Wb$ is regular at the generic point of $Z$ and 
we may assume by \eqref{purityOmega-eq6} that, after shrinking $Y$ around the generic point of $Z$, 
we can complete $f$ to  a regular sequence $f_1=f,f_2,\dots,f_r$ in $A$ such that $I_{\Wb}=(f_1,f_2,\dots,f_r)$.
Put $U_j=\Spec(A[1/f_i])$ for $1\leq j\leq r$.
By the local description \eqref{affineDescrClass}, we have that $\clOmega r W \in \H^r_{W}(X,\WW r X)$ is given by 
the cohomology class of the \v{C}ech cocyle
\[%\label{purityOmega-eq6.5}
\omega^\prime=\dlog {f}\wedge \dlog{f_2}\wedge\cdots\wedge \dlog {f_r}\in
\H^0(X\cap U_1\cap \cdots\cap  U_r,\WW r X)
= \H^0(X\setminus W,  {\Omega^r_{X}}_{|_{X\setminus W}}).
\]
%The commutativity of the diagram
%\[\xymatrix{ \H^0(X\setminus W,\WW r X ) \ar[r]^{\partial} \ar[rd]^{\partial} & \H^r_{(\Wb \cup F_i)}\ar[d]^{res}\\
%&\H^r_{\Wb}(X,\WW r X )
%}
%\]
Since the cohomology class of $\dlog s\wedge \dlog {f_2}\wedge\cdots\wedge \dlog {f_r}$ vanishes, we see that $\clOmega r W \in \H^r_{W}(X,\WW r X)$ can be also represented by the cocycle  
\begin{equation}\label{purityOmega-eq6.5}
\omega=\dlog\; \frac{f}{s}\wedge \dlog\; {f_2}\wedge\cdots\wedge \dlog\; {f_r}\;\in
\H^0(X\cap U_1\cap \cdots\cap  U_r,\WW r X).
\end{equation}
%It suffices then to show that $\omega$ is a restriction of an element of $\H^0(U_1\cap \cdots\cap  U_r,\WYFD r)$, which is true by the local description of $f$ given by Lemma \ref{purityOmega-lem1}.% implies

It suffices then to show that $\dlog \frac{f}{s}$ is a restriction of an element of 
$\H^0(U_1\cap \cdots\cap  U_r,\Omega^1_{Y|D}(\log F))$, which is true by the local description of $f$ given by Lemma \ref{purityOmega-lem1}.
%\[
%\dlog\; {\frac{f}{s_i}}=\dlog {\Big(1+\frac{a}{s_i}\pi\Big)}=
%\frac{\pi}{f}\Big(-a\dlog\; {s_i} + da + a \dlog\; \pi\Big).
%\]
%which proves the desired assertion.
%\bigskip

\subsection{The case of an arbitrary $\Wb$}%We now treat the general case where $\Wb$ is not necessarily normal.
Let $\phi_{\Wb}\colon\WbN \to \Wb$ be the normalization morphism.
Since it is finite, there exist an integer $M$ and a closed immersion
\[{i_{\WbN}}\colon \WbN \hookrightarrow \bP^M_Y=Y\times \bP^M\]
 which fits into the commutative square
\[
\xymatrix{
\WbN \ar[r]^{i_{\WbN}} \ar[d]_{\phi_{\Wb}} & \bP^M_Y \ar[d]^{p_{Y}}\\
\Wb \ar[r]^{i_{\Wb}} & Y\\
}\]
where $p_{Y}$ is the projection (this is an idea due to Bloch, taken from \cite[Appendix]{EWB}).
Noting $\phiW=i_{\Wb}\circ p_{\Wb}=p_Y\circ i_{\WbN}$, the modulus condition \eqref{moduluscond} implies
\[
\WbN\cap \bP^M_X \in C^{r+M}(\bP^M_Y,\bP^M_F,\bP^M_D) \quad \text{(cf. Definition \ref{def-CYFD}).}
\]
By the normal case \ref{NormalCase}, the fundamental class
\[
\clOmega {r+M} {\WbN\cap \bP^M_X}\in \H^{r+M}_{\WbN\cap \bP^M_X}(\bP^M_X,\Omega^{r+M}_{\bP^M_X})
\]
arises from an element of 
$\H^{r+M}_{\WbN}(\bP^M_Y,\Omega^{r+M}_{\bP^M_Y|\bP^M_D}(\log \bP^M_F))$.
Now Theorem \ref{thm-purityOmega} follows from the commutativity of the following diagram
\begin{equation}\label{purityOmega-eq7}
\xymatrix{
\H^{r+M}_{\WbN}(\bP^M_Y,\Omega^{r+M}_{\bP^M_Y|\bP^M_D}(\log \bP^M_F)) \ar@{^{(}->}[r]^-{\eqref{clOmega-eq5}} \ar[d]_{(p_Y)_*}
& \H^{r+M}_{\WbN\cap \bP^M_X}(\bP^M_X,\Omega^{r+M}_{\bP^M_X}) \ar[d]^{(p_X)_*} \\
\H^{r}_{\Wb}(Y,\WYFD r)  \ar@{^{(}->}[r]^-{\eqref{clOmega-eq5}} & \H^{r}_{W}(X,\WW r X)  \\
}
\end{equation}
and from the fact that  $(p_X)_*(\clOmega {r+M} {\WbN\cap \bP^M_X})=\clOmega {r} {W}\in \H^{r}_{W}(X,\WW r X)$, which follows from the compatibility with proper push forward of El Zein's fundamental class (see \cite[III.3.2]{ElZ} and Section \ref{ProperpushclOmega} below).
Here the vertical maps are induced by the trace maps 
\begin{equation}\label{eq.trace}
\bR(p_X)_*\Omega^{r+M}_{\bP^M_X} \to \WW r X \qaq
\bR(p_Y)_*\Omega^{r+M}_{\bP^M_Y|\bP^M_D}(\log \bP^M_F) \to \WYFD r.
\end{equation}
which come from Lemma \ref{pushforDerCat} below where one takes $\YFDp=(\bP^M_Y,\bP^M_F,\bP^M_D)$.
%Note that the commutativity of  \eqref{purityOmega-eq7} follows from the functoriality of the trace maps \eqref{eq.trace}, so that the only thing which is possibly left to be checked is the identity
%\[(p_X)_*(\clOmega {r+M} {\WbN\cap \bP^M_X})=\clOmega {r} {W}\in \H^{r}_{W}(X,\WW r X)\]

\subsection{Compatibility with proper push forward}\label{ProperpushclOmega}

Let $(Y,F,D)$ and $\YFDp$ be two triples satisfying the condition $(\bigstar)$ of \ref{Setting_YFD} and let $f\colon \Yp\to Y$ be a proper morphism. We say that $f$ is \emph{admissible} if the following condition holds:
\begin{enumerate}
\item[$(\clubsuit)$] The pullback of the Cartier divisors $F$ and $D$ along $f$ are defined and 
satisfy $f^*(F) = \Fp$ and $\Dp  \geq f^*(D)$ and $|f^{-1}(D)|=|D|$.
\end{enumerate}

{Note that the condition implies that $\Dp - \Dp_{red} \geq f^*(D-D_{red})$ and $X'=f^{-1}(X)$ so that 
the restriction $f_{|{X^\prime}}$ of $f$ to $X^\prime$ is proper.} %The proof of the following Lemma is standard.

\begin{lem}\label{pushforwCyc}Let $f\colon\YFDp\to(Y,F,D)$ be an admissible proper morphism between the triples $\YFDp$ and $(Y,F,D)$. Let $X=Y-(F+D)$ and $X^\prime= \Yp-(\Fp+\Dp)$. Then the proper pushforward of cycles by $f_{|{X^\prime}}$ induces a homomorphism
	\[f_*\colon C^{r+\dim \Yp - \dim Y}\YFDp\to C^r(Y,F,D)\quad r\geq 0.\]
 	\end{lem}
	The proof is a simple exercise using Lemma \ref{Containment_YFD} and left to the readers.
	%\begin{proof} Let $W$ be a closed integral subscheme of $X^\prime$ satisfying the modulus condition \eqref{moduluscond} with respect to $\Dp$ and $\Fp$. Let $f(W)$ denote the image of $W$ via $f_{|{X^\prime}}$, endowed with the structure of closed integral subscheme. 
%We need check $f(W)$ satisfies \eqref{moduluscond} with respect to $D$ and $F$. 
%We may suppose that $\dim W = \dim f(W)$. 
%Let $\overline{W}$ (resp. $ \overline{f(W)}$) denote the closure of $W$ in $\Yp$ (resp. $f(W)$ in $Y$) and let $\phi_{\overline{W}}\colon \overline{W}^N \to \Yp$ (resp. $\phi_{\overline{f(W)}}\colon \overline{f(W)}^N \to Y$) be the normalization morphism. Note that, by construction, $f(\overline{W}) = \overline{f(W)}$ intersects properly both $D$ and $F$. Then there exists a finite and surjective map $h:\overline{W}^N\to \overline{f(W)}^N$ making the diagram
%		\[\xymatrix{ \overline{W}^N  \ar[r]^{\phi_{\overline{W}}}\ar[d]^{h} & \Yp \ar[d]^f\\ 
%		\overline{f(W)}^N \ar[r]_{\phi_{\overline{f(W)}}} & Y
%		}
%		\]
%	commutative. 
%	The conditon $(\clubsuit)$ implies
%	\begin{align*}
%	h^*( \phi_{\overline{f(W)}}^*(D))  & =  \phi_{\overline{W}}^*( f^*(D))  \leq \phi_{\overline{W}}^*(D^\prime)\leq  \phi_{\overline{W}}^*(\Fp) = \phi^*_{\overline{W}}(f^*(F))=h^*(\phi_{\overline{f(W)}}^*(F)).
%	\end{align*}
%	By \cite[Lemma 2.2]{KP} (see also Lemma \ref{Containment_YFD}), this implies
%$ \phi_{\overline{f(W)}}^*(D) \leq \phi_{\overline{f(W)}}^*(F)$, completing the proof.
%	\end{proof}
	
\begin{rem}Noting that the pushforward defined at the level of cycles commutes with the boundary maps as in the case of Bloch's higher Chow groups, Lemma \ref{pushforwCyc} proves the covariant functoriality of Lemma \ref{lem-zXD}, giving a map of complexes
    \begin{equation*}
  f_* \colon \bZ(s)_{\Xb'|D'} \to \bZXDr\qwith s=r+\dim X' - \dim X.
    \end{equation*}
	\end{rem}
	
\subsubsection{} 

Let $g\colon X^\prime \to X$ be a proper morphism between smooth schemes over $k$.
Put $\delta=\dim X - \dim X^\prime$.
For integers $r,s\geq 0$ with $s-r= \delta$, we can construct a trace map
in the bounded derived category $D^b(X)$ of complexes of $\cO_X$-modules
\begin{equation}\label{eq.Trf} 
Tr_g: {R}g_* \Omega^{s}_{X^\prime}[\delta]\to \Omega^{r}_{X}.
\end{equation}
Indeed, arguing as in \ref{reviewElZ}, it suffices to construct it replacing the absolute differentials 
by the differentials over $k$. It follows then from \cite[VI, 4.2; VII, 2.1]{RD} or \cite[3.4]{Conr}. 

	\begin{lem}\label{pushforDerCat}Let $f\colon\YFDp\to(Y,F,D)$ be a morphism satisfying the condition ($\clubsuit$).
Let $g\colon X'\to X$ be the induced morphism and $\tau:X\to Y$ and $\tau^\prime:X^\prime \to Y^\prime$ be the open immersions. Then, for integers $r,s\geq 0$ with $s-r= \delta:=\dim Y - \dim Y^\prime$, 
there exists a natural map in $D^b(Y)$:
		\begin{equation}\label{eq.PushforDerCat} 
Tr_f: {R}f_* \Omega^{s}_{\Yp|\Dp}(\log \Fp)[\delta]\to \Omega^{r}_{Y|D}(\log F)
        \end{equation}
        which fits into the commutative diagram 
\[
\xymatrix{
{R}f_* \Omega^{s}_{\Yp|\Dp}(\log \Fp)[\delta] \ar[r] \ar[d]^{Tr_f} 
&  {R}f_* R\tau^\prime_*\Omega^{s}_{X^\prime}[\delta] \ar[r]^{\cong}
&  R\tau_*{R}g_* \Omega^{s}_{X^\prime}[\delta] \ar[ld]^{R\tau_*Tr_g} \\
\Omega^{r}_{Y|D}(\log F) \ar[r] & R\tau_*\Omega^{r}_{X}
}\]   
		% Moreover, the map \eqref{eq.PushforDerCat} is compatible with the exterior derivative $d$.
	\end{lem}
	\begin{proof}Let $\Sigma=D_{red}+F$ and $\Sigma'=D'_{red}+ F'$.
We are ought to construct a natural map
\begin{equation*}
Tr_f : Rf_*\Omega_{Y'}^s(\log \Sigma')(-D')[\delta]\to  \Omega_{Y}^r(\log \Sigma)(-D)
\end{equation*}
Let $D''=f^*(D-D_{red})+D'_{red}$. By ($\clubsuit$), we have $D''\leq D'$ and therefore it is enough to show the existence of a natural map
\begin{equation}\label{purityOmega-lem3.eq3}
Tr_f : Rf_*\Omega_{Y'}^s(\log \Sigma')(-D'')[\delta]\to  \Omega_{Y}^r(\log \Sigma)(-D).
\quad\text{in } D^b(Y).
\end{equation}
Arguing as in \ref{reviewElZ}, we may assume that the base field $k$ is finitely generated over its prime subfield $k_0$,
%As before, we may assume that $k$ is finitely generated over its prime subfield $k_0$ 
with $t=\trdeg_{k_0}(k)$ and put $d = t+ \dim Y$ and $d' = t+ \dim X^\prime$. 
We have isomorphisms %(cf. \eqref{eq1.Trf} and \eqref{eq2.Trf})
\begin{align}\label{purityOmega-lem3.eq3.4}
\Omega_{Y}^{r}(\log \Sigma)(-D)\xrightarrow{\simeq} 
\uhom_{D(Y)}(\Omega_{Y}^{d-r}(\log \Sigma)(D)),\Omega_{Y}^{d}(\Sigma)),
\end{align}
\begin{equation}\label{purityOmega-lem3.eq3.5}
\begin{aligned}
Rf_*\Omega_{Y'}^{s}&(\log \Sigma')(-D'')[\delta] \\
&\simeq Rf_*R\uhom_{D(Y')}(\Omega_{Y'}^{d-r}(\log \Sigma'),\Omega_{Y'}^{d'}(\Sigma'- D'')[\delta])\\
&\simeq Rf_*R\uhom_{D(Y')}(\Omega_{Y'}^{d-r}(\log \Sigma')(f^*D),\Omega_{Y'}^{d'}(\Sigma'+f^*D_{red}- D'_{red})[\delta])\\
&\simeq R\uhom_{D(Y)}(Rf_*(\Omega_{Y'}^{d-r}(\log \Sigma')(f^*D)),\Omega_{Y}^{d}(\Sigma))\\
\end{aligned}
\end{equation}
where the last isomorphism follows from the Verdier duality and the isomorphism 
\[
f^!(\Omega_{Y}^{d}(\Sigma))\cong f^!\Omega_{Y}^{d}(f^*\Sigma) 
\cong \Omega_{Y'}^{d'}(\Sigma'+f^*D_{red}- D'_{red})[\delta]
\]
using the assumption $F'=f^*F$.
By  adjunction, the natural map
\[
f^* \Omega_{Y}^{d-r}(\log \Sigma)(D) \to \Omega_{Y'}^{d-r}(\log \Sigma')(f^*D)
\]
induces
\[
\Omega_{Y}^{d-r}(\log \Sigma)(D) \to Rf_*(\Omega_{Y'}^{d-r}(\log \Sigma')(f^*D)),
\]
which induces the desired map \eqref{purityOmega-lem3.eq3} via \eqref{purityOmega-lem3.eq3.4} and \eqref{purityOmega-lem3.eq3.5}. 
%We have now to verify the compatibility of \ref{eq.PushforDerCat} with the exterior derivative $d$.
\end{proof}
\medbreak

Now take $W \in C^{s}\YFDp$. Under the assumption of Lemma \ref{pushforDerCat}, we have 
a commutative diagram
		\begin{equation}\label{eq.diagRestr}\xymatrix{ \H^{s}_{\overline{W}}(Y', \Omega^{s}_{\Yp|\Dp}(\log f^*F))  \ar@{^{(}->}[d]^{\eqref{clOmega-eq5}} \ar[r]^{f_*} & \H^r_{\overline{f(W)}}(Y, \WYFD r) \ar@{^{(}->}[d]^{\eqref{clOmega-eq5}} \\
		\H^{s}_{W}(X', \Omega^{s}_{X'} ) \ar[r]^{g_*} & \H^r_{f(W)}(X, \Omega^{r}_{X})
		}
		\end{equation}
		where $f_*$ (resp. $g_*$) are induced by $\Tr_f$ (resp. $Tr_g$).

\begin{lem}\label{pushforDerCat2}
Let 
\[
cl^{s}_{\Omega}(W/\Yp)\in \H^{s}_{\overline{W}}(Y', \Omega^{s}_{\Yp|\Dp}(\log f^*F)) \qaq
cl^r_{\Omega}(f_*(W)/Y) \in \H^r_{\overline{f(W)}}(Y, \WYFD r)
\]
 denote the fundamental classes of ${W}$ and $f_*([W])$ from Theorem \ref{thm-purityOmega}. %
%(the latter makes sense by Lemma \ref{pushforwCyc}).
Then we have
		\begin{equation}\label{eq.CompClOmega}cl^r_{\Omega}(f_*(W)/Y) = f_*cl^{s}_{\Omega}(W/\Yp) \in \H^r_{\overline{f(W)}}(Y, \WYFD r). \end{equation}
		\end{lem}
\begin{proof}
Since the relative fundamental classes restrict to El Zein's fundamental classes under the maps \eqref{clOmega-eq5}
in \eqref{eq.diagRestr}, the lemma follows from the compatibility of  El Zein's fundamental class for proper morphisms. % See \cite[Prop. 4.2.3]{Gr} for an explicit proof in the case of the crystalline fundamental class. The proof in {\it loc.~cit.}~can be easily adapted to our situation.
\end{proof}
%\newpage 

\section{Lemmas on cohomology of relative differentials}
\def\fn{\frak{n}}

%In this section we prove two lemmas which will be used in the construction of regulator maps to the relative de Rham cohomology and relative Deligne cohomology in \S\ref{clDR} and \S\ref{clDe}. 

\subsection{Independence of relative de Rham complex from the multiplicity of $D$}

Let $(Y,F,D)$ be as in \ref{Setting_YFD}, $X= Y-(F+D)$ and write $\Dred$ for the reduced part of $D$.
Let $D_1,\ldots, D_n$ be the irreducible components of $D$ and $e_i$ be the multiplicity of $D_i$ in $D$.
We have the relative de Rham complex
\begin{equation*}%\label{WXFD}
\WYFD \bullet \colon \cO_Y(-D)\rmapo d \WYFD 1  \rmapo d \WYFD 2 \to \cdots \to\WYFD r \to \cdots 
\end{equation*}
%The following lemma shows that, in some cases (notably in characteristic $0$), $\WYFD \bullet$ does not depend on $e_i$ for $i\in I:=\{1,\ldots, n\}$.

\begin{lem}\label{lemma-Dred-qiso} In the above setting, 
assume further that $e_i< p$ for all $i \in I$ if $p=\ch(k)>0$. Then the natural map
\[
\WYFD \bullet  = \WY \bullet (\log F + D)(-D) \to \WY \bullet (\log F + D)(-\Dred)
\]  
is a quasi-isomorphism {(see Definition \ref{defi-log-diff})}.
\end{lem}
\begin{proof}
%We endow $\bN^I$ with a semi-order as follows:
For $\fm= (m_i)_{i\in I}$ and $\fn = (n_i)_{i\in I}$ in $\bN^I$, we say that $\fm \leq \fn$ if $m_i\leq n_i$ for every $i \in I$. 
%\[
%(m_\lam)_{\lam\in \Lam}\leq (n_\lam)_{\lam\in \Lam}\; \Leftrightarrow\; m_\lam\leq n_\lam\qfor \forall \lam\in \Lam.
%\]
For every multi-index $\fm=(m_i)_{i \in I}\in \bN^I$, we set
\begin{equation}\label{eq-filtr-lemmDred}
\Dm=\underset{i \in I}{\sum} m_i D_i \qaq \Im=\cO_Y(-\Dm).
\end{equation}
Let $\fm_{max} = (e_1, \ldots, e_n)$.  For $\fm\in \bN^I$ with $\fm_{max}\geq \fm\geq (1,\dots,1)$ we have a filtration %The assignment \eqref{eq-filtr-lemmDred} gives rise to a filtration on the de Rham complex indexed by $\fm\in \bN^I$ with $\fm_{max}\geq \fm\geq (1,\dots,1)$:
\[
\WY \bullet (\log F + D)(-\Dred)\supset  \WY \bullet (\log F + D)(-\Dm).
\]
Fix now $\nu\in I$, an integer $q\geq 0$ and $\fm = (m_1,\ldots, m_n) \in \bN^I$. We define a sheaf $\Wmn q$ on $Y_{Zar}$ as
\[
\Wmn q =\WY \bullet (\log F + D)(-\Dm)/\WY \bullet (\log F + D)(-D_{\fm+\delta_\nu})  ( =  \Im\otimes_{\cO_Y} \Omega^q_Y(\log F +D)_{|D_\nu} ),
\]
where $\delta_\nu$ denotes the multi-index $(\delta_{i}^\nu)$ with $\delta_{\nu}^\nu = 1$ and $\delta_i^\nu=0$ for $i\neq \nu$.
The exterior derivative on $\WY \bullet (\log F + D)$ induces  a map
\[
\dmn q:\Wmn q \to \Wmn {q+1}
\]
%locally defined by 
%\[\prod_{i =1}^n \pi_i^{m_i} \otimes \omega \mapsto \prod_{i=1}^n \pi_i^{m_i} \otimes \left(d\omega + \sum_{i=1}^n m_i\cdot\dlog(\pi_i) \wedge \omega \right),
%\]
%where $\pi_i\in \cO_Y$ denotes a local uniformizer of $D_i$, for each $i \in I$. Thus we get a complex
giving a complex
\[
\Wmn \bullet: \Im\otimes \cO_{{ D_\nu}} \rmapo{\dmn 0} \Wmn 1 \rmapo{\dmn 1} \Wmn 2 \rmapo{\dmn 2} \cdots.
\]
Lemma \ref{lemma-Dred-qiso} follows then from a repeated application of the following result:

\begin{lem}\label{lem-omegaacyc} Assume that $\ch(k)=0$ or that  $(m_\nu,p)=1$ if $p=\ch(k)>0$.
Then the complex $\Wmn \bullet$ is acyclic.
\end{lem}
\begin{proof}
This is shown in \cite[Theorem 3.2]{KSS}. We include a sketch of the proof here. Let $\nu \in \{1,\ldots, n\}$ and write
\[
\WDn q= \Omega^q_{D_\nu}(\log(F+ \underset{i\in I-\{\nu\}}{\sum} D_i)_{|D_\nu}).
\]
%for the sheaf of $q$-differential forms on $D_\nu$ with logarithmic poles along the restriction of the divisor  $F+\sum_{i\neq \nu} D_i$  to $D_\nu$.
We have an exact sequence
\[
0 \to \Omega^q_{Y}(\log(F+\underset{i\in I-\{\nu\}}{\sum} D_i)) \to 
\Omega^q_{Y}(\log(F+\underset{i\in I}{\sum} D_i))  \rmapo{\Res_\nu^q}  \WDn {q-1}\to 0,
\]
where $ \Res_\nu^q$ is the residue homomorphism along $D_\nu$. % (see e.g.~\cite[2.3]{EV2}).
This induces an exact sequence
\begin{equation}\label{lemma.eq1}
0 \to \Im\otimes \WDn q \to \Wmn q \rmapo{\Resmn q} \Im\otimes\WDn {q-1}\to 0,
\end{equation}
where $\Resmn q=id_{\Im}\otimes \Res_\nu^q$. Now a direct computation shows
\[
\dmn{q-1}\circ\Resmn q + \Resmn {q+1}\circ\dmn q = m_\nu\cdot id_{\Wmn q},
\]
where $ \Im\otimes \WDn q$ is viewed as a subsheaf of $\Wmn q$ via \eqref{lemma.eq1}. This gives the contracting homotopy of the complex $\Wmn \bullet$,  completing the proof of the lemma.% under the assumption  $(m_\nu,p)=1$ (if $p=\ch(k)>0$), the contracting homotopy of the complex $\Wmn \bullet$, completing the proof of the lemma.
\end{proof}
\end{proof}

\subsection{Analogue of homotopy invariance for relative differentials}\label{lem.HI}
%Let $\Xb$ be a smooth variety over $k$ and $D\subset \Xb$ be an effective Cartier divisor such that $\Dred$
%is a simple normal crossign divisor. Put $X=\Xb-D$. 
%We use the notation in \S\ref{cyclecomplex}.
%For an integer $n\geq 0$ write
%$$
%X_n = X\times \cubb n {\hookrightarrow} \Xb_n=\Xb \times \cubb n 
%{\hookrightarrow} Y_n =\Xb \times \Pp n \hookleftarrow D_n=D \times \Pp n.
%$$ 
%Let $\pi_n:Y_n=\Xb \times \Pp n\to \Xb$ be the projection. Put $\Fn=\Pp n -\cubb n$.
%=\underset{1\leq i\leq n}{\sum} \Fni $.

%\begin{lem}\label{lem-HIOmega}
%The natural map
%\[
%\WXlogDD r \to R(\pi_n)_* \WYFDDlogn r
%\]
%is an isomorphism, where $\Fn$ denotes $X\times \Fn$ for simplicity.
%\end{lem}
%\begin{proof}
%Let $\phi:\Pp n \to \Pp {n-1}$ be the projection to the last $(n-1)$ factors. Then
%$D_n=\phi^{-1}(D_{n-1})$ and $\Fn=\Fna+\phi^{-1}(\Fnp)$.
%Hence, by the induction on $n$, the lemma is reduced to the following proposition
%(where we take $m=1$ and $H$ is a $k$-rational point):

\begin{prop}\label{prop-projformula}
Let $\Xb$ be a smooth variety over a field $k$ and let $D\subset \Xb$ be an effective divisor such that $\Dred$ is simple normal crossing.
Let $\bP=\bP^m_k$ be the projective space of dimension $m$ with $H\subset \bP^m_k$, 
a hyperplane. Let $\pi: \bP \times \Xb \to \Xb$ be the projection. Then the natural map
\[
\pi^*\colon \WXlogDD r \to  \bR\pi_* \Omega^{r}_{\Xb\times \bP}(\log \tilde{H} + \tilde{D})(-\tilde{D})
\]
is an isomorphism for every $r>0$ in the bounded derived category $D(\Xb_{Zar})$ of Zariski sheaves on $\Xb$. Here, in the second term, we let $\tilde{H}$ (resp. $\tilde{D}$) denote $\Xb\times_k H$ (resp. $D\times_k \bP^m_k$)
for simplicity.
\end{prop}
\begin{proof}By the derived projection formula, it is enough to show that the natural map
\begin{equation}\label{projformula.eq3}
\pi^*:\; \WXlogD r \to  \bR\pi_* \Omega^{r}_{\Xb\times \bP}(\log \tilde{H} + \tilde{D})
\end{equation}
is an isomorphism for every $r>0$. Since the logarithmic structure on the left side (resp. on the right side) is taken, by definition, with respect to the reduced structure of $D$ (resp. of $\tilde{D}+ \tilde{H}$), we may assume that $D$ is reduced (see Definition \ref{defi-log-diff}). 

Write $D_1,\ldots, D_n$ for the irreducible components of $D$. We prove \eqref{projformula.eq3} by induction on  $n$. If $n=0$, the assertion is well-known and follows from
the projective bundle formula for sheaves of differential forms. %: we recall the argument.
Suppose now $n\geq 1$ and write $I=\{1,\ldots, n\}$. Following \cite[2]{Sa2}, for each  $1\leq a\leq n$ we define
\[
D^{[a]}= \underset{\{i_1, \ldots, i_a\}\subset I}{\coprod}\; D_{i_1}\cap\cdots\cap D_{i_a},
\]
where { $\{i_1, \ldots, i_a\}\subset I$ range over all pairwise distinct indices.} 
Note that $D^{[a]}$ is the disjoint union of smooth varieties and that we have a canonical finite morphism
\[i_{a}\colon D^{[a]}\to \Xb \qfor a \geq 1.\]
On each $D^{[a]}$, we have a divisor with simple normal crossings
\[
E_a= \underset{1\leq i_1<\cdots< i_a\leq n}{\coprod}\;\Big( 
(D_{i_1}\cap\cdots\cap D_{i_a})\cap (\underset{j\not\in \{i_1,\dots,i_a\}}{\sum} D_j)\Big) 	\subset D^{[a]}.
\]
%or example, if $n=2$ then $D^{[1]}$ is the disjoint union of $D_1$ and $D_2$, and $E_1$ is given by two copies of the divisor $D_1\cap D_2$, one on each component of $D^{[1]}$. 
Then there is an exact sequence of sheaves on $\Xb$
\begin{multline*}%\label{projformula.eq4}
0\to \WX r \xrightarrow{\epsilon_{\Xb}} \WXlogD r \xrightarrow{\rho_1} i_{*} \Omega_{D^{[1]}}^{r-1}(\log E_1) \xrightarrow{\rho_2} i_{*} \Omega_{D^{[2]}}^{r-2}(\log E_2) \to \\
\to  \ldots \to i_{*}\Omega_{D^{[a]}}^{r-a}(\log E_a) \xrightarrow{\rho_{a+1}}  i_{*}\Omega_{D^{[a+1]}}^{r-a-1}(\log E_{a+1})\to \cdots 
\end{multline*}
where  $\epsilon_{\Xb}$ is the canonical inclusion, 
$i_*$ denote for simplicity the pushforwards by $i_a$ for all $a\geq 1$,
and the maps $\rho_a$ are given by the alternating sums of the residues (see \cite[Proposition 2.2.1]{Sa2}). 
%More precisely, let $C$
Similarly we have an exact sequence of sheaves on $\Xb \times \bP$
\begin{multline*}
0\to \Omega^{r}_{\Xb\times \bP}(\log \tilde{H}) \xrightarrow{\epsilon_{X\times \bP}} \Omega^{r}_{\Xb\times \bP}(\log \tilde{H} + \tilde{D}) \to i_{*} \Omega_{D^{[1]}\times \bP}^{r-1}(\log \tilde{H}+ E_1\times \bP) \to\cdots \\
%i_{2,*} \Omega_{D^{[2]}\times \bP}^{r-2}(\log \tilde{H}+ E_2)  
\cdots \to i_{ *}\Omega_{D^{[a]}\times \bP}^{r-a}(\log \tilde{H} + E_a\times \bP)
\to i_{*}\Omega_{D^{[a+1]}\times \bP}^{r-a-1}(\log \tilde{H} + E_{a+1}\times \bP)\to \cdots 
\end{multline*}
%where  for simplicity we write $i_*$ for the pushforwards by $D^{[a]}\times \bP \to \Xb \times \bP$
%and $E_{a}$ for $E_a\times \bP$ for all $a\geq 1$.
By induction assumption, we have the isomorphisms
\[
\WX r  \xrightarrow{\sim} \bR \pi_* \Omega^{r}_{\Xb\times \bP}(\log \tilde{H}) \qaq
i_{*}\Omega_{D^{[a]}}^{r-a}(\log E_a) \xrightarrow{\sim} \bR \pi_* i_{ *}\Omega_{D^{[a]}\times \bP}^{r-a}(\log \tilde{H} + E_a),
\]
which imply the desired assertion \eqref{projformula.eq3} by a standard argument from homological algebra.
%Set $A_a =\Omega_{D^{[a]}\times \bP}^{r-a}(\log \tilde{H} + E_a)$ and $B_a = \Omega_{D^{[a]}}^{r-a}(\log E_a)$ for $a=1,\ldots, n$, and let $F_A$ (resp. $F_B$) be the cokernel of $\epsilon_{\Xb\times \bP}$ (resp. of $\epsilon_{\Xb}$). Breaking down the sequence 
%\begin{multline*}%\label{projformula.equ.long}
%0\to F_A\to i_{1,*} \Omega_{D^{[1]}\times \bP}^{r-1}(\log \tilde{H}+ E_1) \to\cdots 
%i_{2,*} \Omega_{D^{[2]}\times \bP}^{r-2}(\log \tilde{H}+ E_2)  
%\to i_{a, *}\Omega_{D^{[a]}\times \bP}^{r-a}(\log \tilde{H} + E_a) \to \\
%\to i_{a+1, *}\Omega_{D^{[a+1]}\times \bP}^{r-a-1}(\log \tilde{H} + E_{a+1})\to \cdots 
%\end{multline*}
%into short exact sequences and applying the direct image functor, we get in the standard way an exact couple, giving rise to the  spectral sequence
%\[E_{1}^{a, q}= \bR^{a+q}\pi_{*} (i_{a+1, *} A_{a+1}) \Rightarrow \bR^{a+q}\pi_* F_A.\]
%By induction assumption, we have the isomorphism
%\[i_{a+1, *} B_{a+1} \xrightarrow{\sim} \bR \pi_* (i_{a+1, *} A_{a+1}) \qfor a\geq 0\] 
%from which we deduce that $F_B \xrightarrow{\sim } \bR\pi_* F_A$ in the bounded derived category of sheaves on $\Xb$. The proposition then follows from the $5$-lemma applied to the triangles
%\[\Omega^{r}_{\Xb}\xrightarrow{\epsilon_{\Xb}} \Omega^{r}_{\Xb} (\log D) \to F_B \rmapo{+}\]
%and
%\[\bR\pi_*\Omega^r_{\Xb\times\bP}(\log \tilde{H}) \xrightarrow{\epsilon_{\Xb\times \bP}} \bR\pi_*\Omega^r_{\Xb\times\bP}(\log \tilde{H}+D)  \to  \bR\pi_* F_A \rmapo{+}. \]
\end{proof}

\begin{rem}In the notations of Proposition \ref{prop-projformula}, let $U$ denote the open complement of $\Xb\times H$ in $\Xb\times \bP$. %It is isomorphic to a $m$-dimensional affine space $\mathbb{A}^m_{\Xb}$.  
Let ${ j }\colon U\to X\times \bP$ be the open immersion. Then we have a canonical injective map
	\[\Omega^{r}_{\Xb\times \bP}(\log \tilde{H}) \to j_{*} \Omega^{r}_{ U }\]
that allows us to identify { the sheaf of } $r$-differential forms with logarithmic poles along $H$ with a subsheaf of (the push forward of)  the sheaf of $r$-differential forms on an affine space over $\Xb$. The isomorphism of Proposition \ref{prop-projformula} induced by the pullback along the projection $\pi$ can be therefore interpreted as a weak homotopy invariance property. %, justifying the title of this section.
\end{rem}

%\newpage
\section{Regulator maps to relative de Rham cohomology}\label{clDR}

\subsection{Preliminary lemmas}
We resume the assumptions and the notations of \ref{Setting_YFD}.
%	\begin{defi} Let $E$ be a (reduced) simple normal crossing divisor on $Y$ and let $E_1,\ldots, E_s$ be its irreducible components. For the multi-index $I=(i_1, \ldots, i_t)$, write $E_I$ for the intersection 
%		\[E_{i_1}\cap \ldots\cap E_{i_t}.\]
%		Let $Z$ be a smooth  integral closed subscheme of $Y$ of pure codimension $q$. We say that $Z$ is in good position with respect to $E$ (or that $Z$ intersects transversally $E$) if  $Z\cap E_I$ has pure codimension $q$ in $E_I$ for each multi-index $I$. 	\end{defi}

	\def\FYDr{\cF^r_{Y|D}}
	\def\FZDr{\cF^r_{Z|D_Z}}
	
	\subsubsection{}\label{Setting-Pullback}
			{ Let $i\colon Z\;\hookrightarrow \;Y$ be a smooth integral closed subscheme of $Y$ which is transversal with
$D_{red}+F$. By definition, for any irreducible components $E_1,\ldots, E_s$ of $D_{red}+F$, the scheme-theoretic intersection
\[
Z \times_Y E_1\times_Y \times \cdots \times_Y E_s
\]
is smooth. Letting $D_Z=D\times_Y Z$ and $F_Z=F\times_Y Z$, this means that $(\bigstar)$ in \ref{Setting_YFD}
is satisfied for $(Z,D_Z,F_Z)$ instead of $(Y,D,F)$.}
Let $W\in \CYFDr$ and let $W_1, 	\ldots, W_n$ be the irreducible components of the intersection $W\cap (Z\cap X)$, so that each $W_i$ is closed in $Z\cap X$. Suppose that, for $i=1,\ldots, n$, $W$ and $Z\cap X$ intersect properly  at $W_i$ (i.e. that $W_i$ has codimension $r$ in $Z\cap X$ for every $i$). 
	As cycle on $Z\cap X$ we can then write the intersection of $W$ and $Z\cap X$ as
	\[
	i^*W= \underset{1\leq i\leq n}{\sum}\; n_i [W_i]. % \quad\text{where $n_i\in \bZ$ }. % and $W_i\subset Z$ is closed and integral}.
	\]
	%where $n_i\in \bZ$ are the intersection multiplicities of $W_i$ for $i=1,\ldots, n$.
	%	Let $D_Z=D\times_Y Z$ and $F_Z=F\times_Y Z$ be the scheme-theoretic intersections of $Z$ with $D$ and $F$ respectively.
			Lemma \ref{Containment_YFD} shows then that we have  $W_i\in \CZFDr$ for all $i$. 
%Indeed, $W_i$ is closed in $W$ and in $X$ for every $i$ and the closure ${\Wb}_i$ in $Z$ is in good position with respect to $F_Z$ and $D_Z$ by construction.  	%Assume that  $V$ intersects properly with $Z$ and $V\cap Z$ intersects properly with $D$.
			
	\begin{lem}\label{lemma-pullback}Let $Z$ and $W$ be as in \ref{Setting-Pullback}. Let  $\clOmega r W \in \H^r_{\Wb}(Y,\WYFD r)$ be the relative fundamental class of $W$ of \ref{thm-purityOmega}.  We have
	\[
			i^* \clOmega r W  = \underset{1\leq i\leq n}{\sum}\;n_i\clOmega r {W_i},
	\] 
			where $i^*\colon \H^r_{\Wb}(Y,\WYFD r) \to \H^r_{\Wb \cap Z}(Z, \Omega^{r}_{Z|D_Z} (\log F_Z))$ is the pullback along $i$ and $\clOmega r {W_i}$ for $i=1,\ldots,n$ are the relative fundamental classes of $W_i$
{  with respect to $(Z,D_Z,F_Z)$.}
	\end{lem}
	\begin{proof} By the  commutative diagram
		\[\xymatrix{ \H^r_{\Wb}(Y,\WYFD r)  \ar@{^{(}->}[d]^{\eqref{clOmega-eq5}} \ar[r]^{i^*} & \H^r_{\Wb\cap Z}(Z, \Omega^{r}_{Z|D_Z} (\log F_Z)) \ar@{^{(}->}[d]^{\eqref{clOmega-eq5}} \\
		\H^r_{W}(X, \Omega^r_X ) \ar[r]^{i^*} & \H^r_{W\cap (Z\cap X)}(Z\cap X, {\Omega^{r}_{Z}}_{|_{Z\cap X}}),
		}
		\]
		%Since the relative fundamental class of $W$ restricts to El Zein's fundamental class under the map \eqref{clOmega-eq5},
		we   reduce to the case $Y=X$. %, i.e. $D= \emptyset$, with $W$ integral closed subscheme of codimension $r$ in $Y$ and $Z$ smooth integral closed subscheme of codimension $p$ in $Y$, properly intersecting $W$.
		Let $cl_Y(W)\in \H^r_{W}(Y, \Omega^r_Y)$ be the fundamental class of $W$ in $Y$ and let $cl_Z(i^*W) = cl_{Z}(W\cdot Z)$ denote the element 
		\[\underset{1\leq i\leq n}{\sum}\;n_i cl_Z(W_i) \in \H^r_{W	\cap Z} (Z, \Omega^r_Z).\]
	%	where we let again $\{W_i\}_{i}$ be the irreducible components of the intersection $W\cap Z$ and $n_i\in \bZ$ be the intersection multiplicities.
	%	If  $i^*\colon \H^r_{W}(Y,\WW r Y) \to 	\H^r_{W\cap Z}(Z,\WW r Z)$ is the pullback along $i\colon Z \to Y$, 
		We have  to show the following identity
		\begin{equation}\label{compatibility-Gysin}
		i^*cl_Y({  W}) = cl_{Z}(i^*{  W}) \text{ in }\H^r_{W	\cap Z} (Z, \Omega^r_Z).
		\end{equation}
		%The statement then follows from the compatibility of the construction of the fundamental class in \cite{ElZ} with the intersection product and from the corresponding description of the 
	By \cite[III.3, Lemme 1]{ElZ}, the cup product with the fundamental class of the smooth subvariety $Z$ defines an injective Gysin map
	\begin{equation}\label{Gysin}
\iota\colon \H^r_{W\cap Z}(Z,\WW r Z) \to \H^{r+p}_{W\cap Z}(Y,\WW {r+p} Y), \quad \alpha \mapsto \alpha\cup cl_Y(Z)
\end{equation}
that maps, for every $i=1,\ldots, n$, the fundamental class of $W_i$ in $Z$ to the fundamental class of $W_i$ in $Y$. Hence we have  
	\[
\iota(cl_{Z}(i^*W))= cl_{Z}(i^*W)\cup cl_Y(Z) = cl_Y(i^*W).
\]
By \cite[III Theorem 1]{ElZ} (see also \cite[II, 4.2.12]{Gr}), we have the compatibility with the intersection product 
\[
cl_Y(W\cdot Z) = cl_Y(W)\cup cl_Y(Z) \text{ in }\H^{r+p}_{W\cap Z}(Y,\WW {r+p} Y).
\]
Finally, since the composite map
\[
 \H^r_{W}(Y,\WW r Y) \rmapo{i^*} \H^r_{W\cap Z}(Z,\WW r Z)\rmapo{\iota} \H^{r+p}_{W}(Y,\WW {r+p} Y)
\]
is also given by the cup product with $cl_Y(Z)$, we get 
\[
\iota(i^*cl_Y(W)) =cl_Y(W)\cup cl_Y(Z) = cl_Y(W\cdot Z) = \iota(cl_{Z}(i^*W)).
\]
Hence the identity \eqref{compatibility-Gysin}  follows from the injectivity  of the Gysin map \eqref{Gysin}.
\end{proof}

\subsection{Relative de Rham cohomology}\label{fuandclassDR}
\subsubsection{}We resume again the assumptions of \ref{Setting_YFD} and write $\WYFD \bullet$ for the relative de Rham complex. 
%\begin{equation}\label{WXFD}
%\WYFD \bullet \colon \cO_Y(-D)\rmapo d \WYFD 1  \rmapo d \WYFD 2 \to \cdots \to\WYFD r \to \cdots 
%\end{equation}
%It is canonically a subcomplex of the de Rham complex $(\iota_{X,*}\Omega^\bullet_X, d)$.
%For every $r\geq0$, we can consider the (brutal) truncated complex $\FrWYFD = \sigma_{\geq r}(\WYFD \bullet)$, i.e.~the subcomplex of $\WYFD \bullet$ defined by
%\begin{equation}\label{clOmega-eq6-1}
%0 \to \cdots\to 0\to \WYFD r \rmapo d \WYFD {r+1}  \to \cdots 
%\end{equation}
%as well as the truncated de Rham complex $\Omega^{\geq r}_X =\sigma_{\geq r}(\Omega^{\bullet}_X)\subset \Omega^{\bullet}_X$.
Let $T$ be an integral closed subscheme of $Y$ of codimension $c$. For $r\geq 0$, we define the relative de Rham cohomology of $Y$ with support on $T$ as the Zariski hypercohomology with support 
\[\bH^{*}_{T}(Y, \FrWYFD).\]
There is a strongly convergent spectral sequence 
\begin{equation}\label{clOmega-eq7} E_1^{p,q}	\Rightarrow \bH^{p+q}_{T}(Y, \FrWYFD),\end{equation}
	with $E_1^{p,q} = \H^q_T(Y,\WYFD p)$ if $p\geq r$ and $0$ otherwise. %  for $p,q\geq 0$,
%\[
%E_1^{p,q} = 
%\left.\left\{\begin{gathered}
% \H^q_T(Y,\WYFD p)\; \\ 
% 0 \\
%\end{gathered}\right.\quad
%\begin{aligned}
%&\text{if $p\geq r$,}\\
%&\text{if $p<r$.}
%\end{aligned}\right.
%\]
By Lemma \ref{clOmega-lem-Vanishing-Relative}, $E_1^{p,q}=0$ for $q<c$, so that we have
\begin{equation}\label{clOmega-eq8}
\bH^{p+q}_T(Y,\FrWYFD)=0 \qfor p+q<r+c.
\end{equation}
If now  $\codim_Y(T)=r$, \eqref{clOmega-eq7} gives % By Lemma \ref{clOmega-lem-Vanishing-Relative}, %together with the definition of the $E_1$-terms of the spectral sequence, gives us that $E_1^{p,q}\neq 0$ for $p+q=2r$ if and only if $p=q=r$.  We can see similarly that $E_\infty^{r,r} \simeq E_1^{r,r}$, so that 
%we  get
\begin{equation}\label{clOmega-eq9}
	\bH^{2r}_T(Y,\FrWYFD )\xrightarrow{\sim}  \Ker\big(\H^r_T(Y,\WYFD r)\rmapo {d} 	\H^r_T(Y,\WYFD {r+1}) \big)
	\end{equation}
%By \eqref{clOmega-eq8} the natural map $\FrWYFD \to \WYFD r[-r]$ induces an isomorphism
%\begin{equation}
%\bH^{2r}_T(Y,\FrWYFD )\simeq\Ker\big(H^r_T(Y,\WYFD r)\rmapo {d} \H^r_T(Y,\WYFD {r+1}) \big).
%\end{equation}
where $d$ is the map induced by the exterior derivative. 
%\begin{rem}
%Let $V$ be any open subset $V\subset Y$ containing the generic point of $T$. By \eqref{clOmega-eq9}, the localization exact sequence (together with   
%\eqref{clOmega-eq4}) implies the injection
%\begin{equation}\label{clOmega-eq10}
%\bH^{2r}_T(Y,\FrWYFD )\hookrightarrow \bH^{2r}_{T\cap V}(V,{\FrWYFD}_{|V}).
%\end{equation}
%In particular, for  $V=X$ we get from \eqref{clOmega-eq10} the injection
%\begin{equation*}\label{eq-Injection}\iota_X^{*}\colon \bH^{2r}_T(Y,\FrWYFD )\hookrightarrow \bH^{2r}_{T\cap X}(X,\Omega_X^{\geq r}).\end{equation*}
%\end{rem}

We  now give a refinement of Theorem \ref{thm-purityOmega}.

\begin{theo}\label{thm-purityDR}
For $W\in \CYFDr$,
% with closure $\Wb$ in $Y$, 
 there is a unique element, called the fundamental class of $W$ in the relative de Rham cohomology,
\[
 \clDR r W \in \bH^{2r}_{\Wb}(Y,\FrWYFD )
\]
which maps under the map \eqref{clOmega-eq9} to the fundamental class $\clOmega r W \in \H^r_{\Wb}(Y,\WYFD r)$ 
defined in Theorem \ref{thm-purityOmega}.
\end{theo}
\begin{proof} The spectral sequence \eqref{clOmega-eq7} has an analogue in the non relative setting
\[E_1^{p,q} \Rightarrow \bH^{p+q}_{W}(X, \Omega^{\geq r}_X) \]
	for $E_1^{p,q}= \H^q_{W}(X, \Omega^p)$ if $p\geq r$ and $0$ otherwise. Using \ref{clOmega-lem-Vanishing} instead of \eqref{clOmega-eq4} we have %the analogue of \eqref{clOmega-eq9}, namely
	\begin{equation}\label{clOmega-eq9-Absolute}
		\bH^{2r}_W(X,\Omega^{\geq r}_X )\xrightarrow{\sim}  \Ker\big(\H^r_W(X,\Omega^r_X)\rmapo {d} 	\H^r_W(X,\Omega^{r+1}_{X}) \big). 
	\end{equation}
	By Remark \ref{rem-CC1}, the fundamental class $cl^r_{\Omega}(W)\in \H^r_W(X,\Omega^r_X)$ is in the kernel of the map induced by the exterior derivative $d$. Therefore by \eqref{clOmega-eq9-Absolute}, the absolute class $cl^r_{\Omega}(W)$ gives rise to an element of $\bH^{2r}_W(X,\Omega^{\geq r}_X )$. By \eqref{clOmega-eq5}, Theorem \ref{thm-purityOmega} and \eqref{clOmega-eq9}, the same holds for the relative fundamental class $\clOmega r W \in \H^r_{\Wb}(Y,\WYFD r)$, giving rise to $\clDR r W \in \bH^{2r}_{\Wb}(Y,\FrWYFD )$ as required. 
%	The last stated property is then clear by construction.
	%The injectivity of the restriction map \eqref{eq-Injection} finally gives that the 
\end{proof}
%\subsubsection{}Let $\FYDr$ denote either the complex $\Omega^r_{Y|D}(\log F)[-r]$ concentrated in degree $r$ or  the relative de Rham complex $\FrWYFD$.
%For $W\in \CYFDr$, we denote by 
%\[\clYD W \in \bH^{2r}_{\Vb}(Y,\FYDr)\]
%  the fundamental class in Theorems \ref{thm-purityOmega} and \ref{thm-purityDR}.

\begin{lem}\label{lem-pullback-faces}Let $Z$ and $W$ be as in \ref{Setting-Pullback}. Let $\clDR r W \in \bH^{2r}_{\Wb}(Y,\FrWYFD )$ be the fundamental class of $W$ in relative de Rham cohomology of Theorem  \ref{thm-purityDR}. Then we have
		\[
				i^* \clDR r W  = \underset{1\leq i\leq n}{\sum}\;n_i\clDR r {W_i},
		\]
	where $i^*\colon \bH^{2r}_{\Wb}(Y,\FrWYFD) \to \bH^{2r}_{\Wb \cap Z}(Z, \Omega^{\geq r}_{Z|D_Z} (\log F_Z))$ is the pullback along $i$ and $\clDR r {W_i}$ for $i=1,\ldots,n$ are the relative fundamental classes of $W_i$ in de Rham cohomology.
\end{lem}
\begin{proof}By the construction of $cl_{DR}^r$ (cf. \eqref{clOmega-eq9}), the lemma follows from the same assertion for $cl^r_{\Omega}$, that is proven in Lemma \ref{lemma-pullback}.
\end{proof}
\subsection{The construction of the regulator map}\label{cyclemapDR}
\subsubsection{}Let $\Xb$ be a  smooth variety over a field $k$ and let $D\subset \Xb$ be an effective Cartier divisor on $\Xb$ such that the reduced part $\Dred$
is  simple normal crossing. Let $X=\Xb-D$ be the open complement.
%
%Let $\WX 1(\log D)$ be the Zariski sheaf on $\Xb$  of absolute K\"ahler differentials on $X$ with logarithmic poles along $\Dred$ (see Definition \ref{defi-log-diff}).
For $i\geq 1$, write %We write $\WX i(\log D)$ for its $i$-th external power and set 
\[
\WXD i = \Omega^i_{\Xb}(\log D_{red})\otimes_{\cO_{\Xb}} \cO_{\Xb}(-D)
\]
for the sheaf of relative differentials and $\WXD \bullet$ for the relative de Rham complex (see Definition \ref{defi-log-diff}).
%The exterior derivative gives rise to the {\it relative de Rham complex}
%\begin{equation*}
%\WXD \bullet : \;\;\cO_{\Xb}(-D)\rmapo d \WXD 1  \rmapo d \WXD 2 \to \cdots \to\WXD r \to \cdots 
%\end{equation*}
%For every $r\geq0$, we have the (brutally) truncated complex
%\begin{equation*}
%\WXD {\geq r} : \;\;0\to \cdots \to  \WXD r  \rmapo d \WXD {r+1} \to \cdots 
%\end{equation*}
In this section we show that Theorems \ref{thm-purityOmega} and \ref{thm-purityDR} can be used to construct a cycle map in the derived category $D^-(\Xb_{\zar})$ of bounded {above} complexes of Zariski sheaves on $\Xb$
\begin{equation}\label{clrelDR.eq1}
\reggDR \colon \bZXDr \to \WXD {\geq r} ,
\end{equation}
where $\bZXDr$ is the relative motivic complex introduced in \eqref{eq.ZXD}.
The induced maps
\begin{equation}\label{clrelDR.eq2}
\regDR q r: \HMXDr q \to  \bH^q(\Xb_{\zar},\WXD {\geq r}) %= \H_{dR}^{q}(\Xb|D, r)
\end{equation}
are called the {\it regulator maps to relative de Rham cohomology.}

\subsubsection{}\label{Setting-Constructioncycle}In what follows all the cohomology groups are taken over the Zariski site. We use the notation $A_\star$ to denote a cubical object $A\colon\square^{op}\to \cC$ in an abelian category $\cC$. The associated chain complex is denoted $A_*$ and we write $ (A_{*})_{non-degn} = A_*/(A_*)_{degn} $ for the non-degenerate quotient. In the notations of Section \ref{cyclecomplex}, we write
\[
X_n = X\times \cubb n {\hookrightarrow} \Xb_n=\Xb \times \cubb n 
{\hookrightarrow} Y_n =\Xb \times \Pp n \supset D_n=D \times \Pp n.
\] 
Write $\Fn$ for the divisor $\Xb\times (\Pp n -\cubb n)$, $D_n$ for the divisor $D\times \Pp n$ on $Y_n$ and  $\pi_n:Y_n \to \Xb$ for the projection. The triple $(Y_n,F_n,D_n)$ satisfies the condition $(\bigstar)$ of \ref{Setting_YFD} and we can consider the complex $\FrWYFDn $ on $(Y_n)_{\zar}$.
%Let $\FrWYFDn $ be the sheaf on $(Y_n)_{\zar}$ defined as $\FrWYFD$ for $(Y,F,D)=(Y_n,F_n,D_n)$. Note that the condition $(\bigstar)$ of \ref{Setting_YFD} is satisfied for this choice of $(Y,F,D)$.

\def\zzUbD#1#2{\underline{z}^{#1}(\Ub|D\cap\Ub,#2)}
\def\zUbD#1#2{z^{#1}(\Ub|D\cap\Ub,#2)}
\subsubsection{}
For an open subset $\Ub$ of $\Xb$, we write $U$ for the intersection $\Ub \cap X$, $\Ub_n$ for $\pi_n^{-1}(\Ub) \subset Y_n$ and $U_n$ for $\Ub_n \cap X_n$. % To lighten the notation, we will still denote by $D_n$ the restriction of $D_n$ to $\Ub_n$, namely $D_n\times_{\Xb} U = D_n\cap \pi_n^{-1}(\Ub)$.
Let $\SrnU$ be the set of closed subsets of $U_n$ of pure codimension $r$ whose irreducible components
are in $C^{r}(\Ub|D\cap \Ub, n)$ (cf. Definition \ref{def-CXD}). In particular, for every $V\in C^{r}(\Ub|D\cap \Ub, n)$ we have 
\[\phiV^*(D\cap \Ub \times \Pp n)\leq \phiV^*(\Ub\times \Fn)\]
where $\Vb$ denotes the closure of $V$ in $\Ub_n\times \Pp n$.
We apply Theorem \ref{thm-purityDR} to get  a natural map
\begin{equation*}\label{cyclemap.eq1}
cl^{r,n}_U: \zzUbD r n \to \indlim {W\in \SrnU} \bH^{2r}_{\Wb}(\Ub_n,\FrWYFDn )
\end{equation*}
sending a cycle $\underset{1\leq i\leq r}{\sum}\; m_i[W_i]$ with $W_i\in C^{r}(\Ub|D\cap \Ub, n)$ and $m_i\in\bZ$ to
\[
\underset{1\leq i\leq r}{\sum}\; m_i\cdot cl_{DR}^r {W_i} \in \bH^{2r}_{\Wb}(U_n,\FrWYFDn ),
\]
where $\Wb$ is the Zariski closure of $W=\underset{1\leq i\leq r}{\cup}\; W_i$ in $\Ub_n$.
\subsubsection{}\label{cyclemap-cube}
For $i=1, \ldots, n$, $\epsilon \in \{0, \infty\}$, let $\inie$ denote the inclusion of the face of codimension $1$ in $U\times \cubb n$ given by the equation $y_i=\epsilon$. %By the projective bundle formula for K\"ahler differentials we have
%\[ {\Omega^{\geq r}_{Y_{n-1}|D_{n-1}}(\log F_{n-1})} \cong {\inie}^* \FrWYFDn\]
%and
Lemma \ref{lem-pullback-faces} shows then that the diagram
\begin{equation*}\label{eq.contrav-face-maps}\xymatrix{ \zzUbD r n \ar[r] \ar[d]^{{\inie}^*} &\indlim {W\in \SrnU} \bH^{2r}_{\Wb}(\Ub_n,\FrWYFDn ) \ar[d]^{{\inie}^*} \\
\zzUbD r {n-1} \ar[r] &\indlim {W\in \cS^{r,n-1}_U} \bH^{2r}_{\Wb}(\Ub_{n-1},{\Omega^{\geq r}_{Y_{n-1}|D_{n-1}}(\log F_{n-1})} )
}
\end{equation*}
is commutative. %, where the left vertical map ${\inie}^*$ denotes the pullback of cycles along $\inie$ and the right vertical ${{\inie}^*}$ denotes the pullback of differentials along the same map.
The map $cl^{r,n}_U$ is then contravariant for face maps, giving rise to a natural map of cubical objects of complexes 
\begin{equation}\label{eqMapcl} \zzUbD r \star [-2r] \xrightarrow{cl^{r,\star}_U }\indlim {W\in \SrstarU} \bH^{2r}_{\Wb}(U_\star,\FrWYFDstar)[-2r].\end{equation}
\begin{lem}\label{lem-HIOmega}
Let $i$ be a positive integer. The natural map
\[
\Omega^i_{\Xb|D} = \WXlogDD i \to \bR(\pi_n)_*\Omega^{i}_{Y_n|D_n}(\log F_n)
\]
is an isomorphism in $D^-(\Xb_{\zar})$.
\end{lem}
\begin{proof}
The proof is by induction on $n$. The case $n=1$ follows from Proposition \ref{prop-projformula}. %, applied in the case $m=1$ and $H$ the $k$-rational point $y=1$ in $\bP^1$.
For $i=1, \ldots, n$, we denote by $F^{n}_{i,1}$ the face $y_i=1$ of  $\Pp n$. Let $\phi:\Pp n \to \Pp {n-1}$ be the projection to the last $(n-1)$ factors. %Note that
%$D_n=\phi^*(D_{n-1})$ and $\Fn=F^{n}_{1,1}+ \phi^* (F_{n-1})$.
By Proposition \ref{prop-projformula} applied to $Y_{n-1} \times \bP^1 \xrightarrow{\phi} Y_{n-1}$, together with the derived projection formula, we get then
\begin{equation}\label{eq.Lem-HIOmega}\Omega^{i}_{Y_{n-1}}(\log D_{n-1} + F_{n-1})(-D_{n-1})
	\xrightarrow{\sim} \bR (\phi)_* (\Omega^{i}_{Y_n} (\log (\phi^* (D_{n-1} + F_{n-1}) +F^{n}_{1,1})) (-D_n)). \end{equation} %\tensor_{\cO_{Y_{n-1}}} \cO(-F_{n-1})
Applying $\bR (\pi_{n-1})_*$ to \eqref{eq.Lem-HIOmega}, the claim follows from the induction assumption.
\end{proof}
\subsubsection{}
Let $\cI_n(r)^{\bullet}$ be the Godement resolution of the complex $\FrWYFDn$ on $(Y_n)_{\zar}$. Note that
%The assignment $ \FrWYFDn  \to I_n(r)^{\bullet}$ defines a functor from the category of complexes of Zariski sheaves to the full subcategory of flasque sheaves, so that 
%{  By functoriality of Godement resolutions,} 
$\cI_\star(r)^{\bullet}$ has a natural structure of cubical object, making the canonical map $\FrWYFDn \to \cI_n(r)^{\bullet}$ a morphism of cubical objects. 
Let $\Wb$ be a closed subscheme of $Y_n$ of pure codimension $r$ and let $\tau_{\leq 2r}\Gamma_{\Wb}(\Ub_n,\cI_n(r)^\bullet)$ be the canonical (good) truncation of $\Gamma_{\Wb}(\Ub_n,\cI_n(r)^\bullet)$. By \eqref{clOmega-eq8}, we have 
\[\bH^i_{\Wb}(\Ub_n,\FrWYFDn )=0 \text{ for } i<2r,\]
so that the morphisms of complexes 
{  
\[\tau_{\leq 2r}\Gamma_{\Wb}(\Ub_n,\cI_n(r)^\bullet)\xrightarrow{\alpha^{r,n}_{\Wb}} \bH^{2r}_{\Wb}(\Ub_n,\FrWYFDn )[-2r]  \]
\begin{equation}\label{eqMapAlpha}\indlim {W\in \SrnU} \tau_{\leq 2r}\Gamma_{\Wb}(\Ub_n,\cI_n(r)^\bullet)
\xrightarrow{\alpha^{r,n}} \indlim {W\in \SrnU}\bH^{2r}_{\Wb}(\Ub_n,\FrWYFDn )[-2r] 
\end{equation}
}
are quasi-isomorphisms, both compatible with the cubical structure. %, where $\bH^{2r}_{\Wb}(\Ub_n,\FrWYFDn )[-2r]$ is concentrated in degree $2r$.
%\begin{rem} Although cubical (or simplicial) objects are usually not well-behaved in the derived categories, the existence of a functorial acyclic resolution such as the Godement resolution allow us to extend the cubical structure 
%\end{rem}
 \begin{rem}The complex $\cI_{n}(r)^\bullet = \cI_n(r)^\bullet_{(\Xb,D)}$ is contravariantly functorial in the pair $(\Xb, D)$, where by a morphism of pairs $(\Xb, D)\to (\Xb^\prime, D^\prime)$ we mean a morphism of schemes $f\colon \Xb \to \Xb^\prime$ such that $f^*(D^\prime)$ is defined and $f^*(D^\prime) \leq D$ as Cartier divisors on $\Xb$. 
 \end{rem} 
\subsubsection{}Combining \eqref{eqMapAlpha} and \eqref{eqMapcl}, we have a diagram of complexes
\begin{equation*}%\label{cyclemap.eq2}
\xymatrix{
 \zzUbD r n [-2r] \ar[r]^{\hskip -40pt cl^{r,n}_U} & \indlim {W\in \SrnU} \bH^{2r}_{\Wb}(\Ub_n,\FrWYFDn )[-2r]  \\
& \indlim {W\in \SrnU} \tau_{\leq 2r}\Gamma_{\Wb}(\Ub_n,\cI_n(r)^\bullet) \ar[u]^{\alpha^{r,n}} 
  \;\ar[r]^{\hskip 30pt\beta^{r,n}}& \Gamma(\Ub_n,\cI_n(r)^{\bullet}) \\}
\end{equation*}
where $\beta^{r,n}$ is the canonical map ``forget supports''. Since all the morphisms are contravariant for face maps, % Note:
%\begin{enumerate}
%\item[$(\clubsuit)$]
%$\alpha^{r,n}$ is a quasi-isomorphism.
%\end{enumerate}
%Indeed we have $H^i_{\Wb}(U_n,\FrWYFDn )=0$ for $i<2r$ by \eqref{clOmega-eq8}.
%All maps in the above diagram are contravariant for face maps.
%For $\beta^{r,n}$, this is obvious and for $\alpha^{r,n}$ this follows from Lemma \ref{lemma-pullback}.
 we get a diagram of cubical objects of complexes
\begin{equation*}%\label{cyclemap.eq2}
\xymatrix{
 \zzUbD r \star [-2r] \ar[r]^{\hskip -40pt cl^{r,\star}_U} & \indlim {W\in \SrstarU} \bH^{2r}_{\Wb}(\Ub_\star,\FrWYFDstar)[-2r]  \\ %& I(r)^{\bullet}\ar[d]^{\gamma}[-r] \\
& \indlim {W\in \SrstarU} \tau_{\leq 2r}\Gamma_{\Wb}(\Ub_\star,\cI_\star(r)^\bullet)\;\; \ar[u]^{\alpha^{r,\star}} 
  \ar[r]^{\hskip 40pt\beta^{r,\star}}& \;\;\Gamma(\Ub_\star,\cI_\star(r)^{\bullet}). \\}
\end{equation*}
Let $\Tot({ \tau_{\leq 2r}} \Gamma_{\cS_{U}^{r, *}}(\Ub_*,\cI_*(r)^{\bullet}))$ be the total complex of the non-degenerate associated complex 
\[\indlim {W\in \SrstarU} \frac{{  \tau_{\leq 2r}}  \Gamma_{\Wb}(\Ub_\star,\cI_\star(r)^\bullet)}{{\Big( \tau_{\leq 2r}\Gamma_{\Wb}(\Ub_\star,\cI_\star(r)^\bullet)\Big)}_{\degn}}\]
and let 
\[\alpha^{r, *}\colon \Tot({  \tau_{\leq 2r}} \Gamma_{\cS_{U}^{r, *}}(\Ub_*,\cI_*(r)^{\bullet})) \to \indlim {W\in \SrstarU} (\bH^{2r}_{\Wb}(\Ub_*,{\Omega^{\geq r}_{Y_{*}|D_{*}}(\log F_{*})})[-2r] )_{non-\degn} := F_{*, r} \]
 be the induced morphism. Since the maps $\alpha^{r,n}$ are quasi-isomorphisms for every $n$, the same holds for $\alpha^{r, *}$. 
 Let $\WXD {\geq r}\to \cI(r)^{\bullet}$ be  the Godement resolution of the relative de Rham complex $\WXD{\geq r}$ on $\Xb$ and let $\gamma$ be the  inclusion to the factor at $\star=0$, $\cI(r)^{\bullet}\xrightarrow{\gamma}\Gamma(\Ub_\star,\cI_\star(r)^{\bullet})$. 
 %\begin{equation*}\label{eq.gamma}\cI(r)^{\bullet}\xrightarrow{\gamma}\Gamma(\Ub_\star,\cI_\star(r)^{\bullet}).\end{equation*}
 By Lemma \ref{lem-HIOmega}, the induced map
 \[\cI(r)^{\bullet} \xrightarrow{\gamma}\Tot \frac{\Gamma(\Ub_*,\cI_*(r)^{\bullet})}{\Gamma(\Ub_*,\cI_*(r)^{\bullet})_{\degn}} = \tilde{\Gamma}(\Ub_*, \cI_*(r)^\bullet)\]
  is a quasi-isomorphism. Combining it with the previously constructed maps we get a diagram of complexes
 \begin{equation*}%\label{cyclemap.eq2}
 \xymatrix{
  \zUbD r * [-2r] \ar[r]^{cl^{r,*}_U} & 
F_{*, r}  & \cI(r)^{\bullet}\ar[d]^{\gamma}[-r] \\ %\indlim {W\in \SrstarU} (\bH^{2r}_{\Wb}(\Ub_*,{\Omega^{\geq r}_{Y_{*}|D_{*}}(\log F_{*})})[-2r] )_{non-\degn}
 & \Tot(\Gamma_{\cS_{U}^{r, *}}(\Ub_*,\cI_*(r)^{\bullet})) \ar[u]^{\alpha^{r,\star}} 
   \ar[r]^{\hskip 30pt \beta^{r,\star}}& \tilde{\Gamma}(\Ub_*, \cI_*(r)^\bullet) \\}
 \end{equation*} 
  that sheafified on $\Xb_{\zar}$ gives the desired map \eqref{clrelDR.eq1}
  \begin{equation*}
  \reggDR \colon \bZXDr \to \WXD {\geq r}.
  \end{equation*}
\begin{rem}The strategy used to construct regulator map \eqref{clrelDR.eq1}, that relies on the existence of a functorial flasque resolution of $\FrWYFDn$ is due to Sato, taken from \cite[3.5-3.10]{S}.
	\end{rem}
%\begin{equation*}% 
%\xymatrix{
% \zzUD r \star [-2r] \ar[r]^{\hskip -40pt cl^{r,\star}_U} & \indlim {W\in \SrstarU} H^{2r}_{\Wb}(U_\star,\FrWYFDstar)[-2r]  & I(r)^{\bullet}\ar[d]^{\gamma}[-r] \\
%& \indlim {W\in \SrstarU} \tau_{\leq 2r}\Gamma_{\Wb}(U_\star,\cI_\star(r)^\bullet)\;\; \ar[u]^{\alpha^{r,\star}} 
%  \ar[r]^{\hskip 40pt\beta^{r,\star}}& \;\;\Gamma(U_\star,\cI_\star(r)^{\bullet}) \\}
%\end{equation*}
%where Let $\WXD {\geq r}\to I(r)^{\bullet}$ is the Godement resolution of the relative de Rham complex $\WXD{\geq r}$ on $\Xb$ and $\gamma$ is the inclusion to the factor at $\star=0$.
%We can then take the total complex of the associated non-degenerate complexes to get the diagram

%By Lemma \ref{lem-HIOmega}, $\gamma$ is a quasi-isomorphism, and the same holds for   $\alpha^{r,\star}$.

% \begin{equation*}% 
% \xymatrix{
%  \zUD r * [-2r] \ar[r]^{\hskip -50pt cl^{r,*}_U} & 
%\indlim {W\in \SrstarU} (\bH^{2r}_{\Wb}(\Ub_*,{\Omega^{\geq r}_{Y_{*}|D_{*}}(\log F_{*})})[-2r] )_{non-degn}  & I(r)^{\bullet}\ar[d]^{\gamma}[-r] \\
 %& \Tot(\Gamma_{\cS_{U}^{r, *}}(\Ub_*,\cI_*(r)^{\bullet})) \ar[u]^{\alpha^{r,\star}} 
 %  \ar[r]^{\hskip 30pt\beta^{r,\star}}& \Tot \frac{\Gamma(\Ub_*,\cI_*(r)^{\bullet})}{\Gamma(\Ub_*,\cI_*(r)^{\bullet})_{degn}} \\}
 %\end{equation*} 

\subsection{Compatibility with proper push forward}

Let $(Y,F,D)$ and $\YFDp$ be two triples satisfying the condition $(\bigstar)$ of \ref{Setting_YFD} and 
let $f\colon\YFDp\to(Y,F,D)$ be an admissible proper morphism between the triples $\YFDp$ and $(Y,F,D)$
(see \S\ref{ProperpushclOmega}). 
Suppose that $f$ is either a closed immersion or a smooth morphism.
The Gysin map of Lemma \ref{pushforDerCat} can be turned into a map of complexes 
\[f_* \colon Rf_* \Omega^{\bullet + n}_{\Yp|\Dp}(\log \Fp) [n]\to \Omega^{\bullet}_{Y|D}(\log F)\]
where $n= \dim Y- \dim Y^\prime$. We can show this by the same method of \cite[II.5]{RD} (see also \cite[Prop. 2.2]{Ha}). It induces a map of the relative de Rham cohomology groups with supports
\begin{equation}\label{eq.pushDeRham}f_*\colon \bH^{2r'}_{\Wb}(Y',\Omega^{\geq r'}_{\Yp|\Dp}(\log \Fp) ) \to \bH^{2r}_{\overline{f(W)}}(Y,\FrWYFD )
\end{equation}
 that is compatible with the fundamental class of Theorem \ref{thm-purityDR}. %, namely
%\begin{equation}\label{eq.CompClDR}f_* cl_{DR}^{r'}(W) = cl_{DR}^{r}(f_*([W]))\end{equation}
%where the equality \eqref{eq.CompClDR} follows from \eqref{eq.CompClOmega} and the fact that the fundamental class of a cycle is a cohomology class of a closed form (see Remark \ref{rem-CC1}).
\subsubsection{}We resume the notations of Section \ref{cyclecomplex} and \ref{Setting-Constructioncycle}. Let $f\colon \Xb'\to \Xb$ be a proper morphism between smooth varieties over $k$ that is either a closed immersion or a smooth morphism. Let $D'$ and $D$ be effective Cartier divisors such that $D'_{red}$ and $D_{red}$ are simple normal crossing. Write $X' = \Xb'-\Dp$ (resp. $X=\Xb-D$) for the open complement. Suppose for simplicity that $D^\prime = f^*D$. The map of cubical complexes  
\[ \zzUbD r \star [-2r] \xrightarrow{cl^{r,\star}_U }\indlim {W\in \SrstarU} \bH^{2r}_{\Wb}(U_\star,\FrWYFDstar)[-2r]\]
constructed in \ref{cyclemap-cube} is compatible with the pushforward \eqref{eq.pushDeRham}. This can be used to show that the regulator map constructed in Section \ref{cyclemapDR}
 is compatible with the proper pushforward. %, i.e. that the diagram 
%\[\xymatrix{ \bZ(r')_{\Xb'|D'} \ar[d]^{f_*} \ar[r]^{\phi_{DR}} & \Omega^{\geq r'}_{\Xb'|D'}\ar[d]^{f_*}\\
%\bZXDr \ar[r]^{\reggDR} & \WXD {\geq r}
%} 
%\]
%is commutative.

%\newpage

\def\bZYde#1{\bZ_{Y}^\mathcal{D}(#1)}
\def\bZXDde#1{\bZ_{\Xb|D}^\mathcal{D}(#1)}
\def\bZXbde#1{\bZ_{\Xb}^\mathcal{D}(#1)}

\section{Regulator maps to relative Deligne cohomology}\label{clDe}

In this section we work over the base field $k=\bC$.
For an algebraic variety $Y$ over $\bC$, write $\cO_Y$ for the analytic sheaf of holomorphic functions on $Y$ and $\Omega^{i}_{Y}$ for the sheaf of holomorphic $i$-th differential forms. %Let $D^b(Y_{\an})$ denotes the derived category of bounded complexes of sheaves of abelian groups on the analytic site $Y_{\an}$ on $Y(\bC)$.

\subsection{Relative Deligne complex}\label{reldelcpx}

Let $\Xb$ be a smooth algebraic variety over $\bC$ and let $D\subset \Xb$ be an effective Cartier divisor on $\Xb$ such that the reduced part $\Dred$
is  simple normal crossing. Let $j\colon X = \Xb - D \hookrightarrow \Xb$ be the open complement.
Write  $\WX i(\log D)$ for the sheaf of meromorphic $i$-th differential forms on $\Xb_{\an}$ that are holomorphic on $X$ and with at most logarithmic poles along $\Dred$. We write  % Resuming the notations of \ref{cyclemapDR}, we write 
\[
\WXD i = \Omega^i_{\Xb}(\log D)\otimes_{\cO_{\Xb}} \cO_{\Xb}(-D),
\]
and $\WXD \bullet$ for the relative (analytic) de Rham complex. Let $\bC_X$ denote the constant sheaf $\bC$ on $X_{\an}$. The proof of the following Lemma  uses the same strategy of the proof of Proposition \ref{prop-projformula}.

%\subsection{Hodge structure on Betti cohomology with compact support}
% Let $\Xb$ be a smooth variety over $\bC$ and $D\subset \Xb$ be a reduced simple normal crossign divisor. 
% Put $X=\Xb-D$. We write $\cO_{\Xb}$ for the sheaf of holomorphic functions, 
% $\cO_{\Xb}(-D)\subset \cO_{\Xb}$ for the ideal of 
% $D\subset\Xb$ and $\WX i$ for the sheaf of holomorphic differential forms on $\Xan$. We consider the sheaves 
% \[
% \WXD i = \Omega^i_{\Xb}(\log D)\otimes_{\cO_{\Xb}} \cO_{\Xb}(-D)\subset \WX i
% \]
% which form the de Rham complex:
% \begin{equation*}
% \WXD \bullet : \;\;\cO_{\Xb}(-D)\rmapo d \WXD 1  \rmapo d \WXD 2 \to \cdots \to\WXD r \to \cdots 
% \end{equation*}
% %Let $\ZrX$ denote the constant sheaf $(2\pi i)^r\bZ\subset \bC$ on $\Xan$.

\begin{lem}\label{HSHc}Assume $D$ is a reduced simple normal crossing divisor on $\Xb$.
Then the canonical map $j_!\bC_X \to \WXD \bullet$ is a quasi-isomorphism. 
\end{lem}
%\begin{proof}
%Write $D_1,\ldots, D_n$ for the irreducible components of $D$. We do again induction on $n$, the case $n=0$ being clear. %We use the strategy of the proof of Proposition \ref{prop-projformula}, doing induction on $n$.
%If $n=0$ the assertion is clear. Write $I=\{1,\ldots, n\}$ and
 %For each  $1\leq a\leq n$ define $D^{[a]}$ as in Proposition \ref{prop-projformula}. It is 
%\[
%D^{[a]}= \underset{\{i_1, \ldots, i_a\}\subset I}{\coprod}\; D_{i_1}\cap\cdots\cap D_{i_a},
%\]
%where  $\{i_1, \ldots, i_a\}\subset I$ range over all pairwise distinct indices. Then $D^{[a]}$ is the
%disjoint union of smooth varieties and we have a canonical finite morphism $i_{a}\colon D^{[a]}\to \Xb$. 
%\[i_{a}\colon D^{[a]}\to \Xb \qfor a \geq 1.\]
%The lemma follows then from the standard exact sequences
%\[
%0\to j_!\bC_{X} \to \bC_{\Xb} \to (i_1)_*\bC_{D^{[1]}}\to (i_2)_*\bC_{D^{[2]}} \to \cdots,
%\]
%\[
%0\to \WXD \bullet\to \WX \bullet \to (i_1)_*\Omega_{D^{[1]}}^{\bullet} \to 
%(i_2)_*\Omega_{D^{[2]}}^{\bullet} \to \cdots.
%\]
%\end{proof}
\begin{rem}
By Lemma \ref{lemma-Dred-qiso}, in characteristic $0$ the relative de Rham complex $\WXD \bullet$ is independent from the multiplicity of $D$. Lemma \ref{HSHc} gives then 
\[j_!\bC_X \xrightarrow{\sim} \WXDred \bullet \xleftarrow{\sim} \WXD \bullet \]
where $\WXDred \bullet$ is defined as $\WXD \bullet$ with $D$ replaced by $\Dred$.
\end{rem}
\subsubsection{}Let $Y, X, F, D$ be as in \ref{Setting_YFD} with $k=\bC$. Let $\tau\colon  Y-F {\hookrightarrow} Y$, $j\colon X\hookrightarrow \Xb = Y- F$. % and write
%\[
%X=Y-(F+D) \overset{j}{\hookrightarrow} \Xb= Y-F \overset{\tau}{\hookrightarrow} Y.
%\]
On  $Y_{\an}$ we have the relative de Rham complex $\WYFD \bullet$ and the truncated subcomplex $\FrWYFD$. Write $\WYFDred \bullet$ for the variant of $\WYFD \bullet$ with $D$ replaced by $\Dred$.
\begin{lem} We have a natural isomorphism in $D^b(Y_{\an})$
\begin{equation}\label{cyclemapDeligne.eq0.5}
\beta\colon\bR\tau_* j_!\bC_X \simeq \WYFDred \bullet.
\end{equation}
\end{lem}\begin{proof}
By Lemma \ref{HSHc} we have a functorial quasi-isomorphism
\[
j_!\bC_X \xrightarrow{\sim}\WXDred \bullet
\]
and pushing forward along $\tau$ one has the quasi-isomorphisms
\[
\bR\tau_* \WXDred \bullet \simeq \tau_* \WXDred \bullet \xleftarrow{\sim} \WYFDred \bullet
\]
where the canonical map $\WYFDred \bullet\hookrightarrow\tau_* \WXDred \bullet$ is a quasi-isomorphism by  using the same argument as \cite[II, Lemme 6.9]{De1}.
\end{proof}
\subsubsection{}For every integer $r\geq0$, write $\ZrX$ for the constant sheaf $(2 i \pi)^r\bZ\subset \bC$ on $\Xan$. We define the {\it relative Deligne complex} for the triple $(Y, F, D)$ as the object in the bounded derived category $D^b(Y_{\an})$ given by
\begin{equation}\label{relDelignecomplex2}
\bZYFDrde=\Cone[\bR\tau_*j_!\ZrX \;\oplus\;\FrWYFD \rmapo{\iota-\gamma} \bR\tau_*j_!\bC_X][-1],
\end{equation} 
where $\iota$ is induced by $\ZrX \hookrightarrow \bC_X$ and $\gamma$ is the composite
\[
\FrWYFD \to \WYFD \bullet \xrightarrow{\sim} \WYFDred \bullet\overset{\beta}{\simeq} \bR\tau_*j_!\bC_X,
\]
where the last isomorphism $\beta$ is defined in  \eqref{cyclemapDeligne.eq0.5}.
%We have then a natural distinguished triangle in $D^b(Y_{\an})$:
%\begin{equation}\label{relDelignecomplex3}
%\bZYFDrde \to \bR\tau_*j_!\ZrX \;\oplus\;\FrWYFD \to \bR\tau_*j_!\bC_X \overset{+}{\longrightarrow}.
%\end{equation} 
\begin{rem}We note that the map $\gamma$ is a priori defined only at the level of the derived category. However, after replacing $\WXDred \bullet$ with a functorial resolution $ \WXDred\bullet \to \cI^\bullet_{\Xb|D}$, we can lift it to an actual morphism of complexes.
\end{rem}
 
%\begin{proof}
%By Lemma \ref{lemma-Dred-qiso} we have a quasi-isomorphism
%$\WXD \bullet \simeq \WXDred \bullet$, where $\WXDred \bullet$ is defined as $\WXD \bullet$ with
%$D$ replaced by $\Dred$. Hence Lemma \ref{relDecpx.lem1} follows from Lemma \ref{HSHc}(1).
%\end{proof}

%Before going into this, we explain compatibility of $\reggDe$ with the de Rham and Betti cycle maps.

%\subsection{Fundamental class in relative Deligne cohomology}\label{fundamentalclassDe}
%By definition, $W$ is a closed subvariety of $X$ of codimension $r$ whose closure
%$\Wb$ in $Y$ satisfies the modulus condition \eqref{moduluscond}.
\subsubsection{}Let $W\in \CYFDr$ and write $\Wb$ for its closure in $Y$.  Since $\Wb \cap (Y-F) \cap D=\emptyset$, % As noted in after Definition \ref{def-CYFD}, the condition implies that
%\[\Wb \cap (Y-F) \cap D=\emptyset\]
 %and that $W$ is closed in $\Xb=Y-F$. Therefore,
 the localization exact sequence for $X\hookrightarrow \Xb$  gives the isomorphism  
\[ \H^i_{\Wb\cap \Xb}(\Xb_{\an},j_!\ZrX)\xrightarrow{\sim} \H^i_{W}(\Xan,\ZrX).\]
In particular, we have
\begin{equation}\label{purityBetti.eq1} 
\bH^i_{\Wb}(\Yan,\bR\tau_*j_!\ZrX)\simeq \H^i_{\Wb\cap \Xb}(\Xb_{\an},j_!\ZrX)\xrightarrow{\sim}  \H^i_{W}(\Xan,\ZrX),
\end{equation} 
%where the first isomorphism follows from Leray's spectral sequence.
so that purity for the Betti cohomology implies
\begin{equation}\label{purityBetti.eq2} 
\bH^i_{\Wb}(\Yan,\bR\tau_*j_!\ZrX) =0 \qfor i<2r,
\end{equation}
and that for $i=2r$ we have the Betti fundamental class
\begin{equation}\label{purityBetti.eq3} 
\clB r W\in \bH^{2r}_{\Wb}(\Yan,\bR\tau_*j_!\ZrX)=\H^{2r}_{W}(\Xan,\ZrX).
\end{equation}
obtained by the fundamental class of $W$ in $\H^{2r}_{W}(\Xan,\bZ)$ after multiplication by $(2i\pi)^r$. 
%On the other hand Theorem \ref{thm-purityDR} and \eqref{clOmega-eq8} provide the vanishing 
%\begin{equation}\label{purityDRan.eq1} 
%H^{i}_{\Wb}(\Yan,\WYFD {\geq r})0\qfor i<2r
%\end{equation}
%and the de Rham fundamental class
%\begin{equation}\label{purityDRan.eq2} 
%\clDR r W\in H^{2r}_{\Wb}(\Yan,\WYFD {\geq r}).
%\end{equation}
\begin{theo}\label{thm-purityDeligne}
For $W\in \CYFDr$   we have
\[
\bH^{q}_{\Wb}(\Yan,\bZYFDrde )=0 \qfor q<2r
\]
and there is a unique element, called the fundamental class of $W$ in the relative Deligne cohomology,
\begin{equation}\label{purityDeligne.eq1} 
 \clDe r W \in \bH^{2r}_{\Wb}(\Yan,\bZYFDrde )
\end{equation}
which maps to $(\clB r W,\clDR r W)$ under the (injective) map 
\begin{equation}\label{purityDeligne.eq2} 
\bH^{2r}_{\Wb}(\Yan,\bZYFDrde ) \to \bH^{2r}_{\Wb}(\Yan,R\tau_*j_!\ZrX) \oplus \bH^{2r}_{\Wb}(\Yan,\FrWYFD)
\end{equation}
arising from \eqref{relDelignecomplex2}, where $\clDR r W\in \bH^{2r}_{\Wb}(\Yan,\WYFD {\geq r})$ is the  de Rham fundamental class of Theorem \ref{thm-purityDR}.
\end{theo}
\begin{proof}
From \eqref{relDelignecomplex2} we have the long exact sequence
\begin{multline*}
 \bH^{i-1}_{\Wb}(\Yan,\bR\tau_*j_!\bC)\to
\bH^{i}_{\Wb}(\Yan,\bZYFDrde )\to \\
\bH^{i}_{\Wb}(\Yan,\bR\tau_*j_!\ZrX) \oplus \bH^{i}_{\Wb}(\Yan,\FrWYFD) \to \bH^{i}_{\Wb}(\Yan,\bR\tau_*j_!\bC_X).
\end{multline*}
Recall that by \eqref{clOmega-eq8} we have the vanishing
\begin{equation}\label{purityDRan.eq1} 
\H^{i}_{\Wb}(\Yan,\WYFD {\geq r})=0\qfor i<2r,
\end{equation}
so that the first assertion follows from \eqref{purityDRan.eq1}  and \eqref{purityBetti.eq2}, proving at the same time the injectivity of \eqref{purityDeligne.eq2}.
To prove the second assertion, it suffices to show that $\clB r W$ and $\clDR r W$ have the same image in
$\bH^{2r}_{\Wb}(\Yan,R\tau_*j_!\bC)$, giving rise to a unique element in $\bH^{2r}_{\Wb}(\Yan,\bZYFDrde )$, that we denote $\clDe r W $. 

Note that pulling back along the inclusion $\iota_X \colon X\to Y$ gives rise to a commutative diagram
\[\xymatrix{\bH^{2r}_{\Wb}(\Yan,\bR\tau_*j_!\ZrX) \oplus \bH^{2r}_{\Wb}(\Yan,\FrWYFD)  \ar[r] \ar[d]^{\iota_X^{*}} & \bH^{2r}_{\Wb}(\Yan,\bR\tau_*j_!\bC_X) \ar[d]^{\simeq} \\
\H^{2r}_{W}(\Xan,\ZrX) \oplus \bH^{2r}_{W}(\Xan,\Omega^{\geq r}_X)  \ar[r] & \H^{2r}_{W}(\Xan,\bC_X).
}\]
where the right vertical map is an isomorphism by \eqref{purityBetti.eq1}.
%{\color{red} (Remark:I think the injectivity of the left vertical map (by \eqref{clOmega-eq10}) is not necessary to show the desired assertion.)}
%the left vertical map is obtained by the isomorphism \eqref{purityBetti.eq1} and by the map 
%\[ \bH^{2r}_{\Wb}(\Yan,\WYFD {\geq r}) \to \bH^{2r}_{W}(\Xan,\Omega^{\geq r}_{X}),\]  
%that is injective by \eqref{clOmega-eq10}.  
As noticed in \cite[Remark 6.4(b)]{EV}, the fundamental classes
\[\clB r W\in \H^{2r}_{W}(\Xan,\ZrX) \;\text{ and }\;\clDR r W\in \bH^{2r}_{W}(\Xan,\Omega^{\geq r}_{X})\]
have the same image in $\bH^{2r}_{W}(\Xan,\WX {\bullet})\simeq \H^{2r}_{W}(\Xan,\bC)$. The claim follows then from the fact that, by Theorem \ref{thm-purityOmega}, the class $\clDR r W\in \bH^{2r}_{\Wb}(\Yan,\WYFD {\geq r})$ maps, via the pullback along $\iota_X$, to the  class $\clDR r W$.
\end{proof}

\subsection{Deligne cohomology and product structures}
Let $Y$ be a smooth algebraic variety over $\bC$. Recall that for every $r\geq0$, the (analytic) Deligne complex of $Y$ is defined as 
\[\bZYde r \colon 0 \to \bZ(r)_Y\to \cO_Y\xrightarrow{d} \Omega^1_Y\to\ldots\to \Omega^{r-1}_Y \]
with $\bZ(r)_Y = \bZ (2\pi i)^r$ in degree $0$. Following \cite{EV}, we define a multiplication
\begin{equation}\label{eq.cupDeligne}\cup\colon \bZYde p \tensor \bZYde q \to \bZYde {p+q}\end{equation}
 by \begin{equation*}\label{eq.cupDeligne2}x \cup y = 
\left.\left\{\begin{gathered}
 xy \; \\ 
x\wedge dy \\
\end{gathered}\right.\quad
\begin{aligned}
&\text{if $\deg x =0$,}\\
&\text{if $\deg x \neq 0, \deg y = q$,}  
\end{aligned}\right.
\end{equation*}
and $0$ otherwise (the degree refers to the degree in the complex). The cup product is associative, compatible with the differential, satisfies the Leibniz rule, and equips the cohomology $\bigoplus_{p,q} \H^p_\mathcal{D}(X, \bZ(q))$ with a ring structure.
\subsubsection{}We introduce a relative version $\bZXDrde$ of the Deligne complex on $\Xb$:
\begin{equation}\label{relDelignecomplex}
\bZXDrde\colon j_!\ZrX  \to \cO_{\Xb}(-D) \to \WXD 1  \to \cdots\to \WXD  {r-1},
\end{equation} 
where $j_!\ZrX$ is put in degree zero and the map $j_!\ZrX  \to \cO_{\Xb}(-D)$ is  obtained by adjunction from the canonical inclusion $\ZrX\to \cO_X$.
The complex $\bZXDrde$ is naturally a subcomplex of the classical Deligne complex $\bZ(r)^{\mathcal{D}}_{\ol{X}}$ on $\ol{X}$.
The hypercohomology groups
\[
 \HDeXDr q = \bH^{q}(\Xb_{\an},\bZXDrde) %\reggDe\colon 
\]
are called the relative Deligne cohomology groups for the pair $(\Xb,D)$. By Lemmas \ref{HSHc} and \ref{lemma-Dred-qiso}, the definition of $\bZXDrde$ implies the following
\begin{lem}\label{relDecpx.lem1}Let $r\geq 0$ and let  $\FrWXD$ be the $r$-th (brutally) truncated subcomplex of $\WXD \bullet$. Then there is a natural distinguished triangle in $D^b(\Xb_{\an})$
\[
\bZXDrde \to j_!\ZrX \;\oplus\;\FrWXD \to j_!\bC_X \overset{+}{\longrightarrow}.
\]
\end{lem}
\begin{exam}\label{ExDe}For $r=0$ we have $\bZ(0)_{\Xb|D} = j_! \bZ_{X}$, so that $\H^q_{\cD}(\Xb|D, \bZ(0)) = \H^q_{c}(X, \bZ)$. % is nothing but Betti cohomology with compact support.
	 
For $r=1$, let $\cO_{\Xb}^{\times}$ (resp. $\cO_{D}^\times$) denote the sheaf of invertible holomorphic functions on $\Xb$ (resp. on $D$) and let $\cO_{\Xb|D}^\times$ be the kernel of the restriction map
\[1\to \cO_{\Xb|D}^\times \to \cO_{\Xb}^{\times} \to {\iota_{D}}_*\cO_{D}^{\times} \to 1,.\]
{ where $\iota_D:D\to \Xb$ is the closed immersion.}
Then the complex $\bZ(1)_{\Xb|D}$ is  quasi isomorphic to $\cO_{\Xb|D}^\times[-1]$ via the exponential map. % In particular, one has
%\[\H^2_{\cD}(\Xb|D, \bZ(1)) \cong \operatorname{Pic}(\Xb|D).\] 
\end{exam}
Restricting the cup product \eqref{eq.cupDeligne} to $\bZXDde p$ gives a module structure 
\[\bZXDde p \tensor \bZXbde q \to \bZXDde {p+q} \]
that we read in cohomology as
\begin{equation}\label{eq.ModuleProductDeligne}\H^p_\mathcal{D}(\Xb|D, \bZ(q)) \tensor\H^{p'}_\mathcal{D}(\Xb, \bZ(q')) \to \H^{p+p'}_\mathcal{D}(\Xb|D, \bZ(q+q')).\end{equation}

\subsection{The construction of the regulator map}\label{subsec:construction-reg-Deligne}In this section we construct a cycle map in the derived category $D^-(\Xban)$ of bounded above complexes of analytic sheaves on $\Xb$
\begin{equation}\label{cyclemapDeligne.eq1}
\phi_\cD\colon \ep^*\bZXDr \to \bZXDrde,
\end{equation}
where $\bZXDr$ is the relative motivic complex of \eqref{eq.ZXD} and 
$\ep\colon \Xban\to \Xbzar$ is the map of sites.
 The induced maps
\begin{equation*}\label{cyclemapDeligne.eq2}
\regDe q r: \HMXDr q \to \HDeXDr q
\end{equation*}
are called the {\it regulator maps to relative Deligne cohomology.} 

In the notations of Section \ref{cyclecomplex} and \ref{Setting-Constructioncycle}, we write again for $n\geq 0$
\[
X_n = X\times \cubb n {\hookrightarrow} \Xb_n=\Xb \times \cubb n 
{\hookrightarrow} Y_n =\Xb \times \Pp n \supset D_n=D \times \Pp n.
\] 
Let $\Fn$ be the divisor $\Xb\times (\Pp n -\cubb n)$, $D_n$  the divisor $D\times \Pp n$ on $Y_n$ and  $\pi_n:Y_n \to \Xb$  the projection. 
Let $\bZYFDnrde$ be the sheaf on $(Y_n)_{\an}$ defined as $\bZYFDrde$ for the triple $(Y,F,D)=(Y_n,F_n,D_n)$. 
The analogue of Lemma \ref{lem-HIOmega} is given by
\begin{lem}\label{lem-HIZde}
Let $i$ be a positive integer.  The natural map 
\[
\bZXDrde \xrightarrow{\sim} \bR(\pi_n)_* \bZYFDnrde.
\]
is an isomorphism in $D^-(\Xban)$.
\end{lem}
\begin{proof}
By Lemma \ref{relDecpx.lem1} and \eqref{relDelignecomplex2}, the statement follows from the natural isomorphism
\[
 \WXlogDD i \to \bR(\pi_n)_*\Omega^{i}_{Y_n|D_n}(\log F_n) \qfor i>0,
\]
given by Lemma \ref{lem-HIOmega}, and from the isomorphism
\begin{equation}\label{HI-Betti}
j_!\bZ_X \xrightarrow{\sim} \bR(\tilde{\pi}_n)_* (j_n)_!\bZ_{X_n} \qwith  \tilde{\pi}_n : \Xb_n=\Xb\times \cubb n \to \Xb,
\end{equation}
which follows from the homotopy invariance for the Betti cohomology, where $j_n\colon X_n \to \Xb_n$ denotes the open immersion.
\end{proof}\medskip
%\begin{rem}The isomorphism \eqref{HI-Betti} is the analogue for Betti cohomology of Lemma \ref{lem-HIOmega} and \ref{lem-HIZde}.
	%\end{rem}
	The method developed in \ref{cyclemapDR} applies, \textit{mutatis mutandis}, to this setting,  using the fundamental class in relative Deligne cohomology constructed in Theorem \ref{thm-purityDeligne} and Lemma  \ref{lem-HIZde} in place of Lemma \ref{lem-HIOmega}. This gives rise to the natural map \eqref{cyclemapDeligne.eq1}. The same argument (this time using the fundamental class \eqref{purityBetti.eq3} in Betti cohomology and \eqref{HI-Betti} in place of Lemma \ref{lem-HIOmega}) provides a cycle map in  $D^-(\Xban)$
	\begin{equation}\label{cyclemapBetti.eq1}
	\phi_B \colon \ep^*\bZXDr \to j_! \bZ(r)_X,
	\end{equation}
	whose induced maps in cohomology will be called regulator maps to Betti cohomology.
%The method developed in \ref{cyclemapDR} applied to the fundamental classes in Betti \eqref{purityBetti.eq3} and Deligne \eqref{purityDeligne.eq1}  cohomology provide the cycle maps
%\[
%\reggB: \ep^*\bZXDr \to j_!\ZrX,\quad  \reggDe: \ep^*\bZXDr \to \bZXDrde. %\quad \reggDR: \ep^*\bZXDr \to \WXD {\geq r}
%\]
\begin{rem}By construction, we have the following commutative square of distinguished triangles in $D^-(\Xban)$
\begin{equation*}\label{clrelDeDRB}
\xymatrix{
\ep^*\bZXDr \ar[r]^{\hskip -30pt\Delta}\ar[d]^{\reggDe} 
& \ep^*\bZXDr \oplus \ep^*\bZXDr \ar[r]^{\hskip 10pt \delta}\ar[d]^{\reggB\oplus\reggDR} 
& \ep^*\bZXDr \overset{+}{\longrightarrow}\ar[d]^{\reggB} \\
\bZXDrde \ar[r] & j_!\ZrX\oplus \FrWXD \ar[r]  & j_!\bC_X\overset{+}{\longrightarrow}\\
}
\end{equation*}
where $\Delta$ is the diagonal and $\delta$ is the difference of identity maps, and the lower distinguished triangle
comes from Lemma \ref{relDecpx.lem1}.
\end{rem}

%\newpage
\section{Infinitesimal Deligne cohomology and the additive dilogarithm}\label{Sec:addDilog}
\def\Cepsm{\bC[\varepsilon]_m}
\def\Xepsm{X[\varepsilon]_m}
\def\depss{\frac{\mathrm{d}\varepsilon}{\varepsilon}}
\def\OmegaXinfm{\Omega^\bullet_{\Xepsm}}

\subsection{The infinitesimal complex}\label{SettinginfDeligne}We keep working as in the previous section over the field of complex numbers $\bC$. Let $X$ be a smooth algebraic variety over $\bC$. %Let $\cO_X$ be the analytic sheaf of holomorphic functions on $X$ and $\Omega^i_X$ the sheaf of holomorphic $i$-th forms on $X$.
 For $m\geq 2$, we denote by $\Cepsm$ the truncated polynomial ring $\bC[\varepsilon]/(\varepsilon^m)$.
The $m$-th \emph{infinitesimal de Rham} complex $(\OmegaXinfm, d)$ of $X$ is the complex of analytic sheaves on $X$
\begin{equation}\label{infdeRham}\Omega^i_{\Xepsm} = 
\left.\left\{\begin{gathered}
 \cO_X\tensor_\bC \Cepsm \; \\ 
 \Omega_X^i\tensor_\bC \Cepsm \oplus (\Omega^{i-1}_X \tensor_\bC \varepsilon \Cepsm) \wedge \depss \\
\end{gathered}\right.\quad
\begin{aligned}
&\text{if $i = 0$,}\\
&\text{if $i\geq 1$,}  
\end{aligned}\right.
\end{equation}
with differentials $d\colon \Omega^i_{\Xepsm} \to \Omega^{i+1}_{\Xepsm}$ given by linear extension of the formulas
\[ d(\omega \tensor \varepsilon^r)  = (d\omega)\tensor {\varepsilon^r} + (-1)^i r \omega \tensor \varepsilon^r \wedge \depss \quad \text{ for } \omega \in \Omega^i_X\]
\[d(\omega \tensor (\varepsilon^r \wedge \depss)) = (d\omega) \tensor (\varepsilon^r \wedge \depss) \quad \text{ for } \omega \in \Omega^{i-1}_X.\]
 %We think of the sections of $(\Omega^{i-1}_X \tensor_\bC \varepsilon \Cepsm) \wedge \depss$ as $i$-th  differential forms with one $d\varepsilon$.
As $\cO_X\tensor_\bC \bC[\varepsilon]_m$-module, $\Omega^i_{\Xepsm}$  is generated by symbols
\[\varepsilon \omega_1 \wedge \depss, \ldots, \varepsilon^{m-1}\omega_{m-1} \wedge \depss, \quad \text{with } \omega_i \in \Omega^{i-1}_X\]
\[ \nu_0, \varepsilon \nu_1, \ldots, \varepsilon^{m-1}\nu_{m-1}, \quad \text{with } \nu_i \in \Omega^{i}_X\]
and relations $\varepsilon^i(\omega_j\varepsilon^{j}\wedge \depss) = \varepsilon^{i+j}\omega_j\wedge \depss, \omega_m\varepsilon^{m}\wedge \depss = 0$.  

\subsubsection{}Write $X[t]$ for the product $X\times \mathbb{A}^1$. 
%, $\pi\colon X[t]\to X$ for the first projection and $q\colon X[t] \to \mathbb{A}^1$ for the second projection.
For an $\cO_X$-module $\mathcal{M}$, write $\mathcal{M}[t]$ for the pullback on $X[t]$. The de Rham complex of $X[t]$ can be then written as
\[\Omega^\bullet_{X[t]} = \Omega^\bullet[t] \oplus \Omega^{\bullet -1}_X[t]dt.\]
Let $D_m$ denote the effective Cartier divisor on $X[t]$ given by the closed subscheme $X\times_\bC \Spec{\Cepsm}$. Let $\Omega^\bullet_{X[t]|D_m}$ denote the relative de Rham complex of the pair $(X[t], D_m)$, i.e. 
\[ \Omega^\bullet_{X[t]|D_m} = \Omega^\bullet_{X[t]}(\log |D_m|)(-D_m) = t^m(\Omega^\bullet[t] \oplus \Omega^{\bullet -1}_X[t]\frac{dt}{t} ). \]
It is a subcomplex of $\Omega^\bullet_{X[t]}$. According to \eqref{infdeRham}, we have in  $D^b(X_{\an})$
\[ \Cone[ \Omega^\bullet_{X[t]|D_m} \to \Omega^\bullet_{X[t]}] \cong \Omega^\bullet_{\Xepsm}.\]
\subsubsection{} Let $1\leq \nu \leq m-1$. We can define a filtration on $\OmegaXinfm$ by 
\[ F^{\nu} \OmegaXinfm = \Omega^\bullet_X\tensor \varepsilon^\nu \Cepsm \oplus (\Omega^{\bullet-1}_X \tensor \varepsilon^\nu \Cepsm )\wedge \depss.\]
%It's easy to check that $Fil^{\nu} \OmegaXinfm$ is a subcomplex of $\OmegaXinfm$.
\begin{lem}\label{SplittingInfCpx}In the above notations, one has that $F^{\nu} \OmegaXinfm$ is a subcomplex of $\OmegaXinfm$ and there is 	 a canonical splitting 
	\[ \OmegaXinfm = \Omega^\bullet_X \oplus F^1\OmegaXinfm\]
	as direct sum of complexes.
	\end{lem}
\def\bZXdeinf#1{\bZ(#1)^\mathcal{D}_{X[\varepsilon]_m}}
\subsubsection{}We define the infinitesimal Deligne complex of $X$ as 
\[\bZXdeinf r = i_*\bZ(r) \to \cO_X\tensor_\bC \Cepsm \to \Omega^1_{\Xepsm}\to\ldots\to \Omega^{r-1}_{\Xepsm}\]
that we can see alternatively as complex on $X[t]$ or on $X$. 
The filtration on $\OmegaXinfm$ induces a filtration $F^\nu \bZXdeinf r $ that satisfies
\begin{equation}\label{eq.filDelcpx} F^\nu \bZXdeinf r = F^\nu \Omega^r_{\Xepsm}[-1] \qfor \nu \geq 1,\end{equation}
and we have, by Lemma \ref{SplittingInfCpx}, a canonical splitting 
\[\bZXdeinf r = \bZ(r)^\mathcal{D}_X\oplus F^1  \bZXdeinf r =\bZ(r)^\mathcal{D}_X\oplus F^1 \Omega^r_{\Xepsm}[-1] \]
\begin{lem}\label{lem.Delinfsplit} We have a quasi isomorphism
    \[ \bZXdeinf r = \Cone[ \bZ(r)^{\mathcal{D}}_{X[t]|D_m} \to \bZ(r)^\mathcal{D}_{X[t]}] \]
 as complexes of sheaves on $X[t]$.
    \end{lem}
By Lemma \ref{lem.Delinfsplit} together with homotopy invariance for Deligne cohomology, we  have a split exact sequence 
\[ 0\to \H_\mathcal{D}^i(X, \bZ(r)) \to \mathbb{H}^i(X,  \bZXdeinf r) \to \H^{i+1}_\mathcal{D}(X[t] | D_m , \bZ(r))\to 0\]
that using \eqref{eq.filDelcpx} gives the isomorphism 
\begin{equation}\label{eq.Delignecpx-Fil}\mathbb{H}^{i}(X, F^1\Omega^{<r}_{X[\varepsilon]_m}) \xrightarrow{\cong}\H^{i+1}_\mathcal{D}(X[t] | D_m , \bZ(r))  \end{equation}
\subsubsection{}Let $\bZ(r)_{X[t]|D_m}$ be the relative motivic complex of the pair $(X[t], D_m)$. The cycle map
\[ \phi_\mathcal{D} \colon \epsilon^*\bZ(r)_{X[t]|D_m} \to \bZ(r)^\mathcal{D}_{X[t] | D_m}, \quad \text{ with }\epsilon^*\colon D^b( X[t]_{\text{Zar}}) \to D^b(X[t]_{\text{an}}) \]
constructed in \ref{subsec:construction-reg-Deligne} gives rise to regulator maps
\[ \phi_\mathcal{D}^{i, r}\colon \H^{i}_\mathcal{M}(X[t] | D_m, \bZ(r)) \to \H^{i}_\mathcal{D}(X[t] | D_m , \bZ(r)).\]
Composing this with the canonical map from additive higher Chow groups to relative motivic cohomology gives
\[\varphi^{2r-i, r}_\mathcal{D}\colon T\CH^r(X, i+1; m ) = \CH^r(X[t] | D_m, i )\to \H^{2r-i}_\mathcal{D}(X[t] | D_m , \bZ(r)), \]
that using \eqref{eq.Delignecpx-Fil} finally provides the map (that we denote in the same way) to the cohomology of the $F^1$-piece of the truncated infinitesimal de Rham complex
\[\varphi^{2r-i+1, r}_\mathcal{D}\colon  T\CH^r(X, i; m )\to \mathbb{H}^{2r -i}(X, F^1\Omega^{<r}_{X[\varepsilon]_m}). \]
Particularly interesting is the case of $0$-cycles. The infinitesimal regulator from additive higher Chow groups reads, in this case, as 
\[\phi_X^{n,n;m}\colon T\CH^n(X,n;m)\to \mathbb{H}^{n-1}(X, F^1\Omega^{<n}_{X[\varepsilon]_m})\]
that for $n=2$ is identified with
\[\phi_X^{2,2;m}\colon T\CH^2(X, 2; m )\to \mathbb{H}^1\Big(X, {\big[}\cO_X \tensor \varepsilon \Cepsm \xrightarrow{d}  
\Omega^1_X\tensor \varepsilon \Cepsm \oplus  \cO_X \tensor \varepsilon \Cepsm \wedge \frac{d\varepsilon}{\varepsilon}{\big]}\Big). \]
This regulator generalizes  Bloch-Esnault additive regulator map \cite[5]{BE2}, as shown by the computation in the next section.

\begin{rem} For $X$ affine, a direct computation shows that
\[ \mathbb{H}^{n-1}(X, F^1\Omega^{<n}_{X[\varepsilon]_m}) \cong \H^0(X,\Omega^{n-1}_{X/\mathbb{C}} \tensor \varepsilon \mathbb{C}[\varepsilon]_m )\]
so that the regulator map $\phi^{n,n;m}_X$ induces a natural map
\[\phi_X\colon T\CH^n(X, n;m) \to \H^0(X,\Omega^{n-1}_{X/\mathbb{C}} \tensor \varepsilon \mathbb{C}[\varepsilon]_m ) \]
that we can see as Hodge-theoretic incarnation of Bloch-Esnault-R\"ulling regulator map
\[ T\CH^n(\mathbb{C}(X), n;m) \to \mathbb{W}_{m-1} \Omega^{n-1}_{\mathbb{C}(X)},\]
 where $\mathbb{C}(X)$ is the function field of $X$.
\end{rem}
\subsection{The Bloch-Esnault additive dilogarithm} We keep the  notations of \ref{SettinginfDeligne}. The module structure of the relative Deligne cohomology groups over the usual Deligne cohomology groups of \eqref{eq.ModuleProductDeligne} extends to a product 
\[\bZ(p)_{X}^\mathcal{D}\tensor F^1\Omega^{<q}_{\Xepsm} [-1] \to F^1\Omega^{<(p+q)}_{\Xepsm}[-1] \quad \text{ in } D^{b}(X).\]
For $p=q=1$, we have 
\[\bZ(1)^\mathcal{D}_X\tensor [0\to \varepsilon \cO_X \tensor \Cepsm] \to [\varepsilon \cO_X \tensor \Cepsm \xrightarrow{d} \varepsilon \Omega^1_X\tensor \Cepsm \oplus \varepsilon \cO_X \tensor \Cepsm \wedge \frac{d\varepsilon}{\varepsilon}][-1].\]
We can write the product in cohomology explicitly. Let $f\in \Gamma(X, \cO_X^\times)$ be a global section of the sheaf of invertible holomorphic functions on $X$ and let $\{U_\alpha\}$ be an open covering of $X$ such that the logarithm of $f_\alpha = f_{| U_\alpha}$ is defined, denoted $\log_\alpha f$. Mapping $f$ to the analytic \v{C}ech cocycle
\[(\frac{1}{2\pi \sqrt{-1}} { (} \log_\beta f - \log_\alpha f { )},  \frac{1}{2\pi \sqrt{-1}}\log_\alpha f)  \]
gives an explicit inverse to the exponential map $\exp\colon \H_\mathcal{D}^1(X, \bZ(1)) \xrightarrow{\cong} \H^0(X,\cO_X^\times)$.  
Suppose now that $m = 2$, so that a section of $\mathbb{H}^1(X, F^1\Omega^{ <1}_{\Xepsm}[-1])$ is just of the form $\varepsilon a$, for $a\in \Gamma(X, \cO_X)$. Using the same covering $\{U_\alpha\}$ of $X$, we can represent elements of the cohomology group
\[\mathbb{H}^1\Big(X, [\varepsilon \cO_X \tensor \Cepsm \xrightarrow{d} \varepsilon \Omega^1_X\tensor \Cepsm \oplus \varepsilon \cO_X \tensor \Cepsm \wedge \frac{d\varepsilon}{\varepsilon}]\Big)\ni (\varepsilon g_{\alpha\beta},  \varepsilon \omega_\alpha + \epsilon h_\alpha  \frac{ d\varepsilon}{\varepsilon})\]
for
\[g_{\alpha\beta} \in \Gamma(U_{\alpha \beta}, \cO_X),\quad \omega_\alpha\in \Gamma(U_\alpha, \Omega^1_X)\qaq h_\alpha\in \Gamma(U_\alpha, \cO_X),\]
subject to the obvious cocycle condition. In particular, note that
\[ d g_{\alpha \beta}  = \omega_\beta - \omega_\alpha = d (h_\beta - h_\alpha).\]
The explicit description of the product in Deligne cohomology presented in \eqref{relDelignecomplex} gives then
\begin{equation}\label{eq.ExplicitcupProduct}
\begin{aligned}
 \H_\mathcal{D}^1(X, \bZ(1)) \otimes  \mathbb{H}^1(X, &F^1\Omega^{ <1}_{\Xepsm}[-1]) \xrightarrow{\cup} \mathbb{H}^2(X, F^1\Omega^{ <2}_{\Xepsm}[-1]) \\
 (f, a) = (\frac{1}{2\pi \sqrt{-1}}(\log_\beta f - \log_\alpha f, \log_\alpha f), \varepsilon a ) &\mapsto f\cup a  \\ & = \varepsilon \frac{1}{2\pi \sqrt{-1}}\big( a(\log_\beta f - \log_\alpha f ),(\log_\alpha f) da +  (a \log_\alpha f)\frac{d\varepsilon}{\varepsilon}\big) 
\end{aligned}
\end{equation}
since $d(\varepsilon a ) = \varepsilon d a + \varepsilon a \frac{d\varepsilon}{\varepsilon}$, where the right hand side of the equality is a \v{C}ech cocycle representing an element  of $\mathbb{H}^1(X, F^1\Omega^{ <2}_{\Xepsm})$.

For $a, f \in \cO_X^\times$, we can consider the element $\{a,f\} \in T\CH^2(X, 2; 2 )$ arising from the cycle $(a,f) \subset X\times \mathbb{G}_m\times \square^1$. Then, the  computation of \eqref{eq.ExplicitcupProduct} shows precisely that $\{a,f\}$ is mapped under $\phi_X^{2,2;2}$ to the cohomology class represented by the cocycle
\begin{equation}\label{eq:BE-refined}\phi_X^{2,2;2} (\{a, f\}) =  \varepsilon \frac{1}{2\pi \sqrt{-1}}\big( a(\log_\beta f - \log_\alpha f ),(\log_\alpha f) da +  (a \log_\alpha f)\frac{d\varepsilon}{\varepsilon}\big).  \end{equation}
This is a refinement of the Bloch-Esnault formula \cite[(5.8)]{BE2} with the extra infinitesimal term $(a \log_\alpha f) \dlog \varepsilon$. 
% Remark : Depending on the choice of the isomorphism \exp\colon \H_\mathcal{D}^1(X, \bZ(1)) \xrightarrow{\cong} \H^0(X,\cO_X^\times), there might be a factor \frac{1}{2\pi i} in front of the cocycle. I choose not to write it.
\def\Pp1{\mathbb{P}^1\setminus\{1\}}
\subsubsection{} We specialize now to the case where $a\in \Gamma(X, \cO^\times_X)$ such that $1-a$ is also invertible. Consider the parametrized cycle
\[Li_2^{\rm add}(a) = (t, \frac{1}{t}, 1- \frac{t^2}{t-1}(a(1-a))) \subset X\times\mathbb{A}^1\times (\Pp1)^2\]
obtained by intersecting with $X\times \mathbb{A}^1\times (\Pp1)^2$ the codimension $2$-cycle in $X\times \mathbb{A}^1\times (\mathbb{P}^1)^2$ given by the graph of the rational function
\[X[t] \to (\mathbb{P}^1)^2, \quad (\frac{1}{t}, 1- \frac{t^2}{t-1}(a(1-a))).\]
The intersection of $Li_2^{\rm add}(a) $  with the faces of $\square^2 = (\Pp1)^2 $ given by $y_1 = 0,\infty$ and $y_2=\infty$ is empty. Intersecting with $y_2 = 0$ gives the boundary component $(\frac{1}{a}, a) + (\frac{1}{1-a}, 1-a)$. %, that is a $0$-cycle in $\mathbb{G}_m\times \square$.
Moreover, $Li_2^{\rm add}(a) $ satisfies the modulus condition with respect to the divisor $2[0]$ of $\mathbb{A}^1$ and therefore the above computation shows that the Cathelineau element
\begin{equation}\label{eq.Cath.elem.cycle} \big(\frac{1}{a}, a\big) + \big(\frac{1}{1-a}, 1-a\big)\qfor a \in \Gamma(X, \cO_X^\times), a\neq 1\end{equation}  
is $0$ in $\CH^{2}(X[t]| D_2, 1) = T\CH^2(X, 2; 2)$. The cycle $Li_2^{\rm add}(a)$ is thus a valid cycle-theoretic avatar of the additive dilogarithm, in the sense of \cite[5]{BE}.

Note that for $a = 0$ or $1$, the above cycle is still in good position and trivially satisfies the modulus condition, but has empty boundary.
\begin{rem}One can relax the condition that $a$ is an invertible holomorphic function. The above cycle makes sense more generally if $a$ (and $1-a$) is a holomorphic function, but maybe not everywhere invertible. The subset $Z_a$ (resp. $Z_{1-a}$) of $X$ where $a$ (resp. $1-a$) is zero satisfies 
    \[(Z_a\times \mathbb{A}^1 \times (\mathbb{P}^1)^2)\cap Li_2^{\rm add}(a)  \subset X\times \mathbb{A}^1 \times \mathbb{P}^1 \times (y_2 = 1)\] (and similarly for $Z_{1-a}$) and so does not give extra boundary.
    \end{rem}
\begin{rem}  We drop for a moment the assumption of working over $\bC$. Let $k$ be an algebraically closed field of characteristic $0$ and let $X$ be a smooth $k$-scheme. In \cite[6.22]{BE2} the following map was defined
\[ \rho\colon k\tensor_\bZ \wedge^{n-1}_{i=1} k^\times \to T\CH^n(k, n; 2) \]
\[  a\tensor_\bZ ( b_1\wedge \ldots \wedge b_{n-1}) \mapsto (\frac{1}{a}, b_1,\ldots, b_n)\qfor a\neq 0, \text{and }\mapsto 0\qfor a=0  \]
and it was shown that $\rho$ induces an isomorphism $\Omega^{n-1}_{k/\bZ} \xrightarrow{\cong} {T\CH^n(k, n; 2)}$. 
We see then that for $X = \Spec (k)$, the Cathelineau element \eqref{eq.Cath.elem.cycle} is the image under $\rho$ of the element $a\tensor a + (1-a)\tensor (1-a) \in k\tensor_\bZ k^\times$.
\end{rem}
\subsubsection{} Let $X$ be again a smooth algebraic variety over $\bC$ and let $a$ be an invertible holomorphic function on $X$ such that $1-a$ is also invertible. Let $\{U_\alpha\}$ be a covering of $X$ such that the logarithms $\log_\alpha a$ and $\log_\alpha (1-a)$ are both defined.  We can consider the Cathelineau element in Deligne cohomology
\[ (a, a ) + (1-a, 1-a)\in  \H_\mathcal{D}^1(X, \bZ(1)) \times  \mathbb{H}^1(X, F^1\Omega^{ <1}_{\Xepsm}[-1]).\]
By  \eqref{eq.ExplicitcupProduct}, the cup product $a\cup a + (1-a) \cup (1-a)$ can be represented by the following \v{C}ech cocycle
\begin{equation*}
\begin{aligned}
a\cup a + (1-a) \cup (1-a) = & \varepsilon\frac{1}{2\pi \sqrt{-1}}\big( a(\log_\beta a - \log_\alpha a ) + (1-a)(\log_\beta { (1-a)} - \log_\alpha { (1-a)} ), \\ &(\log_\alpha a)da - (\log_\alpha (1-a)) da + (a\log_\alpha a + (1-a)\log_\alpha(1-a))\frac{d\varepsilon}{\varepsilon}  \big).
\end{aligned}
\end{equation*}
This is actually a boundary element, so that it represents the $0$ class in cohomology. Indeed, consider the element
\[ ( \frac{\varepsilon}{2\pi \sqrt{-1}} \int^a { \log_\alpha\Big(\frac{z}{1-z}\Big)\text{d}z})_\alpha= \frac{1}{2\pi \sqrt{-1}}(\varepsilon a \log_\alpha a + \varepsilon(1-a)\log_\alpha(1-a) )_\alpha \in \prod_\alpha\Gamma(U_\alpha, \varepsilon \cO_X), \]
where $ \int^a { \log_\alpha\Big(\frac{z}{1-z}\Big)\text{d}z}$ is Shannon's entropy function. %As explained by Bloch-Esnault \cite{BE}, it is the replacement of Bloch's dilogarithm function in the additive theory and deserves the name of \textit{additive dilogarithm}.
If $\partial$ denotes the \v{C}ech differential, we have
\begin{equation*}
\begin{aligned}
&\partial( \frac{\varepsilon}{2\pi \sqrt{-1}} ( a \log_\alpha a + (1-a)\log_\alpha(1-a) )_\alpha ) = \\
& \frac{\varepsilon}{2\pi \sqrt{-1}}  \big[ a(\log_\beta a - \log_\alpha a ) + (1-a)(\log_\beta (1-a) - \log_\alpha (1-a)), \\
 & \hskip 30pt
d a \log_\alpha a +   d a  + a\log_\alpha(a) \depss + d(1-a) \log_\alpha(1-a) - d a  + (1-a)\log_\alpha(1-a) \depss \big] % \\ 
%  &= a\cup a + (1-a) \cup (1-a).
\end{aligned}
\end{equation*}
Using the expression of the regulator as computed in \eqref{eq:BE-refined}, we find the relation
\begin{equation}\label{eq:add-dilog-refined-eq} \phi_X^{2,2;2}(\{a,a\}) + \phi_X^{2,2;2}(\{1-a,1-a\}) = \partial (\frac{\varepsilon}{2\pi \sqrt{-1}} \int^a { \log_\alpha\Big(\frac{z}{1-z}\Big)\text{d}z})_\alpha.\end{equation}

\subsection{}Let $k$ be an algebraically closed field of characteristic $0$. One can construct interesting cycles in additive higher Chow groups using the technique of Bloch in \cite{BlochMTM},  giving some evidence of Bloch-Esnault conjecture \cite[4.6]{BE2} and on which it is possible to test the higher regulators $\phi_k^{n,n;2}$. Consider, for example, the $2$-dimensional cycle 
\begin{equation}\label{eq:mod-li2} L(t, u) = (t, \frac{1}{t}, 1 + \frac{t^2}{t-1} u , u ) \subset \mathbb{A}^1\times \square^2 \times \mathbb{A}^1,
\end{equation}
obtained, up to reordering the factors, by intersecting with $\mathbb{A}^1\times \square^2 \times \mathbb{A}^1$ the graph of the rational function 
$\bA^1\times \bA^1\to (\P^1)^2$ that sends $(t,u)$ to $(\frac{1}{t}, 1 + \frac{t^2}{t-1} u)$.
Denote by $\mu\colon \bA^1\times \bA^1 \to \bA^1$ the multiplication map $(x,y)\mapsto xy$ and by $\tau\colon \square^1\to \bA^1$ the isomorphism sending the affine coordinate $y$ to $\frac{y}{y-1}$: it sends $0$ to $0$, $\infty$ to $1$ and extends to a morphism $\P^1\to \P^1$ by sending $1$ to $\infty$. We can pullback the cycle $L(t,u)$ along
\[ \tilde{\mu} = \text{id} \times \mu \colon  \bA^1 \times \square^2\times \bA^1\times \bA^1\to  \bA^1 \times \square^2\times \bA^1  \]
setting $u =xy$ in the defining equation \eqref{eq:mod-li2}. Specializing $\tilde{\mu}^*(L(t,u))$ at $x = a(1-a)$
and composing with $(\text{id}\times  \tau)^*$ gives rise to the cycle
\[W(a) = (t, \frac{1}{t}, 1  + \frac{t^2}{t-1} a(1-a)\frac{y}{y-1}, a(1-a)\frac{y}{y-1}, y) \subset  \mathbb{A}^1\times \square^4\]
that is in good position and that satisfies the modulus $2$ condition. This cycle satisfies  
\[\partial W(a) = (t, \frac{1}{t}, \frac{t-1}{t^2}, \frac{1-t}{1-t + t^2(a(1-a)}) - Li_2^{\rm add}(a)\wedge a(1-a)\quad \text{ and } \partial^2 W(a) = 0,\]
so that $\partial W(a) \in T\CH^3(k, 3; 2)$, where $\partial$ denotes
the boundary map of the cycle complex.
\section{The Abel-Jacobi map for relative Chow groups}\label{AJrelchow}
\def\JXDralg{J^r_{\Xb|D,alg}}
\def\JXDd{J^d_{\Xb|D}}
\def\AzXD{A_0(\Xb|D)}
\def\AzZm{A_0(\Cb|\fm)}
\def\AzCm{A_0(\Cb|\fm)}
\def\DeC{\H^2_{\cD}(\Xb|D, \Z(1))}
\def\DeCs{\H^2_{\cD, |\alpha|}(\Xb|D, \Z(1))}
\def\rhoXD{\rho_{\Xb|D}}
\def\LXD{\varOmega(\Xb|D)}
\def\LXDt{\widetilde{\varOmega}(\Xb|D)}

\def\LCm{\varOmega(\Cb|\fm)}

\def\nb{\underline{n}}
\def\pinb{\pi^{\nb}}
\subsection{Relative intermediate Jacobians}
\subsubsection{}\label{RelJac.Dia}We resume the setting of \ref{clDe}: let $X$ be a smooth variety over $\bC$ equipped with an open embedding $X\hookrightarrow \Xb$ into a smooth proper variety $\Xb$ such that $X$ is the complement of an effective Cartier divisor $D$, with $\Dred$ simple normal crossing. 
By definition, the relative Deligne complex \eqref{relDelignecomplex} fits into the distinguished triangle in $D^b(X_\an)$
\begin{equation}\label{eq.triangleDeligne}\Omega^{<r}_{\Xb|D}[-1]\to \bZ(r)^{\cD}_{\Xb|D} \to j_!\bZ(r)_X \overset{+}{\longrightarrow}\end{equation}
from which we get an exact sequence
% Assume moreover that $\Xb$ is projective and smooth over $\bC$. From \eqref{relDelignecomplex} we get an exact sequence
\[
0 \to \EXD {q-1} r \to \HDeXDr q \to \H^q(\Xb_{\an},j_!\ZrX),
\]
where $\EXD {q-1} r$ is defined to be the cokernel 
\[
\EXD {q} r =\Coker\big(\H^q(\Xb_{\an},j_!\ZrX)\to \H^q(\Xb_{\an},\WXD {< r}) \big).
\]
By Theorem \ref{thm-purityDeligne}, we have a commutative diagram
\[
\xymatrix{
&& \HMXDr q \ar[d]^{\regDe q r} \ar[rd]^{\regB q r}\\
0 \ar[r]&\EXD {q-1} r \ar[r] & \HDeXDr q \ar[r] & \H^q(\Xb_{\an},j_!\ZrX)\\
}\]
Thus we get the induced map
\begin{equation}\label{clmaprelDRB}
\rho^{q,r}_{\Xb|D} \colon {\HMXDr q}_{hom} \to \EXD {q-1} r
\end{equation}
where ${\HMXDr q}_{hom}$ is, by definition, the kernel of the regulator map to Betti cohomology $\regB q r$.
\subsubsection{} For  $q=2r$,  we write $\JXDr$ for $\EXD {2r-1} r$ and we call it the {\it$r$-th relative intermediate Jacobian} (for the pair $\Xb, D$). By Theorem \ref{Thm.relchow},  %, we have a natural map
%\[
%\CH^r(\Xb|D) \to \HMXDr {2r},
%\]
%where $\CH^r(\Xb|D)$ is the relative Chow group of codimension $r$ cycles on $\Xb$. 
the morphism $\rho^{2r,r}_{\Xb|D}$ of \eqref{clmaprelDRB} induces a map
\begin{equation}\label{relAJmap1}
\rho^r_{\Xb|D}  : \CH^r(\Xb|D)_{hom} \to \JXDr
\end{equation}
that we call the {\it relative Abel-Jacobi map.} 

The following lemma is shown by the same strategy of the proof of Proposition \ref{prop-projformula}.
\begin{lem}\label{HSHc-2}Let the notation be as above and assume that $D$ is a reduced normal crossing divisor on $\Xb$. Then the Hodge to de Rham spectral sequence
\[
E_1^{a,b} = \H^b(\Xb_{\an},\WXD a)\;\Rightarrow\; \bH^{a+b}(\Xb_{\an},\WXD\bullet)
\simeq \H^{a+b}(\Xb_{\an},j_!\bC_X)
\]
degenerates at the page $E_1$ and 
\[
F^r\H^{*}(\Xb_{\an},j_!\bC_X)= \H^*(\Xb_{\an},\WXD {\geq r})
%\subset H^{*}(\Xb_{\an},j_!\bC_X)
\]
is the $r$-th Hodge filtration for the Hodge structure on $\H^{*}(\Xb_{\an},j_!\bZ_X)$.
\end{lem}%\begin{proof}This follows from the same argument as in the proof of Lemma \ref{HSHc}, using again the standard exact sequences as in loc.~cit.
%\[
%0\to j_!\bC_{X} \to \bC_{\Xb} \to (i_1)_*\bC_{D^{[1]}}\to (i_2)_*\bC_{D^{[2]}} \to \cdots,
%\]
%\[
%0\to \WXD \bullet\to \WX \bullet \to (i_1)_*\Omega_{D^{[1]}}^{\bullet} \to 
%(i_2)_*\Omega_{D^{[2]}}^{\bullet} \to \cdots.
%\]
%\end{proof}
\subsubsection{}Let $\JXDredr$ be defined as $\JXDr$ with $D$ replaced by $\Dred$.
By Lemma \ref{HSHc-2}, we can write $\JXDredr$ as quotient
\begin{equation}\label{eq.Hodgedescr}
\JXDredr = \H^{2r-1}(\Xb_{\an},j_!\bC_X)/F^r+ \H^{2r-1}(\Xb_{\an},j_!\ZrX), %\\
%& \simeq \Ext_{MHS}(\bZ, H^{2r-1}(\Xb_{\an},j_!\ZrX)),
\end{equation}
where $F^r=\H^{2r-1}(\Xb_{\an},\WXDred {\geq r})$ is the $r$-th Hodge filtration on $\H^{2r-1}(\Xb_{\an},j_!\bC_X)$. By \cite[Prop. 2]{Ca}, we further have
\[\H^{2r-1}(\Xb_{\an},j_!\bC_X)/F^r+ \H^{2r-1}(\Xb_{\an},j_!\ZrX)  \simeq \Ext_{MHS}(\bZ, \H^{2r-1}(\Xb_{\an},j_!\ZrX)). \]
For $r=1$ or $\dim X$, $\JXDredr$ is an extension of the Jacobian $J^r_{\Xb}$ by a finite product of copies of $\bC^\times$. In the intermediate case, the canonical map
\[\JXDredr \to J^r_{\Xb}\]
is not surjective in general, but $\JXDredr$ is still a non compact complex Lie group, extension of a complex torus by a product of copies of $\bC^\times$ (see \cite[Lemma 6]{Ca}).
\begin{rem}When $D$ is not reduced, the relative intermediate Jacobian $\JXDr$ still has an interpretation as an extension group, but this time in the category of enriched Hodge structure $EHS$ defined by Bloch and Srinivas \cite{BS}.
	\end{rem}

\subsubsection{}We note that there is an exact sequence
\begin{equation}\label{eq.JExtension}
0 \to \UXDr \to \JXDr \rmapo{\pi} \JXDredr \to 0,
\end{equation}
where
\[
\UXDr=\Ker\big(\H^{2r-1}(\Xb_{\an},\WXD {<r}) \to \H^{2r-1}(\Xb_{\an},\WXDred {<r})\big).
\]
The only thing to check is the surjectivity of $\pi$, which is a consequence of the commutative diagram
\[
\xymatrix{
 \H^q(\Xb_{\an},\WXD \bullet) \ar[r] \ar[d]_{\simeq} & \H^q(\Xb_{\an},\WXD {< r})\ar[d]\\
 \H^q(\Xb_{\an},\WXDred \bullet) \ar[r]^{\alpha} &  \H^q(\Xb_{\an},\WXDred {< r}) %\simeq \HXD q/F^i\HXD q,
}\]
where the isomorphism comes from Lemma \ref{lemma-Dred-qiso} and $\alpha$ is surjective by Lemma \ref{HSHc-2}.
Thus we may view the map \eqref{relAJmap1} as the Abel-Jacobi map with $\bG_a$-part.
An analogous construction has been made in \cite{ESV} and \cite{BS} for Chow groups for singular varieties.

\subsection{Universality of Abel-Jacobi maps for zero-cycles with moduli}

In this section we prove a universal property of the Abel-Jacobi maps for zero-cycles with moduli
(see Theorem \ref{thm.universailty}).
It is an analogue of \cite[Th.4.1]{ESV} where a similar property is shown for Abel-Jacobi maps for zero-cycles
on singular varieties. 
We also note that it is a Hodge theoretic analogue of \cite[Th.3.29]{Ru} (cf. also \cite{KR}).

\subsubsection{}\label{SettingUniversality}Let the notation be as in \S\ref{reldelcpx} with $d=\dim(X)$. We consider the Abel-Jacobi map
\begin{equation}\label{relAJmap}
\rhoXD: \AzXD \to \JXDd,
\end{equation}
where $A_0(\Xb|D)=\CH^d(\Xb|D)_{hom}$ is the degree-$0$ part of the Chow group $\CH_0(\Xb|D)$
of zero cycles with modulus $D$.  Recall (cf. \eqref{eq.relchow})  that $\CH_0(\Xb|D)$ is the quotient of the group of zero-cycles on $X$ by
an equivalence relation which refines the rational equivalence by a modulus condition with respect to $D$. %
% \begin{equation}\label{relCH0}
% \CH_0(\Xb|D) = \Coker\big(\underset{C\in C^N_1(X)}{\bigoplus}\; G(\Cb,\gamma_C^*D) \rmapo{\delta} Z_0(X) \big),
% \end{equation}
% where $C^N_1(X)$ is the set of the normalizations of integral closed curves on $X$, and for $C\in C_1^N(X)$,
%  $\Cb$ is the smooth compactification of $C$ with the natural morphism $\gamma_C:\Cb\to \Xb$,
% and $G(\Cb,\gamma_C^*D)$ is the subgroup of $\bC(C)^\times$ defined as \eqref{eq.DefG},
% and $\delta$ is induced by the divisor map on $C$ and the pushforward map of zero cycles via  $\gamma_C$.
By definition, we have
\begin{equation}\label{eq.JXDd}
\JXDd =\Coker\big(\H^{2d-1}(\Xban,j_!\ZdX)\to \H^{2d-1}(\Xban,\WXD {< d}) \big).
\end{equation}
%where $\WXD \bullet$ for the relative (analytic) de Rham complex with
%\[\WXD i = \Omega^i_{\Xb}(\log D)\otimes_{\cO_{\Xb}} \cO_{\Xb}(-D).\]
We endow $\JXDd$ with the structure of a complex Lie group as a quotient of the finite-dimensional complex vector space
$\H^{2d-1}(\Xban,\WXD {< d})$ by a discrete subgroup. 
By \eqref{eq.JExtension}, we have an exact sequence
\[
0 \to \UXDd \to \JXDd \rmapo{\pi} \JXDredd \to 0,
\]
where $\UXDd$ is a finite-dimensional complex vector space. By \eqref{eq.Hodgedescr} we see that $\JXDredd $
%\[
%\JXDredd = H^{2d-1}(\Xban,j_!\bC_X)/F^d+ H^{2d-1}(\Xban,j_!\ZdX),
%\]
%which
is a semi-abelian variety, due to the fact that the non-zero Hodge numbers of 
\[\H^{2d-1}(\Xban,j_!\bZ(d)_X)\]are among $\{(-1,0),(0,-1),(-1,-1)\}$ (cf. \cite[3]{ESV}).

\begin{lem}\label{lem1.universailty}
$\JXDd$ has a unique structure as a commutative algebraic group for which $\pi$ is a morphism of algebraic groups.
\end{lem}
\begin{proof}
This follows from the fact noted in \cite[(10.1.3.3)]{De} that the isomorphism classes of analytic and algebraic group
extensions of an abelian variety by $\bG_a$ or $\bG_m$ coincide (see \cite[Lem.3.1]{ESV}).
\end{proof}

\begin{defi}\label{def.regular}
Take a point $o\in X$ and define a map of sets
\[
\iota_o: X \to A_0(\Xb|D)\;;\; x \to \text{the class of } [x] - [o].
\]
For a commutative algebraic group $G$, a homomorphism of abelian groups
\[\rho: \AzXD \to G\]
is called \emph{regular} if $\rho\circ\iota_o: X\to G$ is a morphism of algebraic varieties.
\end{defi}

The following theorem implies that $\JXDd$ is the \emph{universal regular quotient} of $\AzXD$.

\begin{theo}\label{thm.universailty}Let the notation be as in \ref{SettingUniversality}.
\begin{enumerate}
\item
The map $\rho_{\Xb|D}: \AzXD \to \JXDd$ is surjective and regular.
\item
For a regular map $\rho: \AzXD \to G$, there is a unique morphism 
$h_\rho: \JXDd \to G $ of algebraic groups such that $\rho=h_\rho\circ\rhoXD$.
\end{enumerate}
\end{theo}
\begin{rem}\label{rem.regular}
\begin{itemize}
\item[(1)]
It is easy to see that the universality does not depend on the choice of the base point $o\in X$.
\item[(2)]
By the same argument as \cite[Lem.1.12]{ESV}, one can shows that the image of $\rho\circ\iota_o$
is contained in the connected component of $G$. 
\item[(3)]%\label{rem.thm.universality}
Suppose that $\dim(X)=1$. Then, by Lemma \ref{lem2.universality} below, $\JXDd$ is the generalized Jacobian of $\Xb$ with modulus $D$.
%Let $\varphi_{\Xb|D}:X \to \JXDd$ be the composite of $\rhoXD$ and $\iota_0$.
% and and $\varphi_{\Xb|D}$ is the canonical morphism (cf. \cite[Ch.V \S8]{Se}). 
Thus $\rho_{\Xb|D}$ is an isomorphism and Theorem \ref{thm.universailty} in this case 
follows from \cite[Ch.V Th.1]{Se}.
\end{itemize}
\end{rem}
\subsection{The proof of the universality theorem}
%\subsection{Analytic maps to complex Lie groups}
\subsubsection{}We recall some basic facts on structure of a complex Lie groups.
Let $G$ be a connected commutative complex Lie group and $\varOmega(G)$ be 
the space of the invariant holomorphic $1$-forms on $G$.
We have a natural isomorphism
\[
\tau_G: G \isom \varOmega(G)^\vee/\H_1(G,\bZ)\;;\; g \to \biggl\{\omega \to \int_{e}^g \omega \biggl\}\quad 
(g\in G,\;\omega\in \varOmega(G)),
\]
where $\H_1(G,\bZ)\to \varOmega(G)^\vee$ is given by integration of $1$-forms over topological cycles,
$e\in G$ is the unit and the integration is over a chosen path from $e$ to $x$.
Note that $\varOmega(G)^\vee$ is the identified with the space $Lie(G)$ of the invariant vector fields on $G$ and 
$\tau_G^{-1}$ is given by the exponential map $Lie(G) \to G$. 

For a given morphism $f: M \to G$ of complex manifolds and a point $o\in M$ with $e=f(o)$, we have a formula
\begin{equation}\label{eq.fMG}
\tau_G(f(x)) =  \biggl\{\omega \to \int_{o}^x f^*\omega\biggl\} \quad (x\in M,\;\omega\in \varOmega(G)),
\end{equation}
where $f^*: \varOmega(G) \to \H^0(M,\Omega^1_M)$ is the pullback along $f$.
\medbreak

%\subsection{Proof of the universality theorem}

\begin{lem}\label{lem2.universality}
Put
\[
\LXD := \biggl\{\omega\in \H^0(\Xban,\Omega^1_{\Xb}(D))\;|\; 
d\omega=0\in \H^0(\Xan,\Omega^2_X)\biggl\}.
%H^0(\Xb,\Omega^1_{\Xb}(D))_{cl}=
\]
\begin{itemize}
\item[(1)]
There is a canonical isomorphism of complex Lie groups
\[
\tau_{\Xb|D}: \JXDd \simeq \LXD^\vee/\Image(\H_1(\Xan,\bZ)),
\]
where $\H_1(X_{\an},\bZ)\to \LXD^\vee$ is given by integration of $1$-forms over topological cycles.
\item[(2)]
Let $\varphi_{\Xb|D}:X \to \JXDd$ be the composite of $\rhoXD$ and $\iota_0$.
Then 
\[ \tau_{\Xb|D}(\varphi_{\Xb|D}(x)) = \biggl\{\omega \to \int_{o}^x \omega \biggl\}\;\in \LXD^\vee \quad(x\in X).\]
\item[(3)]
The pullback of holomorphic $1$-forms by $\varphi_{\Xb|D}$ induces an isomorphism
\[\varphi_{\Xb|D}^*: \varOmega(\JXDd) \isom \LXD\subset \H^0(\Xan,\Omega^1_X).\] 
\end{itemize}
\end{lem}
\begin{proof}
Recall the definition of the Jacobian \eqref{eq.JXDd}. By the Poincar\'e duality we have a canonical isomorphism
\begin{equation}\label{eq.PD}
\H^{2d-1}(\Xban,j_!\ZdX) \cong \H_1(\Xan,\bZ).
\end{equation}
From a standard spectral sequence argument, 
%\[
%E_1^{p,q} = \left.\left\{\begin{gathered}
% \H^q(\Xban,\WXD p)
% \;\quad \text{if $p<d$} \\ 
% 0 \;\quad \text{otherwise}\\
%\end{gathered}\right.\quad
%\right.
%\;\Rightarrow \H^{p+q}(\Xban,\WXD {< d})
%\]
we get an isomorphism
\[
\H^{2d-1}(\Xban,\WXD {< d}) \isom\Coker\big(\H^d(\Xban,\WXD {d-2}) \rmapo d  \H^d(\Xban,\WXD {d-1})\big).
\]
Hence, by  Serre duality, we have a natural isomorphism 
\[
\H^{2d-1}(\Xban,\WXD {< d})^\vee \cong \LXDt,
\]
where
\[
\LXDt=
\biggl\{\omega\in \H^0(\Xban,\Omega^1_{\Xb}(\log D)\otimes_{\cO_{\Xb}} \cO_{\Xb}(D'))\;|\; 
d\omega=0\in \H^0(\Xan,\Omega^2_X)\biggl\}
\]
and $D'$ denotes the divisor $D-\Dred$. Lemma \ref{lem2.universality}(1) follows then from \eqref{eq.PD} and the following claim.
\begin{equation}\label{claim.LXD}
\LXDt =\LXD.
\end{equation}
To show \eqref{claim.LXD}, we may work locally at a point $x\in D$. 
Choose a system of regular parameters $(\pi_1,\dots,\pi_r,t_1\dots,t_s)$ in $\cO_{X,x}$ such that
$\pi_i$ are the local equations of the irreducible components $D_i$ of $D$ passing through $x$.
Let $n_i$ be the multiplicity of $D_i$ in $D$. Then a local section
$\omega$ of $\Omega^1_{\Xb}(D)$ at $x$ is written as
\[
\omega=\frac{\xi}{\pi_1^{n_1}\cdots \pi_r^{n_r}}
\qwith\; \xi=\underset{1\leq i\leq r}{\sum} a_i d\pi_i +\underset{1\leq j\leq s}{\sum} b_j d t_j
\quad(a_i,b_j\in \cO_{X,x}).
\] 
Put $\pi=\pi_1\cdots\pi_r$. If $\omega\in \LXD$, we have 
\[
0=d\omega = \frac{1}{\pi_1^{n_1}\cdots \pi_r^{n_r}}
\biggl[-\underset{1\leq l\leq r}{\sum} n_l \frac{d \pi_l}{\pi_l}\; \wedge \xi +  d\xi\biggl],
\]
which implies
\begin{equation}\label{eq.eta}
\eta:= \underset{1\leq l\leq r}{\sum} n_l \frac{d \pi_l}{\pi_l}\; \wedge \xi \;\in \Omega_{\Xb,x}.
\end{equation}
We compute
\[
\eta= \underset{1\leq i<l\leq r}{\sum} (n_l \pi_i a_i - n_i\pi_l a_l) 
\frac{d \pi_l}{\pi_l}\wedge \frac{d \pi_i}{\pi_i}\;+\;
\underset{1\leq l,j \leq r}{\sum} n_lb_j \frac{d \pi_l}{\pi_l}\wedge dt_j
\]
Thus \eqref{eq.eta} implies $b_j$ are divisible by $\pi_l$ for all 
$j,l$ and $n_l \pi_i a_i - n_i\pi_l a_l$ are divisible by $\pi_l\pi_i$ for all $i,l$. 
This implies that $b_j$ and $\pi_ia_i$ are divisible by $\pi$ for all $i,j$. Hence
\[
\omega=\frac{1}{\pi_1^{n_1-1}\cdots \pi_r^{n_r-1}}\big(
\underset{1\leq i\leq r}{\sum} a'_i \frac{d\pi_i}{\pi_i} +\underset{1\leq j\leq s}{\sum} b'_j d t_j\big)
\qwith b'_j=b_j/\pi,\;a'_i=\pi_ia_i/\pi \in \cO_{X,x}
\]
so that $\omega$ is a local section of $\Omega^1_{\Xb}(\log D)(D')$ at $x$. This proves \eqref{claim.LXD} and 
the proof of Lemma \ref{lem2.universality}(1) is complete. 
\medbreak
\def\bZXDonede{\bZ(1)_{\Xb|D}^{\cD}}
\def\tclDeb#1#2{\widetilde{cl}_{\cD}^{#1}(#2)}
\def\clDemap{cl_{\cD}}

We now prove Lemma \ref{lem2.universality}(2). Suppose first that $\dim X = 1$. %, i.e.  that $\Xb$ is a smooth complex connected projective curve.
In this case we have
%\begin{equation}\label{eq.ZDone}
%\bZXDonede = j_!\bZ(1)_X \to  \cO_{\Xb}(-D),
%\end{equation}
%where $j\colon X=\Xb \setminus D\hookrightarrow \Xb$ is the open immersion. 
%The long exact sequence arising from \eqref{eq.ZDone} gives 
an exact sequence
\[
0 \to J^1_{\Xb|D} \to \DeC \to \bZ \to 0.
\]
%where $J^1_{\Xb|D}= H^1(\Xban,\cO_{\Xb}(-D))/H^1(\Xban,j_!\bZ(1)_X)$ and we used the trace isomorphism
%\begin{equation}\label{eq.traceiso}
%H^2(\Xban,j_!\bZ(1)_X)\cong \bZ.
%\end{equation}
Let $ Z_0(X)$ be the group of $0$-cycles on $X$.
According to Theorem \ref{thm-purityDeligne}, % (where one takes $F=\emptyset$ and $Y=\Xb$), 
we have defined the fundamental class 
\begin{equation}\label{eq.clDe1alpha}
\clDe 1 \alpha \in \bH^{2}_{|\alpha|}(\Xban,\bZXDonede )\qfor \alpha\in Z_0(X)
\end{equation}
as the unique element which maps to the pair $(\clB 1{\alpha}, \clDR 1{\alpha})$ in
\[\H^2_{|\alpha|}(\Xban, j_! \Z(1))\oplus \bH^2_{|\alpha|}(\Xban, \Omega^{\geq 1}_{\Xb|D}) = \H^2_{|\alpha|}(\Xban, j_! \Z(1))\oplus \H^1_{|\alpha|}(\Xban, \Omega^1_{\Xb|D}).\]
This gives us a homomorphism
\[
\clDemap : Z_0(X) \to \DeC=\bH^{2r}(\Xban,\bZXDonede ).
\]
%which maps $\alpha$ to the image of $\clDe 1 \alpha$ under $\DeCs\to \DeC$, the forget support map.
By the definition \eqref {relAJmap1} we have a commutative diagram
\begin{equation}\label{eq.CD1}
\xymatrix{
Z_0(X)_{\deg 0} \ar[r]^{\clDemap} \ar[d]  & \DeC\\
\CH^1(\Xb|D)_{hom} \ar[r]^{\rho^1_{\Xb|D}} & J^1_{\Xb|D} \ar[u]}
\end{equation}
To compute $\clDemap$ we use an isomorphism
\begin{equation}\label{eq.expisom}
\exp: \bZXDonede \cong  \cO_{\Xb|D}^\times[-1]\;\;\text{ in }\;  D^b(\Xban),
\end{equation}
which is induced by the exponential sequence
\begin{equation}\label{eq.expseq}
0 \to j_!\bZ(1)_X \to  \cO_{\Xb|D} \rmapo{\exp}  \cO_{\Xb|D}^\times\to 0.
\end{equation}
% in view of \eqref{eq.ZDone}.
 The composite map
\begin{equation*}\label{eq.expclDe}
Z_0(X) \rmapo{\clDemap}  \bH^{2r}(\Xban,\bZXDonede ) \rmapo {\exp} H^1(\Xban,\cO_{\Xb|D}^\times)
\end{equation*}
is computed as follows: Let $\cK_{\Xb|D}^\times $ be the subsheaf of the constant sheaf of rational functions on $\Xb$ that are congruent to $1$ modulo $D$. We have an isomorphism
\[
\div_X: \H^0(\Xban,\cK_{\Xb|D}^\times/\cO_{\Xb|D}^\times) \isom Z_0(X),
\]
given by taking the divisors of rational functions on $X$. This gives us a map
\begin{equation}\label{eq.mapL}
\cL: Z_0(X)\rmapo{\div_X^{-1}}   \H^0(\Xban,\cK_{\Xb|D}^\times/\cO_{\Xb|D}^\times) \rmapo{\partial}
\H^1(\Xban,\cO_{\Xb|D}^\times), 
\end{equation} 
where $\partial$ is the boundary map arising from the exact sequence
\begin{equation}\label{eqCurv1}
1\to  \cO_{\Xb|D}^\times \to \cK_{\Xb|D}^\times \to \cK_{\Xb|D}^\times/\cO_{\Xb|D}^\times \to 1.
\end{equation}

\begin{lem}\label{lem.expclDe}
We have $\exp\circ \clDemap = \cL$.
\end{lem}

\begin{lem}\label{lem.expint}
Consider the composite map
\begin{equation*}\label{eq.lem.expint}
\epsilon: H^0(\Xb,\WXD 1)^\vee\cong H^1(\Xban,\cO_{\Xb}(-D)) \rmapo{\exp} H^1(\Xban,\cO_{\Xb|D} ^\times).
\end{equation*}
where the first isomorphism is due to the Serre duality and the second map is induced by \eqref{eq.expseq}.
Take points $x,o\in X$ and consider 
\[
\gamma_{[o,x]}=\biggl\{\omega \to \int_{o}^x \omega \biggl\}\;\in \LXD^\vee/\H_1(\Xan,\bZ).
\]
Then we have $\epsilon(\gamma_{[o,x]})= \cL([x]-[o])$ with $[x]-[o]\in Z_0(X)$.
\end{lem} 

Note that in case $\dim(X)=1$, we have $\LXD=H^0(\Xb,\WXD 1)$ and also a commutative diagram
\[\xymatrix{
H^1(\Xban,\cO_{\Xb}(-D)) \ar[r]^\cong & \H^0(\Xb,\WXD 1)^\vee\\ 
\H^1(\Xban,j_!\bZ(1)_X) \ar[r]^\cong \ar[u] & \H_1(\Xan,\bZ)\ar[u]\\
}\]
where the lower isomorphism is due to Poincar\'e duality and the right vertical map is given by
integration on topological $1$-cycles. Hence Lemma \ref{lem2.universality}(2) in case $\dim(X)=1$ follows from \eqref{eq.CD1} and Lemmas \ref{lem.expclDe} and \ref{lem.expint}.
 
\medbreak\noindent
{\it Proof of Lemma \ref{lem.expint}:}  
Let $\gamma$ be a path in $X$ from $o$ to $x$. Let $V = \Xb\setminus \gamma$ be the complement of $\gamma$ in $\Xb$. Let $t_x$ (resp. $t_o$ ) be a holomorphic function having a simple zero on $x$ (resp. on $o$) defined on a small neighborhood of $x$ (resp. $o$). Let $U$ be an open neighborhood of $\gamma$, disjoint from $D$, such that the function $g= \frac{1}{2\pi i }\log(\frac{t_x}{t_o})$ is single valued on $V\cap U$. Then the cocycle $\{V\cap U, \frac{t_x}{t_o}\}\in \H^1(\Xban,\cO_{\Xb|D} ^\times)$ represents the element $\cL([x]-[o])$. By multiplying by a $\mathcal{C}^{\infty}$-function $f$, that we can choose identically $0$ on $D$ and identically $1$ on neighborhood of $\gamma$ containing $U$,  we can consider a $\overline{\partial}$-closed form
\[\alpha = \frac{1}{2\pi i} \overline{\partial}f\log{\Big(\frac{t_x}{t_o}\Big)} \]
of type $(0,1)$, representing a lifting of the class of $g$ in $\H^1(\Xban,\cO_{\Xb}(-D))/\H^1(\Xban,j_!\bZ(1)_X)$. % (we are using the fact that one can compute the cohomology group $\H^1(\Xb, \cO_{\Xb}(-D))$ by means of Dolbeaut cohomology).
This gives rise to an element in $\LXD^\vee/\H_1(\Xan,\bZ)$
\[\biggl\{ \omega\mapsto \int_{\Xb}\alpha\wedge \omega \biggl\}\]
and it suffices to show that 
\[\int_{\Xb}\alpha\wedge \omega = \int_{o}^x \omega \]
for every form $\omega \in \LXD$. The proof of this fact is standard and we omit it. %For $\epsilon >0$, let $\Gamma_{\epsilon}$ denote a tubular neighborhood of $\gamma$ of radius $\epsilon$. Since $\omega$ is $d$-closed, we have
%\[\alpha \wedge \omega = \frac{1}{2\pi i} \overline{\partial} f \log{\Big(\frac{t_x}{t_o}\Big)}\wedge \omega = \frac{1}{2 \pi i} d (f\log{\Big(\frac{t_x}{t_o}\Big)} %\omega) \]
%on $\Xb\setminus \Gamma_\epsilon$. By Stokes theorem we get then
%\[\int_{\Xb}\alpha\wedge \omega = \lim_{\epsilon \to 0} \int_{\Gamma_{\epsilon}} \alpha \wedge \omega = \lim_{\epsilon \to 0 } \int_{\partial \Gamma_{\epsilon}} { \frac{1}{2\pi i} \log\Big( \frac{t_x}{t_o}\Big) \omega}\]
%where we used the fact that $f$ is $1$ on $\Gamma_{\epsilon}$ (for $\epsilon$ sufficiently small). By cutting down the boundary $\partial \Gamma_{\epsilon}$ into pieces, one finally sees that
%\[\lim_{\epsilon \to 0 } \int_{\partial \Gamma_{\epsilon}} { \frac{1}{2\pi i} \log\Big( \frac{t_x}{t_o}\Big) \omega} = \int_{o}^{x} \omega\]
%completing the proof of the Lemma.
 
\medbreak\noindent
{\it Proof of Lemma \ref{lem.expclDe}:}  
Take $\alpha\in Z_0(X)$. Note
\[
\div_X^{-1}(\alpha)\in \H^0_{|\alpha|}(\Xban,\cK_{\Xb|D}^\times/\cO_{\Xb|D}^\times)
\]
since the restriction of $\div_X^{-1}(\alpha)$ to $\H^0(\Xb-|\alpha|,\cK_{\Xb|D}^\times/\cO_{\Xb|D}^\times)$ vanishes.
We have a commutative diagram
\[\xymatrix{
\H^0_{|\alpha|}(\Xban,\cK_{\Xb|D}^\times/\cO_{\Xb|D}^\times) \ar[r]^{\partial} \ar[d]^{\iota} & 
\H^1_{|\alpha|}(\Xban,\cO_{\Xb|D}^\times)\ar[d]^{\iota} 
& \ar[l]_{\exp}\bH^{2}_{|\alpha|}(\Xban,\bZXDonede )\ar[d]^{\iota}\\
\H^0(\Xban,\cK_{\Xb|D}^\times/\cO_{\Xb|D}^\times) \ar[r]^{\partial}  & 
\H^1(\Xban,\cO_{\Xb|D}^\times)& \ar[l]_{\exp} \bH^{2}(\Xban,\bZXDonede ).\\
}\]
Thus it suffices to show (cf. \eqref{eq.clDe1alpha})
\begin{equation*}\label{eq.claim1}
\partial(\div_X^{-1}(\alpha)) = \exp(\clDe 1 \alpha) \in \H^1_{|\alpha|}(\Xban,\cO_{\Xb|D}^\times).
\end{equation*}
For this we first note that the composite map
\[
\bH^{2}_{|\alpha|}(\Xban,\bZXDonede ) \rmapo{\exp} \H^1_{|\alpha|}(\Xban,\cO_{\Xb|D}^\times) 
\rmapo{d\log} H^1_{|\alpha|}(\Xban,\WXD 1) 
\]
coincides with the map induced by the map $\bZXDonede \to \WXD {\geq 1}= \WXD 1[-1]$
from lemma \ref{relDecpx.lem1}. Hence $d\log(\exp(\clDe 1 \alpha)))=\clDR 1\alpha\in \H^1_{|\alpha|}(\Xban,\cO_{\Xb|D}^\times)$. Hence the statement is a consequence of the following.

\begin{claim}\label{claim.lem.expclDe}We have
\begin{itemize}
\item[(1)]
$d\log: H^1_{|\alpha|}(\Xban,\cO_{\Xb|D}^\times) \to \H^1_{|\alpha|}(\Xban,\WXD 1)$ is injective.
\item[(2)]
$d\log(\partial(\div_X^{-1}(\alpha))) = \clDR 1\alpha\in \H^1_{|\alpha|}(\Xban,\cO_{\Xb|D}^\times)$.
\end{itemize}
\end{claim}

To show (1), we consider a commutative diagram 
\[\xymatrix{
0\ar[r]&\H^1_{|\alpha|}(\Xban,\cO_{\Xb}(-D)) \ar[r]^{\exp} \ar[d]^{=} & \H^1_{|\alpha|}(\Xban,\cO_{\Xb|D}^\times)
\ar[r] \ar[d]^{d\log} & \H^2_{|\alpha|}(\Xban,j_!\bZ(1)_X)\ar[d]\\
0\ar[r]&\H^1_{|\alpha|}(\Xban,\cO_{\Xb}(-D)) \ar[r]^{d}  & \H^1_{|\alpha|}(\Xban,\WXD 1) 
\ar[r] & \H^2_{|\alpha|}(\Xban,j_!\bC_X)\\
}\]
where the horizontal sequences are exact arising from \eqref{eq.expseq} and the exact sequence
\[
0\to j_!\bC_X \to \cO_{\Xb}(-D) \rmapo{d} \WXD 1\to 0.
\]
The injectivity of the first map in the upper (resp. lower) sequence follows from the vanishing of 
$\H^1_{|\alpha|}(\Xban,j_!\bZ(1)_X)$ (resp. $\H^1_{|\alpha|}(\Xban,j_!\bC_X)$) by semi-purity.
Thus (1) follows from the injectivity of the right vertical map due to the trace isomorphism.% \eqref{eq.traceiso}.
%\[\H^2_{|\alpha|}(\Xban,j_!\bZ(1)_X)\isom \bZ,\;\; \H^2_{|\alpha|}(\Xban,j_!\bC_X)\isom \bC.\]
%\medbreak

To show (2) take a sufficiently small open $U\subset X=\Xb-|D|$ such that $|\alpha|\subset U$ and that there is 
$f\in \Gamma(U-|\alpha|,\cO_U^\times)$ such that $\alpha=\div_U(f)$. We have a commutative diagram
\[\xymatrix{
\H^0_{|\alpha|}(\Xban,\cK_{\Xb|D}^\times/\cO_{\Xb|D}^\times) \ar[r]^{\partial} \ar[d]^{\cong} & 
\H^1_{|\alpha|}(\Xban,\cO_{\Xb|D}^\times)\ar[r]^{d\log} \ar[d]^{\cong}& \H^1_{|\alpha|}(\Xban,\WXD 1)\ar[d]^{\cong}\\
\H^0_{|\alpha|}(U,\cK_{U}^\times/\cO_U^\times) \ar[r]^{\partial} & 
\H^1_{|\alpha|}(U,\cO_{U}^\times)\ar[r]^{d\log} & \H^1_{|\alpha|}(U,\Omega^1_U)\\
& \ar[lu]_{\psi}  \H^0(U-|\alpha|,\cO_U^\times)
\ar[r]^{d\log}\ar[u]_{\delta} & \H^0(U-|\alpha|,\Omega^1_U)\ar[u]_{\delta}\\
}\]
where the vertical maps in the upper column are isomorphisms by excision, $\delta$ is a boundary map
in the localization sequence associated to $\tau: U-|\alpha|\hookrightarrow U$,
and $\delta$ is induced by the inclusion $\tau_*\cO_{U-|\alpha|}^\times \to \cK_U^\times$.
By definition we have $\psi(f)=\div_X^{-1}(\alpha)$. Hence Claim \ref{claim.lem.expclDe}(2) follows from the fact
$\delta(d\log f)=\clDR 1\alpha$, which follows from \eqref{affineDescrClass}. 
This completes the proof of Claim \ref{claim.lem.expclDe} and Lemma \ref{lem2.universality}(2) for $\dim(X)=1$.
\medbreak

For $\dim \Xb>1$, the assertion follows from the covariant functoriality of the cycle class maps \eqref{clrelDR.eq1} and
\eqref{cyclemapDeligne.eq1} for proper morphisms of pairs.
Finally Lemma \ref{lem2.universality}(3) follows from (2). 
\end{proof}
\medbreak

\begin{lem}\label{lem3.universality}
Let $\rho: \AzXD \to G$ be a regular map with $G$ connected and $\psi_\rho: X\to G$ be the composite of $\rho$ and $\iota_o$ 
(see Definition \ref{def.regular}). Then $\varOmega(G) \to \H^0(\Xan,\Omega^1_X)$, the pullback map on holomorphic $1$-forms, induces 
\[
\psi_\rho^*: \varOmega(G) \to \LXD.
\]
\end{lem}

The proof of this Lemma will be given later.
In view of \eqref{eq.fMG}, Lemma \ref{lem3.universality} implies the following corollary.

\begin{coro}\label{lem3cor.universality}
Under the notation of Lemma \ref{lem3.universality}, we have $\rho=h_\rho\circ\rhoXD$, 
where $h_\rho: \JXDd \to G$ is the morphism of algebraic groups defined by the commutative diagram
\[
\xymatrix{
\JXDd \ar[r]^{\hskip -60pt\simeq} \ar[d]_{h_\rho} & \LXD^\vee/\Image(\H_1(\Xan,\bZ))\ar[d]^{\lambda_\rho} \\
G \ar[r]^{\hskip -60pt\simeq} & \varOmega(G)^\vee/\H_1(\Gan,\bZ)\\
}\]
where $\lambda_\rho$ is induced by $\psi_\rho^*$ in Lemma \ref{lem3.universality} and 
${\psi_\rho}_*: \H_1(\Xan,\bZ)\to \H_1(\Gan,\bZ)$. 
\end{coro}

\bigskip

We need some preliminaries for the proof of Lemma \ref{lem3.universality}.

\begin{lem}\label{lem3-1.universality}
Let $\rho: \AzXD \to G$ be a regular map with $G$ connected. 
%Let $C\in C_1(X)$ and $\gamma_C:\Cb^N\to \Xb$ be the induced fintie map, where $\Cb^N$ is the normalizationa
Let $\Cb$ be a smooth projective curve and $\gamma: \Cb \to \Xb$ be a morphism such that $C=\gamma^{-1}(X)$ is not empty.
Take $o\in C$ and write $o$ also for its image in $X$. Consider the composite map
\[
\psi: C \rmapo{\gamma} X \rmapo {\iota_o} \AzXD \rmapo \rho G.
\]
Then the image of $\psi^*:\varOmega(G) \to \H^0(C,\Omega^1_C)$ is contained in $\H^0(\Cb,\Omega^1_{\Cb}(\gamma^*D))$.
\end{lem}
\begin{proof}
Put $\fm=\gamma^*D$ and let $\rho_{\Cb|\fm}: \AzCm\to J^1_{\Cb|\fm}$ be the Abel-Jacobi map for the curve $\Cb$ with modulus $\fm$. 
We have a commutative diagram
\[
\xymatrix{
C \ar[r]^{\hskip -20pt\iota_o} \ar[d]^{\gamma} & \AzCm \ar[d]^{\gamma_*} \\
X \ar[r]^{\hskip -20pt\iota_o}  & \AzXD \ar[r]^{\hskip 10pt\rho} & G \\
}\]
where the right vertical map is induced by $\gamma_*:Z_0(C) \to Z_0(X)$. % (see \eqref{relCH0}).
By the assumption that $\rho$ is regular, $\rho\circ\gamma_*:\AzCm \to G$ is also regular.
By Remark \ref{rem.regular}(3) we know that the generalized Jacobian $J^1_{\Cb|\fm}$ is universal, so that there exists a morphism $h: J^1_{\Cb|\fm} \to G$ of algebraic groups such that
$\rho\circ\gamma_*= h\circ \rho_{\Cb|\fm}$.
%Recalling 
%\[\AzCm=\Coker\big(G(\Cb,\fm) \rmapo{\div_C} Z_0(C)^{\deg 0}\big),\]
%this implies that $\psi$ has modulus $\fm$ in the sense of \cite[Ch.III Def.1]{Se}.
%Noting that $J^1_{\Cb|\fm}$ is the generalized Jacobian of $\Cb$ with modulus $\fm$ (cf. \cite[Ch.V Prop.11]{Se}), 
Hence $\psi$ factors as
\[
\psi: C \rmapo {\varphi} J^1_{\Cb|\fm} \rmapo h G,
\]
where $\varphi$ is the composite $C \rmapo {\iota_o} \AzCm \rmapo{\rho_{\Cb|\fm}} J^1_{\Cb|\fm}$.
Hence $\psi^*$ in the lemma factors as
\[
\varOmega(G) \rmapo{h^*} \varOmega(J^1_{\Cb|\fm}) \rmapo {\varphi^*} \H^0(C,\Omega^1_C).
\]
Now the lemma follows from the fact (cf. \cite[Ch.V Prop.5]{Se} and Lemma \ref{lem2.universality}(3)) that
\[\varphi^*(\varOmega(J^1_{\Cb|\fm}))=\H^0(\Cb,\Omega^1_{\Cb}(\fm))\subset \H^0(C,\Omega^1_C).\]
\end{proof}
\medbreak

\begin{lem}\label{lem3-2.universality}
%Let the notation be as in  \eqref{relCH0}.
The restriction map
\[
\theta: \H^0(X,\Omega^1_X)/\H^0(\Xb,\Omega^1_{\Xb}(D)) \to \underset{C\in C^N_1(X)}{\prod}\; 
\H^0(C,\Omega^1_{C})/\H^0(\Cb,\Omega^1_{\Cb}(\gamma_C^*D))
\]
where $C^N_1(X)$ is the set of the normalizations of integral closed curves on $X$ (see \ref{eq.relchow}), is injective. 
\end{lem}
\begin{proof}
Let  $Z$ be an irreducible component of $D$. It suffices to show the map
\[
\H^0(\Xb,\Omega^1_{\Xb}(D+Z))/\H^0(\Xb,\Omega^1_{\Xb}(D)) \xrightarrow{\theta} \underset{C\in C_1(X)}{\prod}\; 
\H^0(\Cb,\Omega^1_{\Cb}(\gamma_C^*(D+Z)))/\H^0(\Cb,\Omega^1_{\Cb}(\gamma_C^*D)).
\]
is injective. Let $\omega\in  \H^0(\Xb,\Omega^1_{\Xb}(D+Z))$ such that $\theta(\omega)=0$.
We want to show $\omega\in  \H^0(\Xb,\Omega^1_{\Xb}(D))$. % Consider the Residue map
%\[0\to \]
Since $\Omega^1_{\Xb}(D+Z)/\Omega^1_{\Xb}(D)$ is a locally free $\cO_Z$-module, it suffices to show
$\omega_{|U}\in \H^0(U,\Omega^1_{\Xb}(D))$ for some open subset $U\subset \Xb$ with $U\cap Z\not=\emptyset$.
Choose an affine open $U=\Spec(A)$ satisfying the following conditions:
\begin{itemize}
\item[$(\spadesuit1)$]
$Z_U:=Z\cap U\not=\emptyset$ and $U\cap Z'=\emptyset$ for any component $Z'\not=Z$ of $D$.
\item[$(\spadesuit2)$]
There exists a regular system of parameters $\pi, t_1\dots,t_r$ in $A$, with $r=d-1$, such that $Z_U=\Spec(A/(\pi))$ and 
\[
\H^0(U,\Omega^1_{\Xb})= A d\pi \oplus \underset{1\leq i\leq r}{\bigoplus} A dt_i.
\]
\end{itemize}
Let $n$ be the multiplicity of $Z$ in $D$. We can write
\[
\omega_{|U}=\frac{1}{\pi^{n+1}}\big(a d\pi + \underset{1\leq i\leq r}{\bigoplus} b_i dt_i)
\qwith a,b_i\in A.
\]
Assume, by contradiction, that $a$ is not divisible by $\pi$ in $A$. Then we can take a closed point $x\in Z_U$ such that $a\in \cO_{X,x}^\times$.
Consider the ideal \[I=(t_1-t_1(x),\dots,t_r-t_r(x))\subset A,\] where for a  section $f\in A$, $f(x) \in \bC$ denotes the residue class at the point $x$.
By construction, there exists a unique irreducible component  $W\subset U$  of $\Spec(A/I)$ passing through $x$.
The condition $(\spadesuit2)$ above implies then that $\dim W=1$ and that $W$ is regular at $x$. Let $\Cb$ be the normalization of the closure of $W$ in $\Xb$.
Then $\overline{\pi}=\pi\mod I$ is a local parameter of $\Cb$ at $x$. By definition, the pullback of $\omega$ to $\Cb$ is written locally at $x$ as
\[
\omega_{|\Cb} = \frac{1}{\pi^{n+1}} \overline{a} d\overline{\pi}\quad (\overline{a}=a\mod I).
\]
Now recall that, by assumption, we have $\omega_{|\Cb} \in \H^0(\Cb,\Omega^1_{\Cb}(\gamma_C^*D))$. On the other hand, 
$\overline{a}\in \cO_{\Cb,x}^\times$ since $a\in \cO_{X,x}^\times$. 
This is a contradiction and  so $a$ must be divisible by $\pi$.

We repeat the same argument: if $b_1$ is not divisible $\pi$ in $A$, we  can take a closed point $x\in Z_U$ such that $b_1\in \cO_{X,x}^\times$.
Considering this time the ideal \[I^\prime= (t_1-t_1(x)-\pi,t_2-t_2(x),\dots,t_r-t_r(x))\subset A,\] we get in the same way a contradiction, proving that also $b_1$ must be divisible $\pi$. Iterating the argument for $b_i$ with $i\geq 2$ completes the proof. %, we finally have that the restriction  $\omega_{|U}$ belongs to $\H^0(U,\Omega^1_{\Xb}(D))$, completing the proof.
\end{proof}

\medbreak\noindent
{\it Proof of Lemma \ref{lem3.universality}:}\;
Since an invariant differential form on a commutative Lie group is closed, it suffices to show the image of 
$\psi_\rho^*: \varOmega(G) \to \H^0(\Xan,\Omega^1_X)$ is contained in $\H^0(\Xban,\Omega^1_{\Xb}(D))$.
The assertion follows then from  Lemma \ref{lem3-1.universality} and Lemma \ref{lem3-2.universality}.
\medbreak
We can finally proof the main Theorem of this section

\medbreak\noindent
{\it Proof of Theorem \ref{thm.universailty}:}\;
Theorem \ref{thm.universailty}(2) follows from Corollary \ref{lem3cor.universality} and Remark \ref{rem.regular}(2).
We are left to show Theorem \ref{thm.universailty}(1). Let $\varphi_{\Xb|D}:X \to \JXDd$ be as Lemma \ref{lem2.universality}(2).
By loc.cit, it is analytic. One can then show that it is a morphism of algebraic varieties by the same argument as in 
the proof of \cite[Th.4.1(i)]{ESV}.% {\color{blue} (Shall we leave this like this, or would you try to write more details?)} 

It remains to show the surjectivity of $\rho_{\Xb|D}$.
Let $C\in C_1^N(X)$ and put $\fm=\gamma_C^*D$. By Lemma \ref{lem2.universality} we have a commutative diagram
\[ \xymatrixcolsep{5pc}
\xymatrix{
\AzCm \ar[r]^{\rho_{\Cb|\fm}}_{\simeq} \ar[d]_{{\gamma_C}_*} & 
J^1_{\Cb|\fm} \ar[r]^{\hskip -40pt\tau_{\Cb|\fm}}_{\hskip -40pt\simeq} & \LCm^\vee/Image(\H_1(C_{an},\bZ)) \ar[d]\\
\AzXD \ar[r]^{\rho_{\Xb|D}}  & \JXDd \ar[r]^{\hskip -40pt\tau_{\Xb|D}}_{\hskip -40pt\simeq} & 
\LXD^\vee/Image(\H_1(X_{an},\bZ)) \\
}\]
where the right vertical map is induced by the pullback $\gamma_C^*:\LXD\to \LCm$,
and $\rho_{\Cb|\fm}$ is an isomorphism by Remark \ref{rem.regular}(3).
Noting that $\H^0(\Xban,\Omega^1_{\Xb}(D))$ is of finite dimension, the argument of the proof of Lemma \ref{lem3-2.universality}
shows that there is a finite subset 
$\{C_i\}_{i\in I}\subset C_1^N(X)$ such that the pullback map
\[
\LXD \to \underset{i\in I}{\bigoplus}\; \varOmega(\Cb_i|\gamma_{C_i}^*D)
\]  
is injective. From the above diagram, this implies the composite map
\[
 \underset{i\in I}{\bigoplus}\; A_0(\Cb_i|\gamma_{C_i}^*D)\; \to \AzXD \rmapo{\rho_{\Xb|D}}  \JXDd 
\]
is surjective and hence so is $\rho_{\Xb|D}$. This completes the proof. 
%\newpage

    \bibliography{bib-relchow} 
    \bibliographystyle{siam}
\end{document}